\documentclass{amsart}
\usepackage[nohug]{diagrams}
\usepackage{amscd}
\usepackage[mathscr]{eucal}
\usepackage{amsfonts}
\usepackage{amsmath}
\usepackage{amsthm}
\usepackage{amssymb}
\usepackage{stmaryrd}
\usepackage{latexsym}
\usepackage{tabularx,array}
\usepackage[all]{xy}
\usepackage{hyperref}
\usepackage{bbm}
\usepackage{latexsym}
\usepackage{enumitem}
\newfont{\msam}{msam10}

\newtheorem{theorem}[]{Theorem}
\newtheorem{proposition}[]{Proposition}
\newtheorem{corollary}[]{Corollary}
\newtheorem{lemma}[]{Lemma}

\theoremstyle{definition}
\newtheorem{definition}[]{Definition}
\newtheorem{remark}[]{Remark}
\newtheorem{example}[]{Example}
\newtheorem{conjecture}[]{Conjecture}

\def\dotimes{\bar{\otimes}}
\let\nc\newcommand

\def\bthm{\begin{theorem}}
\def\ethm{\end{theorem}}
\def\blemma{\begin{lemma}}
\def\elemma{\end{lemma}}
\def\bproof{\begin{proof}}
\def\eproof{\end{proof}}
\def\bprop{\begin{proposition}}
\def\eprop{\end{proposition}}
\def\bcor{\begin{corollary}}
\def\ecor{\end{corollary}}
\nc{\la}{\label}
\def\Z{\mathbb{Z}}

\def\c{\mathbb{C}}
\def\M{\mathcal{M}}

\def\L {\mathbb{L}}
\def\R{\mathbb{R}}

\def\Com{\mathtt{Com}}

\def\Alg{\mathtt{Alg}}

\def\LieAlg{\mathtt{LieAlg}}
\def\DGL{\mathtt{DGLA}}
\def\DGC{\mathtt{DGC}}
\def\cDGC{\mathtt{DGCC}}
\def\DGLC{\mathtt{DGLC}}
\def\Tw{\mathtt{Tw}}
\def\Mod{\mathtt{Mod}}
\def\cMod{\mathtt{CoMod}}

\def\Bimod{\mathtt{Bimod}}
\def\cAlg{\mathtt{Comm\,Alg}}

\def\Sets{\mathtt{Sets}}
\def\DGA{\mathtt{DGA}}

\def\cDGA{\mathtt{DGCA}}
\def\bDGA{\mathtt{BiDGA}}
\def\bcDGA{\mathtt{BiDGCA}}

\def\D{\mathcal{D}}
\def\U{\mathcal{U}}
\def\C{\mathcal{C}}

\def\G{\mathfrak{G}}

\def\Ho{{\mathtt{Ho}}}
\nc{\Ob}{{\rm Ob}}
\nc{\Hom}{{\rm{Hom}}}
\nc{\Homcont}{{\mathcal{H}om}}
\nc{\HOM}{\underline{\rm{Hom}}}
\nc{\DER}{\underline{\rm{Der}}}
\nc{\END}{\underline{\rm{End}}}
\nc{\bSym}{\mathbf{Sym}}
\nc{\Ext}{{\rm{Ext}}}
\nc{\Rep}{{\rm{Rep}}}
\nc{\DRep}{{\rm{DRep}}}
\nc{\NCRep}{\widetilde{\rm{Rep}}}
\nc{\RAct}{{\rm{RAct}}}
\nc{\bs}{\backslash}
\nc{\ob}{{\tt{Obs}}}
\nc{\CE}{\mathcal{C}}
\nc{\nn}{{\mathrm{ab}}}
\nc{\n}{{{\natural}}}
\nc{\A}{\mathbb A}
\nc{\B}{{\mathrm{B}}}
\nc{\Ba}{\overline{\mathrm{B}}}
\nc{\bC}{\overline{C}}
\nc{\bOmega}{\boldsymbol{\Omega}}
\nc{\bB}{\boldsymbol{B}}
\nc{\EXT}{\underline{\rm{Ext}}}
\nc{\TOR}{\underline{\rm{Tor}}}
\def\H{\mathrm H}
\def\d{\mathrm D}
\def\HC{\mathrm{HC}}

\def\rHC{\overline{\mathrm{HC}}}
\def\rHD{\overline{\mathrm{HD}}}
\def\CC{\mathrm{CC}}

\def\T{\mathrm T}

\nc{\End}{{\rm{End}}}
\nc{\GL}{{\rm{GL}}}
\nc{\gl}{{\mathfrak{gl}}}
\nc{\rgl}{\overline{{\mathfrak{gl}}}}
\nc{\g}{{\mathfrak{g}}}
\nc{\aA}{{\mathfrak{a}}}
\nc{\h}{{\mathfrak{h}}}
\nc{\PGL}{{\rm{PGL}}}
\nc{\SL}{{\rm{SL}}}
\nc{\sll}{\mathfrak{sl}}
\nc{\cn}{ \mbox{\rm c\^{o}ne} }
\nc{\PSL}{{\rm{PSL}}}
\nc{\ad}{{\rm{ad}}}
\nc{\Ad}{{\rm{Ad}}}
\nc{\dlim}{\varinjlim}
\nc{\plim}{\varprojlim}
\nc{\colim}{{\tt{colim}}}

\newcommand{\HH}{{\rm{HH}}}

\newcommand{\Sym}{{\rm{Sym}}}

\newcommand{\id}{{\rm{Id}}}
\newcommand{\Der}{{\rm{Der}}}

\newcommand{\Tr}{{\rm{Tr}}}
\newcommand{\cTr}{{\rm{coTr}}}

\newcommand{\Ker}{{\rm{Ker}}}

\newcommand{\diag}{{\rm{diag}}}

\newcommand{\into}{\,\hookrightarrow\,}

\newcommand{\onto}{\,\twoheadrightarrow\,}
\newcommand{\sonto}{\,\stackrel{\sim}{\twoheadrightarrow}\,}
\def\cb{\boldsymbol{\Omega}}

\def\bs{\backslash}
\def\mfc{\mathfrak{G}}
\def\mfd{\mathfrak{H}}
\def\mfo{\mathfrak{so}}
\def\sl{\mathfrak{sl}}

\newcommand{\rar}{\xrightarrow{}}

\numberwithin{equation}{section}
\numberwithin{theorem}{section}
\numberwithin{lemma}{section}
\numberwithin{proposition}{section}
\numberwithin{corollary}{section}
\numberwithin{example}{section}
\numberwithin{remark}{section}

\newcommand{\dual}{\text{!`}}
\newcommand{\Sh}{\mathrm{Sh}}
\newcommand{\ab}{\nn}

\newcommand{\TTr}{\overline{\Tr}}

\def\arbreBA{\vcenter{\xymatrix@R=2pt@C=2pt{
&&&&\\
&&&*{}\ar@{-}[ul] & \\
&&*{}\ar@{-}[uurr] \ar@{-}[uull] \ar@{-}[d]     &&\\
&&&&
}}}

\def\arbreAB{\vcenter{\xymatrix@R=2pt@C=2pt{
&&&&\\
&*{}\ar@{-}[ur] &&& \\
&&*{}\ar@{-}[uurr] \ar@{-}[uull] \ar@{-}[d]     &&\\
&&&&
}}}

\def\arbreABC{\vcenter{\xymatrix@R=1pt@C=1pt{
&&&&&&\\
&*{}\ar@{-}[ur] &&&&& \\
&&*{}\ar@{-}[uurr] &&&&\\
&&&*{}\ar@{-}[uuurrr] \ar@{-}[uuulll] \ar@{-}[d] &&&\\
&&&&&&
}}}

\def\arbreBAC{\vcenter{\xymatrix@R=1pt@C=1pt{
&&&&&&\\
&&&*{}\ar@{-}[ul] &&& \\
&&*{}\ar@{-}[uurr] &&&&\\
&&&*{}\ar@{-}[uuurrr] \ar@{-}[uuulll] \ar@{-}[d] &&&\\
&&&&&&
}}}

\def\arbreACB{\vcenter{\xymatrix@R=1pt@C=1pt{
&&&&&&\\
&*{}\ar@{-}[ur] &&&&& \\
&&&&*{}\ar@{-}[uull] &&\\
&&&*{}\ar@{-}[uuurrr] \ar@{-}[uuulll] \ar@{-}[d] &&&\\
&&&&&&
}}}

\def\arbreBCA{\vcenter{\xymatrix@R=1pt@C=1pt{
&&&&&&\\
&&&&&*{}\ar@{-}[ul] & \\
&&*{}\ar@{-}[uurr] &&&&\\
&&&*{}\ar@{-}[uuurrr] \ar@{-}[uuulll] \ar@{-}[d] &&&\\
&&&&&&
}}}

\def\arbreCAB{\vcenter{\xymatrix@R=1pt@C=1pt{
&&&&&&\\
&&&*{}\ar@{-}[ur] &&& \\
&&&&*{}\ar@{-}[uull] &&\\
&&&*{}\ar@{-}[uuurrr] \ar@{-}[uuulll] \ar@{-}[d] &&&\\
&&&&&&
}}}

\def\arbreCBA{\vcenter{\xymatrix@R=1pt@C=1pt{
&&&&&&\\
&&&&&*{}\ar@{-}[ul] & \\
&&&&*{}\ar@{-}[uull] &&\\
&&&*{}\ar@{-}[uuurrr] \ar@{-}[uuulll] \ar@{-}[d] &&&\\
&&&&&&
}}}

\def\arbreACA{\vcenter{\xymatrix@R=1pt@C=1pt{
&&&&&&\\
&*{}\ar@{-}[ur] &&&&*{}\ar@{-}[ul] & \\
&&&&&&\\
&&&*{}\ar@{-}[uuurrr] \ar@{-}[uuulll] \ar@{-}[d] &&&\\
&&&&&&
}}}

\date{March 4, 2015}

\title{Representation Homology, Lie Algebra Cohomology and \\
the Derived Harish-Chandra Homomorphism}

\author{Yuri Berest}
\address{Department of Mathematics,
Cornell University, Ithaca, NY 14853-4201, USA}
\email{berest@math.cornell.edu}
\author{Giovanni Felder}
\address{Departement Mathematik,
ETH Z\"urich,
8092 Z\"urich, Switzerland}
\email{giovanni.felder@math.ethz.ch}
\author{Sasha Patotski}
\address{Department of Mathematics,
Cornell University, Ithaca, NY 14853-4201, USA}
\email{apatotski@math.cornell.edu}

\author{Ajay C. Ramadoss}
\address{Department of Mathematics,
Indiana University,
Bloomington, IN 47405, USA}
\email{ajcramad@indiana.edu}

\author{Thomas Willwacher}
\address{Institut f\"ur Mathematik,
Universit\"at Z\"urich,
8057 Z\"urich, Switzerland}
\email{thomas.willwacher@math.uzh.ch}

\begin{document}
\begin{abstract}
We study the derived representation scheme $ \DRep_n(A) $ parametrizing the $n$-dimensional
representations of an associative algebra $A$ over a field of characteristic zero.
We show that the homology of $ \DRep_n(A) $ is isomorphic to the Chevalley-Eilenberg homology
of the current Lie coalgebra $\, \gl_n^*(\bar{C}) $ defined over a Koszul dual
coalgebra of $A$. This gives a conceptual explanation to main results of \cite{BKR} and \cite{BR}, relating them (via Koszul duality) to classical theorems on (co)homology of current Lie algebras $ \gl_n(A) $. We extend the above isomorphism to representation schemes of Lie algebras: for a
finite-dimensional reductive Lie algebra $\g$, we define the derived affine scheme $ \DRep_{\g}(\aA) $ parametrizing the representations (in $\g$)  of a Lie algebra $\mathfrak{a}$; we show that the homology of $ \DRep_{\g}(\aA) $ is isomorphic to the Chevalley-Eilenberg homology of the Lie coalgebra $ \g^*(\bar{C}) $, where $C$ is a cocommutative DG coalgebra Koszul dual to the Lie algebra $ \aA $. We construct a canonical DG algebra map $ \Phi_{\g}(\aA):
\DRep_{\g}(\aA)^G \to \DRep_{\h}(\aA)^W $, relating the $G$-invariant part of representation homology of a Lie algebra $ \aA $ in $ \g $ to the $W$-invariant part of representation homology of $ \aA $ in a Cartan subalgebra of $ \g$. We call this map the derived Harish-Chandra 
homomorphism as it is a natural homological extension of the classical Harish-Chandra 
restriction map.

We conjecture that, for a {\it two-dimensional} abelian Lie algebra $\mathfrak{a}$,
the derived Harish-Chandra homomorphism is a quasi-isomorphism. We provide some evidence
for this conjecture, including proofs for $ \gl_2 $ and $ \sl_2 $ as well as for
$\, \gl_n,\, \sl_n,\, \mathfrak{so}_n\,$ and $\,\mathfrak{sp}_{2n}\,$
in the inductive limit as $\,n \to \infty$. For any complex reductive Lie algebra $ \g $,
we compute the Euler characteristic of $ \DRep_{\g}(\aA)^G $ in terms of matrix integrals over
$ G $ and compare it to the Euler characteristic of $ \DRep_{\h}(\aA)^W $.
This yields an interesting combinatorial identity,
which we prove for $ \gl_n $ and $ \mathfrak{sl}_n $ (for all $n$).
Our identity is analogous to the classical Macdonald identity, and
our quasi-isomorphism conjecture is analogous to the strong Macdonald conjecture
proposed in \cite{Ha1, F} and proved in \cite{FGT}. We explain this analogy
by giving a new homological interpretation of Macdonald's conjectures in terms of
derived representation schemes, parallel to our Harish-Chandra quasi-isomorphism conjecture.
\end{abstract}
\maketitle

\setcounter{tocdepth}{1}
\tableofcontents
\newpage
\section{Introduction}
This paper is a sequel to \cite{BKR} and \cite{BR} (see also \cite{BFR}),
where we study the  derived representation scheme $ \DRep_n(A) $
parametrizing the $n$-dimensional representations of an associative algebra
$A$ over a field $k$ of characteristic $0$. This
scheme is constructed in \cite{BKR} in an abstract way by extending the
classical representation functor $ \Rep_n(\,\mbox{--}\,) $ to the category of differential
graded (DG) algebras and deriving it in the sense of non-abelian homological algebra
\cite{Q1, DS}. $ \DRep_n(A) $ is represented by a commutative DG algebra as an
object in the homotopy category of DG algebras; its homology
$\, \H_\bullet[\DRep_n(A)]\,$ depends only on $A$ (and $n$)
and is called the $n$-dimensional {\it representation homology} of $A$.

From the very beginning, it was clear that representation homology must be somehow
related (dual) to the Chevalley-Eilenberg homology of matrix Lie algebras but the 
precise form of this relation has been elusive. The first goal of the present paper 
is to clarify the relation between representation homology and Lie algebra (co)homology
and offer a simple explanation of the formalism developed in \cite{BKR, BR}.
Our starting point is a basic principle of homological algebra called {\it Koszul duality}. 
In concrete terms, it can be stated as follows.
Associated to a (non-unital or augmented) algebra $A$ is a co-associative cofree DG
coalgebra $ \bB A $ called the {\it bar construction} \cite{EM}. Any natural
construction $ \C(A) $ on algebras can be formally dualized (by reversing the arrows)
to give the corresponding construction for coalgebras. When applied to
$ \bB A\, $, this dual coalgebra construction gives a natural homological 
construction for $A$ which we call the Koszul dual of $\, \C(A) \,$.
Many interesting complexes and homological structures related to
associative algebras arise in this way. For example, the Connes-Tsygan complex 
$ \CC(A) $ defining the cyclic homology of an algebra $A$ can be identified 
(up to shift in degree) with the cocommutator subspace of $ \bB A\, $, which is the Koszul dual construction of the universal algebra trace $ A \onto A/[A,A] $. The cyclic bicomplex (originally introduced in \cite{T} and \cite{LQ} to explain periodicity properties of
cyclic homology) can be interpreted as the Koszul dual of a noncommutative de Rham complex
$\,X(A) = [\ldots \to A \to \Omega^1_{A, \n} \to A \to \Omega^1_{A, \n} \to \ldots] \,$ called 
the periodic $X$-complex of $A$ (see \cite{Q3}). Another example (unrelated to cyclic homology) is Stasheff's construction of the Gerstenhaber bracket on Hochschild cohomology $ \HH^\bullet(A,A)\,$ of an algebra $A\,$: this bracket turns out  to be the Koszul dual of the usual Lie bracket on the space $ \Der(A) $ of derivations of $ A $ (see \cite{St}).

The main observation of the present paper is that $ \DRep_n(A) $ is the Koszul dual of the classical Chevalley-Eilenberg complex $ \C(\gl_n(A); k) $ computing homology of the current Lie algebra
$ \gl_n(A) $. Recall that, for any Lie algebra $ \g $, the Chevalley-Eilenberg complex $ \C(\g; k) $  has a natural structure of a cocommutative DG coalgebra.
The dual construction -- the Chevalley-Eilenberg complex $ \C^c(\G; k) $ of a Lie coalgebra $ \G $ -- has therefore the structure of a commutative DG algebra. We show
(see Theorem~\ref{main}) that, for any augmented DG algebra $A$, there is a natural isomorphism of commutative DG algebras:
\begin{equation}
\la{int1}
\DRep_n(A) \cong \C^c(\gl^*_n(\bar{\bB}A);\, k)\ ,
\end{equation}
where $\,\bar{\bB} A\,$ is the reduced bar construction of $A$ and $ \gl_n^* $
is the Lie coalgebra (linearly) dual to the matrix Lie algebra $ \gl_n\, $. Furthermore,
under \eqref{int1}, the $ \GL_n$-invariant part of $ \DRep_n(A) $
(see Section~\ref{glinv} for a precise definition) corresponds to the
{\it relative} Chevalley-Eilenberg complex of the natural Lie coalgebra map
$ \gl^*_n(\bB A) \onto \gl_n^*(k) \,$:
\begin{equation}
\la{int2}
\DRep_n(A)^{\GL} \cong \C^c(\gl^*_n(\bB A), \, \gl_n^*(k);\, k)\ ,
\end{equation}
which is Koszul dual to the relative Chevalley-Eilenberg complex
$ \C(\gl_n(A), \, \gl_n(k);\, k) $ of the Lie algebra inclusion
$ \gl_n(k) \subset \gl_n(A)\, $.

As a consequence of \eqref{int1}, the representation homology
of an algebra $ A $ is isomorphic to the Chevalley-Eilenberg homology of the
matrix Lie coalgebra  $ \gl^*_n(\bar{\bB} A) $.
This gives a conceptual explanation to many results of \cite{BKR} and \cite{BR}.
For example, the derived character maps
$\,\Tr_n(A)_\bullet:\,\HC_\bullet(A) \to \H_\bullet [\DRep_n(A)] \,$ constructed in \cite{BKR}
are Koszul dual to the natural trace maps
$\,\H_\bullet(\gl_n(A);\,k) \to \HC_{\bullet-1}(A) \,$ relating homology of matrix Lie algebras
to cyclic homology, the degree shift in cyclic homology being explained by the fact
that $ \HC_\bullet(\bB A) \cong \HC_{\bullet -1}(A) $. The stabilization theorem for
representation homology proved in \cite{BR} is Koszul dual to the classical theorem of 
Tsygan \cite{T} and Loday-Quillen \cite{LQ}, although one result does not automatically
follow from the other ({\it cf.} Section~\ref{LQTco}).

It is important to note that, in \eqref{int1} and \eqref{int2}, we can replace $ \bB A $
by any DG coalgebra $ C $, which is {\it Koszul dual} to the algebra $A$ ({\it cf.} Section~\ref{koszulduality}). In fact, the bar construction $ \bB A $ is the universal Koszul dual coalgebra of $A$ but it is often convenient to work with other coalgebras\footnote{For example, it is known ({\it cf. \cite{Le}}) that every conilpotent augmented DG coalgebra admits a (unique) minimal model, so replacing $ \bB A  $  in \eqref{int1} by its minimal model, we get another canonical complex for computing the homology of $ \DRep_n(A) $.}.

Another natural way to generalize \eqref{int1} and \eqref{int2} is to replace $ \gl_n $
by an arbitrary finite-dimensional reductive Lie algebra $ \g $. In place
of the derived representation scheme $ \DRep_n(A) $, one should consider its Lie analogue:
the derived scheme $ \DRep_{\g}(\aA) $ parametrizing the representations of a given
Lie algebra $ \aA $ in $ \g $. By a representation of $ \aA $ in $ \g $ we simply
mean a Lie algebra homomorphism $\,\aA \to \g \,$, and $ \DRep_{\g}(\aA) $
is defined by extending the representation functor $ \Rep_{\g}(\,\mbox{--}\,) $ to the
category of DG Lie algebras $ \DGL_k $ and deriving it using a natural model
structure on  $ \DGL_k $. For an arbitrary $ \g\,$, the isomorphisms \eqref{int1} and \eqref{int2} then become (see Theorem~\ref{drepg})\footnote{The
derived schemes $\DRep_n(A)$ and $ \DRep_{\g}(\aA) $ are examples of
a general operadic construction that we sketch in the Appendix.  The isomorphisms
\eqref{int1}, \eqref{int2} and \eqref{int11} are special cases of Theorem~\ref{dreplieoperads} proved in the Appendix.}
\begin{equation}
\la{int11}
\DRep_{\g}(\aA) \cong \C^c(\g^*(\bar{C});\, k)\ ,\qquad
\DRep_{\g}(\aA)^G \cong \C^c(\g^*(C),\,\g^*;\, k)\ ,
\end{equation}
where $ C $ is a cocommutative DG algebra Koszul dual to the Lie algebra $ \aA $.

Now, let $ \h $ be a Cartan subalgebra of $ \g $, and $W$ the
corresponding Weyl group. By linear duality, the natural inclusion
$ \h \into \g $ gives a morphism of Lie coalgebras $ \g^* \onto \h^* $,
which, in turn, extends to a map of commutative DG algebras
$\,\C^c(\g^*(C),\,\g^*;\, k) \onto \C^c(\h^*(C),\,\h^*;\, k) \,$.
The image of this last map consists of chains that are invariant under the
action of $W$ ({\it cf.} Proposition~\ref{Winv}): thus,
in combination with \eqref{int11}, we get a canonical map
\begin{equation}
\la{hch}
\Phi_{\g}(\aA): \DRep_{\g}(\aA)^G\, \to\, \C^c(\h^*(C),\,\h^*;\, k)^W \ .
\end{equation}
More formally, the map $\, \Phi_{\g}(\aA) $ can be defined as   
$\,\DRep_{\g}(\aA)^G \to \DRep_{\h}(\aA)^W $, which is a functorial derived extension 
of the natural restriction map $\, k[\Rep_\g(\aA)]^G \to k[\Rep_{\h}(\aA)]^W $. 
In the simplest case, 
when $ \aA $ is a one-dimensional Lie algebra,
$\,\DRep_{\g}(\aA) \cong k[\g] \,$, and  \eqref{hch} becomes
$\, k[\g]^G \to k[\h]^W $, which is the classical Harish-Chandra homomorphism\footnote{
Harish-Chandra actually defined a homomorphism $ \D(\g)^G \to \D(\h)^W $
of rings of invariant differential operators that reduces to
$ k[\g]^G \to k[\h]^W $ on the zero order differential operators (see \cite{HC}).}
(see Example~\ref{1dliehc}). In general, we will refer to \eqref{hch} as the {\it derived
Harish-Chandra homomorphism}. By a well-known theorem of Chevalley \cite{C}, the classical 
Harish-Chandra homomorphism is actually an isomorphism:
\begin{equation}
\la{chiso}
k[\g]^G \stackrel{\sim}{\to} k[\h]^W .
\end{equation}
In general, it is therefore natural to ask:
$$
\mbox{\it Is the map \eqref{hch} a quasi-isomorphism}\, ?
$$

In the present paper, we address this question for finite-dimensional {\it abelian} Lie
algebras. Note that, if $ \aA $ is abelian, a choice of linear basis in $ \aA $ identifies 
$\,\Rep_\g(\aA) $ with the commuting scheme of the reductive Lie algebra $\g $ ({\it cf.} \cite{R}). Hence, in this case, $ \DRep_{\g}(\aA) $ should be thought of as the {\it derived commuting scheme} of $\g $. We show that \eqref{hch} cannot be a quasi-isomorphism (for all $ \g $) if $\, \dim_k(\aA) \ge 3 \,$ ({\it cf.} Section~\ref{otherex});
on the positive side, we expect that the following is true (see Conjecture~\ref{conj3}):
\begin{equation}
\la{c3}
\mbox{\it If $ \aA $ is abelian and $ \dim_k(\aA) =2\,$,
then $ \Phi_{\g}(\aA) $ is a quasi-isomorphism}.
\end{equation}
In the case of $ \gl_n $, this implies (see Conjecture~\ref{conj1})
\begin{equation*}
\DRep_n(k[x,y])^{\GL} \,\cong \,k[x_1,\ldots ,x_n,y_1, \ldots ,y_n,\theta_1, \ldots ,\theta_n]^{S_n} ,
\end{equation*}
where the polynomial ring on the right is graded homologically so that the variables $
x_1,\ldots ,x_n $ and $ y_1, \ldots ,y_n $ have degree $0\,$, the variables
$ \theta_1, \ldots ,\theta_n $ have degree $1\,$, and the differential is identically
zero. The symmetric group $ S_n $ acts on this polynomial ring
diagonally by permuting the triples $ (x_i, y_i, \theta_i) $.

Conjecture \eqref{c3} can be restated in elementary terms, without using the language of derived schemes. To this end, consider the graded commutative algebra $\,k[\g \times \g] \otimes \wedge \g^* \,$, where
$\, k[\g \times \g] \,$ is the ring of polynomial functions on $ \g \times \g \,$  assigned homological degree $0$, and $\, \wedge \g^* \,$ is the exterior algebra of the dual Lie
algebra $ \g^* $ assigned homological degree $1$. The differential on
$\,k[\g \times \g] \otimes \wedge \g^* \,$ is defined by
$$
d \varphi (\xi, \eta) := \varphi([\xi, \eta])\ ,\quad \forall\,(\xi, \eta) \in
\g \times \g\, ,\quad \forall\,\varphi \in \g^* .
$$
The DG algebra $\,(k[\g \times \g] \otimes \wedge \g^*,\, d) \,$ represents the
derived scheme $ \DRep_\g(\aA) $ in the homotopy category of commutative DG algebras,
and the derived Harish-Chandra homomorphism $ \Phi_\g(\aA) $ is given in this case
by the natural restriction map (see Proposition~\ref{conj5})
\begin{equation}
\la{c33}
(k[\g \times \g] \otimes \wedge \g^*)^G \,\to\,(k[\h \times \h] \otimes \wedge \h^*)^W .
\end{equation}
Conjecture \eqref{c3} is thus equivalent to the claim that \eqref{c33} is a quasi-isomorphism.

Our second goal in this paper is to provide evidence for conjecture
\eqref{c3} and discuss some of its implications. First, in the case of $ \gl_n $,
we prove that \eqref{c3} holds for $ n = 2 $ and $ n = \infty $ (see
Theorem~\ref{conj1n2} and Theorem~\ref{surjhc} respectively); we also
prove that the map $ \H_\bullet(\Phi_{\gl_n}) $ induced by \eqref{hch} on homology
is surjective for all $ n $ (see Theorem~\ref{surjhc}). Second, we show (see
Theorem~\ref{glnvssln}) that our conjecture for $ \sl_n $ is equivalent to
that for $ \gl_n $ (and hence holds for $ \sl_2 $). Using a version of stabilization
theorem of \cite{BR}, we also verify \eqref{c3}
for the orthogonal and symplectic Lie algebras, $ \mathfrak{so}_{n} $ and
$ \mathfrak{sp}_{2n} $, in the inductive limit as $ n \to \infty \,$ (see Section~\ref{osl}).
Finally, we compute the weighted
Euler characteristics of both sides of \eqref{hch}
and show that \eqref{c3} implies the following constant term
identity (see Conjecture~\ref{conj4}):
\begin{equation}
\la{c4}
\frac{(1-qt)^{l}}{(1-q)^{l}(1-t)^{l}}\ \mathrm{CT}\, \left\{ \prod_{\alpha \in R}  \frac{(1-qte^{\alpha})(1-e^{\alpha})}{(1-qe^{\alpha})(1-te^{\alpha})}
\right\}\, = \,
\sum_{w \in W}\,\frac{\det(1 - qt\,w)}{\det(1-q\,w)\,\det(1-t\,w)}\ .
\end{equation}
Here $ R $ is a system of roots of the Lie algebra $\g$, $\,l := \dim_k(\h) \,$ is
its rank and $ \,\mathrm{CT}:\,\Z[Q] \to \Z \,$ is the constant term map defined on
the group ring of the root lattice of $ R \,$, see \eqref{cterm}. The determinants
on the right are taken in the natural (reflection) representation of $W$ on $ \h $.
One of our main results (Theorem~\ref{Tconj2}) is that the identity
\eqref{c4} holds for $ \gl_n $ and $ \sl_n $ for all $\,n \,$, in which cases
it can be described in purely combinatorial terms, see \eqref{conj2} and also
Remark~3 after Theorem~\ref{Tconj2}.

If $ k = \c $, the identity \eqref{c4} can be written in a more symmetric, integral form:
\begin{equation}
\la{ieqkt}
\int_{G}\, \frac{\det(1\,-\,qt\,\mathrm{Ad}\,g)}
{\det(1\,-\,q\,\mathrm{Ad}\,g)\,\det(1\,-\,t\,\mathrm{Ad}\,g)}\,
d g\ = \ \frac{1}{|W|}\, \sum_{w \in W}\,\frac{\det(1 - qt w)}{\det(1-q w)\,\det(1-t w)}\ .
\end{equation}
Here the integration is taken over a real compact form of the complex Lie group $ G $, which
is equipped with the invariant Haar measure $ dg $ normalized so that $ \int_{G} dg = 1 $.
Notice that, if we specialize $t=0\,$, \eqref{ieqkt} becomes the well-known identity
\begin{equation}
\la{ieqk0}
\int_{G}\, \frac{d g}{\det(1\,-\,q\,\mathrm{Ad}\,g)} \ = \ \prod_{i=1}^l \frac{1}{1-q^{d_i}}\ ,
\end{equation}
which exhibits the equality of the Poincar\'e series of both sides of the Chevalley isomorphism \eqref{chiso}. The Chevalley isomorphism \eqref{chiso} has a natural `odd' analogue: the
Hopf-Koszul-Samelson isomorphism $\,(\wedge \g)^G \cong \wedge({\tt Prim}\,\g)\,$,
identifying the space of invariants in the exterior algebra of $ \g $ with the exterior
algebra of its subspace of primitive elements (see, e.g., \cite[Chap.~10]{Me}). At the level of  Poincar\'e series,
the Hopf-Koszul-Samelson isomorphism gives the identity
\begin{equation}
\la{ihsk}
\int_{G}\, \det(1\,+\,q\,\mathrm{Ad}\,g)\,dg \ = \ \prod_{i=1}^l \,(1+q^{2d_i-1})\ ,
\end{equation}
which may be viewed as an `odd' analogue of \eqref{ieqk0}. In his original paper \cite{M}
on Macdonald conjectures, I.~Macdonald observed that \eqref{ihsk} arises as a specialization
of his constant term identity\footnote{Macdonald's constant term identity \eqref{iqtid} as well as
its generalization (the so-called inner product identity) was proved
for an arbitrary root system by I.~Cherednik \cite{Ch}, using the theory of double affine Hecke algebras. For a gentle introduction to Cherednik's theory
and his proof of Macdonald's conjectures we refer the reader to \cite{Kir}.}
\begin{equation}
\la{iqtid}
\frac{1}{|W|}\, \mathrm{CT}\left\{\prod_{n \geq 0} \prod_{ \alpha \in R} \frac{1-q^{n}e^{\alpha}}{1-q^nte^{\alpha}} \right\}
  =  \prod_{n \geq 0} \prod_{i=1}^l  \frac{(1-q^nt)(1-q^{n+1}t^{d_i-1})}{(1-q^{n+1})(1-q^nt^{d_i})} \,,
\end{equation}
and he asked  ({\it cf.} \cite{M}, Remark~2, p. 997) whether \eqref{ieqk0} admits a
$(q,t)$-generalization analogous to \eqref{iqtid}. It seems that
\eqref{c4} is an answer to Macdonald's question.

The above analogy raises the question if Macdonald's identity \eqref{iqtid} has a homological
origin similar to that of \eqref{c4}. Building on \cite{M}, Hanlon \cite{Ha1, Ha2} 
(and indepedently Feigin \cite{F}) gave an interpretation of \eqref{iqtid} in terms of cohomology of certain nilpotent Lie algebras: they made a precise conjecture ({\it cf.} \cite{Ha1}, Conjecture 1.5)
on the structure of this cohomology that entails \eqref{iqtid}. The Hanlon-Feigin conjecture
(a.k.a the strong Macdonald conjecture) was proved in full generality by Fishel, Grojnowski and
Teleman \cite{FGT}. In the present paper, we will
give a different interpretation of the Macdonald identity that clarifies its relation to the identity \eqref{c4}.

We begin with a general remark. Working with derived representation schemes $ \DRep_{\g}(\aA) $,
it is natural to put on $ \aA $ a (homological) grading: indeed, even when
$ \Rep_{\g}(\aA) $ is trivial (for example, when the grading on $ \aA $ does not allow any representations $ \aA \to \g $ other than zero), the derived scheme $ \DRep_{\g}(\aA) $ may be
highly nontrivial. If $ \aA $ is abelian and $ \dim_k(\aA) = 2 $, putting a
grading on $ \aA $ amounts to splitting it into the sum of two one-dimensional
subspaces of homological degrees $p$ and $r$ (we will write $  \aA = \aA_{p,r} $ in this case).
It turns out that the structure of $ \DRep_{\g}(\aA_{p,r}) $ essentially depends only on
the parities of $p$ and $r$ ({\it cf.} Proposition~\ref{parity}), and therefore there are 3 possibilities, which we refer to as the {\it even}, {\it mixed} and {\it odd} cases
(depending on whether $p$ and $r$ are both even, have opposite parities or
both odd).

In the even case, we show that Conjecture~\eqref{c3}
holds for $ \aA_{p,r} $ if and only if it holds for $ \aA $ with trivial grading
(i.e., $p=r=0$); thus, in this case, we expect the  Harish-Chandra
homomorphism \eqref{hch} to be a quasi-isomorphism, and the resulting constant
term identity is \eqref{c4}.

In the mixed case, the situation is quite different: the derived Harish-Chandra homomorphism 
is no longer a quasi-isomorphism, and we need to construct a new map. 
To explain the construction we return for a moment to the general situation.
In \cite{Dr}, Drinfeld introduced a natural functor on the category of Lie algebras 
that associates to a Lie algebra $ \aA $ the universal invariant bilinear form:
$$
\lambda(\aA) = \Sym_k^2(\aA)/\langle [x,y]\cdot z - x \cdot [y,z]\,:\,x,y,z \in \aA \rangle\ .
$$
Following a suggestion of Kontsevich \cite{K}, Getzler and Kapranov \cite{GK} defined
(an analogue of) cyclic homology for Lie algebras (and more generally, for algebras over an
arbitrary cyclic operad) as the non-abelian derived functor of the functor $ \lambda $.
Our starting point is a natural extension of the Drinfeld-Getzler-Kapranov construction:
for an integer $ d \ge 1 $, we consider the functor $\, \lambda^{(d)}: \DGL_k \to \Com_k \,$ assigning to a Lie algebra $ \aA $ (the target of) the universal invariant multilinear form on $ \aA $ of degree $d$ (so that $ \lambda^{(2)} = \lambda $). We prove
(see Theorem~\ref{dlambda}) that, for any $ d \,$, this functor has a left derived functor
$ \,\L\lambda^{(d)}:\ \Ho(\DGL_k) \to \Ho(\Com_k) \,$ defined
on the homotopy category of DG Lie algebras, and we let $ \HC^{(d)}_\bullet({\tt Lie}, \aA)
 $ denote the homology of $\,\L\lambda^{(d)}(\aA) \,$. The meaning of this construction
is clarified by Theorem~\ref{hodgeuel}, which asserts that the (reduced) cyclic homology
of the universal enveloping algebra $ \U(\aA) $ of any Lie algebra $ \aA $
has a canonical Hodge-type decomposition\footnote{This result provides a (partial) answer to a question of V.~Ginzburg about the
existence of Hodge decomposition for cyclic homology of noncommutative algebras.}
\begin{equation}
\la{lhodge}
\rHC_\bullet(\U \aA) \,=\, \bigoplus_{d \ge 1} \, \HC^{(d)}_\bullet({\tt Lie}, \aA)\ .
\end{equation}
The decomposition \eqref{lhodge} may be viewed as a Koszul dual of the classical Hodge decomposition of the cyclic homology of commutative algebras. In particular, as in the case of commutative
algebras ({\it cf.}  \cite{BV}), there are Adams operations $ \psi^p $ acting on
$ \rHC_\bullet(\U\aA) $, whose (graded) eigenspaces are precisely
$ \HC^{(d)}_\bullet({\tt Lie}, \aA) $ (see Section~\ref{LieHodge}).
Next, for any homogeneous invariant polynomial on $\g$ of degree $d$,
we construct natural trace maps (see Section~\ref{secdrinfeldtr})
$$
\Tr^{(d)}_{\g}(\aA):\ \L\lambda^{(d)}(\aA) \to \DRep_{\g}(\aA)^G\ ,
$$
that are analogues of the derived character maps of \cite{BKR} for Lie algebras.
Letting $ d $ run over the set $ \{d_1. \ldots, d_l\} $ of
fundamental degrees of $ \g $, we then define the homomorphism of commutative
DG algebras
\begin{equation}
\la{drtr}
\bSym_k[\oplus_{i=1}^l \Tr^{(d_i)}_{\g}(\aA)] :\
\bSym_k[\oplus_{i=1}^l \L\lambda^{(d_i)}(\mathfrak{a})] \rar \DRep_{\g}(\aA)^G ,
\end{equation}
which we call the {\it Drinfeld trace map}. In the simplest case when $ \aA $ is
a one-dimensional Lie algebra, the Drinfeld trace map coincides with the inverse of
the Chevalley isomorphism \eqref{chiso} ({\it cf.} Example~\ref{1dieliedt}).

Returning to derived commuting schemes, we may now state our last main result (see
Theorem~\ref{macdonqism}): for the two-dimensional
abelian Lie algebra $ \aA = \aA_{p,r} $ graded in such a way that $p$ and
$r$ have opposite parities, the Drinfeld trace map \eqref{drtr} is a quasi-isomorphism, and
at the level of Euler characteristics, it gives precisely the Macdonald identity
\eqref{iqtid}. Unfortunately, our proof of Theorem~\ref{macdonqism} is not entirely
self-contained: apart from results proved in this paper, it relies on one of the main
theorems of \cite{FGT}. Still, we believe that our interpretation of the Macdonald
identity in terms of representation homology is, in some respects, more natural than
the classical one in terms of Lie cohomology, and at the very least, it clarifies the
relation between \eqref{iqtid} and  \eqref{c4}.
We would also like to mention an interesting recent paper \cite{Kh}
which gives yet another homological interpretation of Macdonald's theory in terms
of representation theory of current Lie algebras. It seems that our approach is related
to that of \cite{Kh} in a very natural way (via Koszul duality at the
level of derived module categories); it would be interesting to study this relation
in detail, especially with a view towards understanding \eqref{c4}.

Finally, we have to mention that, in the odd case (when $p$ and $r$ are both odd), the structure of
the derived commuting scheme $ \DRep_{\g}(\aA_{p,r})^G $ remains mysterious to us. We do not
know whether there exists a numerical identity analogous to \eqref{c4} and \eqref{iqtid} in this case.


The paper is organized as follows. Section~\ref{sec2} is preliminary: here, we recall
basic facts of differential homological algebra and review the construction of derived representation schemes from \cite{BKR} and \cite{BR}.
In Section~\ref{sec3}, we prove
our first main result, Theorem~\ref{main}, which relates representation homology to Lie (co)homology, and discuss its implications.
In Section~\ref{sec4}, we construct the derived Harish-Chandra homomorphism for representation schemes of associative algebras and state our main conjecture for $ \gl_n $ (see Conjecture~\ref{conj1}).
The key results of this section are Theorem~\ref{conj1n2} (proof of Conjecture~\ref{conj1}
for $n=2$), Theorem~\ref{surjhc} (surjectivity of the Harish-Chandra homomorphism on homology for all $n $ and bijectivity for $ n = \infty $) and Theorem~\ref{q-conj1} (proof of the analog of Conjecture~\ref{conj1} for $q$-polynomial algebras). Next,
in Section~\ref{sec5}, we compute Euler characteristics and deduce our
constant term identity in the case of $ \gl_n $, see \eqref{conj2}. The main result of this section is Theorem~\ref{Tconj2}, which proves the
constant term identity for $ \gl_n $ for all $n$.
In Section~\ref{sec6}, we define representation homology for an arbitrary (reductive) Lie algebra $ \g $. The main result of this section, Theorem~\ref{drepg}, is a natural generalization of Theorem~\ref{main}.
In Section~\ref{sec7}, we construct the derived Harish-Chandra homomorphism and the Drinfeld trace maps for representation schemes of Lie algebras. The main result is Theorem~\ref{hodgeuel} that gives a Hodge decomposition for the cyclic homology of
the universal enveloping algebra of any DG Lie algebra.
In Section~\ref{sec8}, we consider the derived commuting variety
 of a reductive Lie algebra; we state our quasi-isomorphism conjecture in full generality (Conjecture~\ref{conj3}) and deduce the corresponding constant term identity (Conjecture~\ref{conj4}); we also provide some evidence in favor of Conjecture~\ref{conj3}, including its verification for orthogonal and symplectic Lie algebras in the limit $ n\to \infty $ (Theorem~\ref{oddorthlimit}).
In Section~\ref{Macd}, we explain the relation of our conjectures to the classical Macdonald conjecture \cite{M} and the strong Macdonald conjecture of \cite{Ha1} and \cite{FGT}. The main result of this section is Theorem~\ref{macdonqism} showing that the Drinfeld trace map is
a quasi-isomorphism in the mixed case.
Finally, in Appendix~\ref{appendix}, we generalize our construction
of derived representation schemes to algebras over an arbitrary binary
quadratic operad. This puts some of the main results of present paper in proper perspective.

\subsection*{Acknowledgements}{\footnotesize
We would like to thank A.~Alekseev, I.~Cherednik, V.~Ginzburg, I.~Gordon,
D.~Kaledin, V.~Lunts, E.~Meinrenken and B.~Tsygan for interesting discussions and comments.
We are particularly grateful to I.~Cherednik for bringing to our attention the paper
\cite{Kh} and to V.~Ginzburg for several useful suggestions and for the reference \cite{G2}.
Research of G.F. and T.W. was partially supported by the NCCR SwissMAP,
funded by the Swiss National Science Foundation.}

\section{Preliminaries} \la{sec2}
In this section, we introduce notation and recall some basic results from the literature. In particular, we review the construction of derived representation schemes and derived character
maps from~\cite{BKR}.

\subsection{Notation and conventions}
\la{notation}
Throughout this paper, $ k $ denotes a base field of characteristic zero. An unadorned tensor product
$\, \otimes \,$ stands for the tensor product $\, \otimes_k \,$ over $k$. An algebra means an
associative $k$-algebra with $1$; the category of such algebras is denoted $ \Alg_k $. Unless stated otherwise, all differential graded (DG) objects are equipped with differentials of degree $-1$, and the Koszul sign rule is systematically used. The homological degree of a homogeneous element $x$ in a graded vector space $V$ will often be denoted by $|x|$. If $V$ is a graded $k$-vector space, we denote by
$T_kV$ its tensor algebra and by $\bSym_k(V)$ its graded symmetric algebra. Thus, $\bSym_k(V)= \Sym_k(V_{\mathrm{ev}}) \otimes \Lambda_k(V_{\mathrm{odd}})$, where $V_{\mathrm{ev}}$ and $V_{\mathrm{odd}}$ are the even and odd components of $V$ respectively. $\Sym^d(V)$ shall denote the $d$-th symmetric power of $V$. Thus, $\Sym^d(V)\,=\,\oplus_{p+q=d} \Sym^p(V_{\mathrm{ev}}) \otimes \wedge^q(V_{\mathrm{odd}})$ and $\bSym_k(V)\,=\,\oplus_{d=0}^{\infty} \Sym^d(V)$. The graded symmetric {\it coalgebra} on $V$ will be denoted by $\bSym^{c}(V)$.
For $V, W \,\in\,\Com_k$, we define $\mathbf{Hom}(V, W)$ to be the complex whose space of $p$-chains is
$$ \mathbf{Hom}(V, W)_p\,:=\, \prod_n \Hom_k(V_n, W_{n+p}) $$
and the differential is given by $\, df := d_W \circ f -{(-1)}^p f \circ d_V\,$,
where $\, f \in  \mathbf{Hom}(V, W)_p\,$.

\subsection{The bar/cobar construction}
In this and the next section, we briefly recall some classical results from
differential homological algebra which are needed for the present paper.
An excellent modern reference for this material is \cite{LV}.
\subsubsection{DG algebras and coalgebras}
\la{basicdga}
Let $\DGA_k$ denote the category of associative unital DG algebras over $k$ with differential of degree $-1$. Recall that $ A \in \DGA_k $ is {\it augmented} if it is given together with a DG algebra map $ \varepsilon: A \to k $. A morphism of augmented algebras $\, (A, \varepsilon_A) \to (B, \varepsilon_B) \,$ is a morphism $ f: A \to B $ in $ \DGA_k $ satisfying $ \varepsilon_B \circ f = \varepsilon_A $.
We denote the category of augmented DG algebras by $ \DGA_{k/k} $. This category
is known to be equivalent to the category $ \DGA $ of non-unital DG algebras: the
mutually inverse functors are given by
$$
B' \mapsfrom B\ ,\ \,\,\,   \DGA_{k/k} \,\rightleftarrows \, \DGA \ ,\,\,\,
A \mapsto \bar{A}
$$
where $ \bar{A} := \Ker(\varepsilon) $ is the kernel of the augmentation map of $A$
and $ B' := k \oplus B $ is the unitalization of $B$ equipped with the canonical projection $ \varepsilon: B' \to k $. Similarly, we define the (equivalent) categories $\cDGA_{k/k}$ and $ \cDGA $ of commutative DG algebras: augmented and nonunital, respectively.

Let $\DGC$ (resp., $\DGC_k$) denote the
category of coassociative (resp., coassociative counital) DG coalgebras over $k$. We shall often work with augmented coassociative counital DG coalgebras $ C $ which are {\it conilpotent}
in the sense that
\begin{equation}
\la{cocom}
\bar{C} = \bigcup_{n \ge 2}\, \Ker\,[\,C \xrightarrow{\Delta^{(n)}} C^{\otimes n}
\onto \bar{C}^{\otimes n}\,]
\end{equation}
where $ \Delta^{(n)} $ denotes the $n$-th iteration of the comultiplication map
$\, \Delta_C: C \to C \otimes C \,$ and
$\, \bar{C} \,$ is the cokernel of the augmentation map $ \varepsilon_C: k \to C \,$.
We denote the category of such coalgebras by $ \DGC_{k/k} $. Similarly, $\cDGC$ (resp., $\cDGC_k$) shall denote the category of cocommutative (resp., cocommutative counital) DG coalgebras over $k$ and $\cDGC_{k/k}$ shall denote the category of coaugmented conilpotent cocommutative DG coalgebras over $k$.

\subsubsection{Twisting cochains}
Given an algebra $ R \in \DGA_{k/k} $ and a coalgebra $ C \in \DGC_{k/k} $,
we define a {\it twisting cochain} $\, \tau: {C} \to {R} $ to be a linear map of
degree $-1$ satisfying
$$
d_R \,\tau + \tau\,d_C + m_R \,(\tau \otimes \tau)\,\Delta_C = 0\,,\,\,\,\tau \circ \varepsilon_C =0\,,\,\,\,\varepsilon_R \circ \tau=0\,,
$$
where $ d_R $ and $ d_C $ are the differentials on $ R $ and $ C $ and $ m_R $ is the
multiplication map on $ R $.
We write $ \Tw(C,R) $ for the set of all twisting cochains from $C$ to $R$. It is easy to show that, for a fixed algebra $ R $, the functor
$$
\Tw(\mbox{--}, R):\,\DGC_{k/k} \to \Sets\ ,\quad C \mapsto \Tw(C,R)\ ,
$$
is representable; the corresponding coalgebra $ \bB(R) \in \DGC_{k/k} $ is called the {\it bar construction} of $ R \,$: it is defined as the tensor coalgebra $ T_k(\bar{R}[1]) $ with differential lifting $ d_R $ and $ m_R $. Dually, for a fixed coalgebra $ C $, the functor
$$
\Tw(C, \mbox{--}):\,\DGA_{k/k} \to \Sets\ ,\quad R \mapsto \Tw(C,R)\ ,
$$
is corepresentable; the corresponding algebra $ \bOmega(C) \in \DGA_{k/k} $ is called the
{\it cobar construction} of $ C \,$: it is defined as the tensor algebra $ T_k(C[-1]) $
with differential lifting $ d_C $ and $ \Delta_C $.
Thus, we have canonical isomorphisms
\begin{equation}
\la{fism}
\Hom_{\DGA_{k/k}}(\bOmega(C), R)\, = \,\Tw(C,R)\, =\, \Hom_{\DGC_{k/k}}(C, \bB(R))
\end{equation}
showing that $ \bOmega: \DGC_{k/k} \rightleftarrows \DGA_{k/k}: \bB $ are adjoint functors.

\subsubsection{Koszul-Moore equivalence}
The categories $\DGA_k$ and $\cDGA_k$ carry natural model structures (see~\cite{Hi}). The weak equivalences in these model categories are the quasi-isomorhisms and the fibrations are the degreewise surjective maps. The cofibrations are characterized in abstract terms: as morphisms satisfying the left lifting property with respect to the acyclic fibrations. The model structures on $\DGA_{k}$ and $\cDGA_{k}$ naturally induce model structures on the corresponding categories of augmented algebras ({\it cf.} \cite[3.10]{DS}). In particular, a morphism $ f\,:\,A \rar  B$ in $\DGA_{k/k}$ is a weak equivalence (resp., fibration; resp., cofibration) iff $ f\,:\,A \rar B$ is a weak equivalence (resp., fibration; resp., cofibration) in $\DGA_k $. All objects in $ \DGA_{k} $ and $ \DGA_{k/k} $ are fibrant. The cofibrant objects in $ \DGA_{k/k} $ can be described more explicitly than in $ \DGA_k $: every cofibrant
$ A \in \DGA_{k/k} $ is isomorphic to a retract of $ \boldsymbol{\Omega}(C) $, where $ \boldsymbol{\Omega}(C) $ is the cobar construction of an augmented conilpotent coassociative DG coalgebra $ C $ (see \cite[Theorem~4.3]{Ke}).

There is a dual model structure on $ \DGC_{k/k} $, where the weak equivalences are the
morphisms $f$ such that $ \bOmega(f) $ is a quasi-isomorphism. A well-known
result, due to J.~C.~Moore, asserts that the model categories $ \DGA_{k/k} $ and
$ \DGC_{k/k} $ are Quillen equivalent: more precisely,
\begin{theorem}[\cite{Mo}]
\la{cbbeq}
The pair of adjoint functors
\begin{equation} \la{cobarbar} \bOmega: \DGC_{k/k} \rightleftarrows \DGA_{k/k}: \bB \end{equation}
is a Quillen equivalence.
In particular, the functors~\eqref{cobarbar} induce mutually inverse equivalences between the homotopy categories
$$ \L \cb\,:\,\Ho(\DGC_{k/k}) \,\cong\, \Ho(\DGA_{k/k})\,:\, \R \bB \,\text{.}$$
\end{theorem}
The proof of this theorem can be found, for example, in \cite{LV}.

\subsection{Koszul duality}
\la{koszulduality}

Let $C \in \DGC_{k/k}$ and let $R \in \DGA_{k/k}$. Let $ \Mod(R) $ denote the category of right DG modules over $ R $, and dually
let $ \cMod(C) $ denote the category of right DG comodules over $ C $ which
are conilpotent in a sense similar to \eqref{cocom}.
Given a twisting cochain $\, \tau \in \Tw(C,R) \,$  one can define the functors
\begin{equation*}
\mbox{--} \otimes_{\tau} C :\, \Mod(R) \rightleftarrows \cMod(C)\,:
\mbox{--} \otimes_{\tau} R
\end{equation*}
called the {\it twisted tensor products}. Specifically, if $ M \in \Mod(R) $,
then $ M \otimes_{\tau} C $ is defined to be the DG $C$-comodule whose
underlying graded comodule is $ M \otimes_k C $ and whose differential is given by
$$
d = d_M \otimes \id + \id \otimes d_C +
(m \otimes \id)\,(\id \otimes \tau \otimes \id)\,(\id \otimes \Delta)\ .
$$
Similarly, for a DG comodule $ N \in \cMod(C) $, one defines a
DG $R$-module $ N \otimes_\tau R $. In the same fashion, for a DG bicomodule $\mathcal N \,\in\, \mathtt{Bicomod}(C)$, one defines a DG $R$-bimodule
$ R_{\tau} \otimes \mathcal N \otimes_{\tau} R$.

Next, recall that the {\it derived category}
$ \D(R) $ of DG modules is obtained by localizing $ \Mod(R) $ at the class of all quasi-isomorphisms. To introduce the dual notion for DG comodules one has to replace
the quasi-isomorphisms by a more restricted class of morphisms in $ \cMod(C) $.
We call a morphism $f$ in $ \cMod(C) $ a weak equivalence if
$\,f \otimes_{\tau_C} \bOmega(C) \,$ is quasi-isomorphism in $ \Mod\,\bOmega(C)\,$, where $ \tau_C: C \to \bOmega(C) $ is the universal twisting cochain corresponding to the identity map under \eqref{fism}. The {\it coderived category} $ \D^c(C) $ of DG comodules is then defined by localizing $ \cMod(C) $ at the class of weak equivalences. It is easy to check that the twisted tensor products induce a pair of adjoint functors
\begin{equation}
\la{quieq1}
\mbox{--} \otimes_{\tau} C \, :\, \D(R) \rightleftarrows \D^c(C)\,: \,
\mbox{--} \otimes_{\tau} R\ .
\end{equation}
The following theorem characterizes the class of twisting cochains for which \eqref{quieq1}
are equivalences.
\bthm[see \cite{LV}, Theorem~2.3.1]
\la{Kosc}
For $\,\tau \in \Tw(C,R)\,$, the following are equivalent:
\begin{enumerate}
\item[{\rm (i)}] the functors \eqref{quieq1} are mutually inverse equivalences of categories;
\item[{\rm (ii)}] the complex $\,C \otimes_\tau R \,$ is acyclic;
\item[{\rm (iii)}] the complex $\,R \otimes_\tau C \,$ is acyclic;
\item[{\rm (iv)}] the natural morphism $ R \otimes_\tau C \otimes_\tau R \stackrel{\sim}{\to} R $ is a quasi-isomorphism;
\item[{\rm (v)}] the morphism $ \bOmega(C) \stackrel{\sim}{\to} R $ corresponding to $\tau $ under \eqref{fism} is a quasi-isomorphism in $ \DGA_{k/k}$;
\item[{\rm (vi)}] the morphism $ C \stackrel{\sim}{\to} \bB(R) $ corresponding to $\tau $ under
\eqref{fism} is a weak equivalence in $ \DGC_{k/k} $.
\end{enumerate}
If the conditions {\rm (i)} - {\rm (vi)} hold, the DG algebra $R$ is determined by $C$ up to isomorphism in
$ \Ho(\DGA_{k/k}) $ and the DG coalgebra $C$ is determined by $ R $ up to isomorphism
in $ \Ho(\DGC_{k/k}) $.
\ethm
A twisting cochain $ \tau \in \Tw(C,R) $ satisfying the conditions of
Theorem~\ref{Kosc} is called {\it acyclic}. In this case, the DG coalgebra $C$
is called {\it Koszul dual} to the DG algebra $ R $ and $ R $
is called {\it Koszul dual} to $ C $.

\subsection{Derived representation schemes}
\la{secdrep}
Derived representation schemes were originally discussed in \cite{CK} as part of a general program of deriving {\it Quot}  schemes and other moduli spaces in algebraic geometry. A different, more algebraic approach was developed in \cite{BKR}, where it was shown, among other things, that the representation functor is a (left) Quillen functor on the model category of associative DG algebras. In this section, we briefly review the basic construction of \cite{BKR}.

\subsubsection{The representation functor}
\la{S1.2.1}
For an integer $ n \ge 1 $, let $ \M_n(k) $ denote the algebra of $ n \times n $ matrices with entries in $ k $. If $ B \in \DGA_k $, we write $\, \M_n(B) := \M_n(k) \otimes B \,$; this
gives a functor on the category of DG algebras: $\,\M_n(\,\mbox{--}\,):\,\DGA_k \to \DGA_k \,$.
Next, we define
\begin{equation}
\la{root}
\sqrt[n]{\,\mbox{--}\,}\,:\,\DGA_k \rar \DGA_k\ ,\quad
A \mapsto  [A \ast_k \M_n(k)]^{\M_n(k)}\ ,
\end{equation}
where $\,A \,\ast_k\, \M_n(k) \,$ is the free product (coproduct) in $ \DGA_k $ and $\,[\,\ldots\,]^{\M_n(k)} $ denotes the (graded) centralizer of $ \M_n(k) $ as the subalgebra in $\,A \,\ast_k\, \M_n(k) \,$.

The following proposition is a generalization (to DG algebras) of a classical
result of G.~Bergman (see \cite{BKR}, Proposition~2.1).
\bprop
\la{Bprop}
The functor $\,\sqrt[n]{\,\mbox{--}\,}\,$ is left adjoint to
$\,\M_n(\,\mbox{--}\,)\,$.
\eprop

Now, let $ \cDGA_k $ be the category of commutative DG algebras. The natural inclusion 
functor $\, \cDGA_k \to \DGA_k \,$ has an obvious left adjoint given by
\begin{equation}
\la{ab}
(\,\mbox{--}\,)_{\nn} :\ \DGA_k \rar \cDGA_k \ ,\quad A \mapsto A_\nn := A/\langle [A,A]\rangle\ ,
\end{equation}
where $\,\langle [A,A]\rangle\,$ is the two-sided DG ideal of $A$ generated by the
graded commutators. Combining  \eqref{root} and \eqref{ab}, we define
\begin{equation}
\la{rootab}
(\,\mbox{--}\,)_n :\ \DGA_k \rar \cDGA_k\ ,\quad A \mapsto A_n := (\!\sqrt[n]{A})_{\nn}\ .
\end{equation}
Then, as a consequence of Proposition~\ref{Bprop}, we get
\begin{theorem}
\la{nS2.1t1}
The functor $\,(\,\mbox{--}\,)_n\,$ is left adjoint to $\,\M_n(\,\mbox{--}\,)\,$
on the category of commutative DG algebras. Thus, for any $\,A \in \DGA_k\,$, the
DG algebra $ A_n $ (co)represents the functor of points of the affine DG scheme
\begin{equation}
\la{rep}
\Rep_n(A):\ \cDGA_k \to \Sets\ ,\quad B \mapsto \Hom_{\DGA_k}(A,\, \M_n(B))\ ,
\end{equation}
parametrizing the $n$-dimensional representations of $A$.
\end{theorem}

\vspace{1ex}

\noindent
Theorem~\ref{nS2.1t1} implies that there is a natural bijection
\begin{equation}
\la{eS1.2.1}
\Hom_{\DGA_k}(A,\M_n(B)) \,\cong \, \Hom_{\cDGA_k}(A_n, B) \ ,
\end{equation}
functorial in $ A \in \DGA_k $ and $B \in \cDGA_k $. Accordingly, we refer to
\eqref{rootab} as the functor of $n$-dimensional representations, or
simply the {\it $n$-th representation functor} on $\DGA_k $.
Letting $\,B = A_n\,$ in \eqref{eS1.2.1}, we have a canonical DG algebra map
$\, \pi_n: \, A \to \M_n(A_n) \,$, called the {\it universal}\,
$n$-dimensional representation of $A$.

The algebra $\, A_n \,$ has the following
canonical presentation described in \cite[Section~2.4]{BKR}.
Let $ \{e_{ij}\}_{i,j=1}^n $ be the basis of elementary matrices in $ \M_n(k) $.
For each $\,a \in A \,$, define the `matrix' elements of $ a $ in
$\,A \ast_k \M_n(k) \,$ by
$$
a_{ij} := \sum_{k=1}^n\, e_{ki}\,a\,e_{jk}\ .
$$
Then $\,a_{ij} \in \sqrt[n]{A} \,$ for all $\,i,j=1,2,\ldots, n\,$, and we
also write $ a_{ij} $ for the corresponding elements in $ A_n $.

\blemma
\la{exppr}
The algebra $A_n$ is generated by the elements $\, \{ a_{ij}\, :\, a \in A\} \,$
satisfying the relations
$$
(a+b)_{ij} = a_{ij} + b_{ij}\ ,\quad
(ab)_{ij} = \sum_{k=1}^n\, a_{ik} \, b_{kj}\ ,\quad  
\lambda_{ij} = \delta_{ij}\,\lambda\ ,\quad \forall\,a,b \in A\ ,\ \lambda \in k\ .
$$
The differential on $ A_n $ is determined by the formula: $ d(a_{ij}) = (d a)_{ij} \,$.
The universal $n$-dimensional representation of $A$ is given by
$$
\pi_n: A \to \M_n(A_n) \ ,\quad a \mapsto \|a_{ij}\|\ .
$$
\elemma

\vspace{1ex}

In this paper, we will work with augmented DG algebras. Note that, if
$\, A \,$ is augmented, then
 $\, A_n $ has a natural augmentation $\, \varepsilon_n : A_n \to k $ coming from \eqref{rootab} applied to the augmentation map of $A$. This defines a functor $\,\DGA_{k/k} \to \cDGA_{k/k} $, which we again denote by $ (\,\mbox{--}\,)_n $. On the other hand, the matrix algebra
functor  can be modified in the following way:
\begin{equation} \la{augmatrixfunctor}
\M_n'(\mbox{--}):\,\cDGA_{k/k} \rar \DGA_{k/k}\ ,\qquad \M_n'(B) := k \oplus \M_n(\bar{B})\ ,
\end{equation}
With this modification, Theorem~\ref{nS2.1t1} holds for augmented DG algebras.
Moreover, as a special case of~\cite[Theorem~2.2]{BKR}, we have

\begin{theorem}
\la{nS2.1t2}
$(a)$ The adjoint functors $\,(\,\mbox{--}\,)_n\,:\, \DGA_{k/k} \rightleftarrows \cDGA_{k/k} \,:\, \M_n'(\mbox{--}) \,$ form a Quillen pair.

$(b)$ The functor $\, (\, \mbox{--}\,)_n\,:\,\DGA_{k/k} \rar \cDGA_{k/k}$ has a total left derived functor defined by
$$
\L(\,\mbox{--}\,)_n:\,\Ho(\DGA_{k/k}) \to \Ho(\cDGA_{k/k}) \ ,
\quad A  \mapsto  (QA)_n\ ,
$$
where $QA \stackrel{\sim}{\twoheadrightarrow} A$ is any cofibrant resolution of $A$ in $\DGA_{k/k}$.

$(c)$ For any $A$ in $\DGA_{k/k}$ and $B$ in $\cDGA_{k/k}$, there is a canonical isomorphism
$$ \Hom_{\Ho(\cDGA_{k/k})}(\L(A)_n \,,\, B)\,\cong\, \Hom_{\Ho(\DGA_{k/k})}(A\,,\,\M_n'(B))\,\text{.}$$
\end{theorem}
\vspace{1ex}

\begin{definition} Given $ A \in \Alg_{k/k}\,$, we define
$\, \DRep_n(A) := \L(QA)_n \,$, where $ QA \sonto A $ is a cofibrant resolution of $ A $ in $\DGA_{k/k} $. The homology of $ \DRep_n(A) $
is an augmented (graded) commutative algebra, which is independent of
the choice of resolution (by Theorem~\ref{nS2.1t2}). We set
\begin{equation}
\la{rephom} \mathrm{H}_{\bullet}(A,n)\,:=\,\mathrm{H}_{\bullet}[\DRep_n(A)] \end{equation}
and call \eqref{rephom} the  {\it $n$-dimensional representation homology} of $A$.
\end{definition}

\vspace{1ex}

By \cite[Theorem~2.5]{BKR}, for any $ A \in \Alg_k $, there is a natural isomorphism of algebras
$$
\H_0(A,n) \cong A_n \ .
$$
Hence $\, \DRep_n(A) $ may indeed be viewed as  a `higher' derived version of the representation
functor \eqref{rootab}.

\subsubsection{$\GL$-invariants}
\la{glinv}
The group $ \GL_n(k) $ acts naturally on  $\,A_n\,$ by DG algebra
automorphisms. Precisely, each $ g \in \GL_n(k) $ defines a unique automorphism of $\,A_n\,$ corresponding under \eqref{eS1.2.1} to the composite map
\begin{equation}
\la{GLact}
A \xrightarrow{\pi_n} \M_n(A_n) \xrightarrow{\text{Ad}(g)} \M_n(A_n)\ .
\end{equation}
This action is natural in $A$ and thus defines the functor
\begin{equation}
\la{vgl}
(\,\mbox{--}\,)_n^{\GL} :\ \DGA_{k/k} \to \cDGA_{k/k}\ ,\quad
A \mapsto A_n^{\GL_n}\ ,
\end{equation}
which is a subfunctor of the representation functor $ (\,\mbox{--}\,)_n $.
On the other hand, there is a natural action of $ \GL_n(k) $ on the $n$-th
representation homology of $ A $ so we can form the invariant subalgebra
$ \H_\bullet(A, n)^{\GL_n} $. The next theorem, which is a consequence of
\cite[Theorem~2.6]{BKR}, shows that these two constructions agree.
\begin{theorem}
\la{nS2.1t5}
$(a)$ The functor $(\mbox{--})_n^{\GL}$ has a total left derived functor
$$
\L[(\,\mbox{--}\,)_n^{\GL}]:\,\Ho(\DGA_{k/k}) \to \Ho(\cDGA_{k/k})\,,\,\,\,\, A \mapsto (QA)_n^{\GL} \ .
$$

$(b)$ For any $ A \in \DGA_{k/k} $, there is a natural isomorphism of graded algebras
$$
\H_\bullet[\L(A)_n^{\GL}] \cong \H_\bullet(A, n)^{\GL_n}\ .
$$
\end{theorem}
\noindent
Abusing notation we will often write $ \DRep_n(A)^{\GL} $ instead of $ \L(A)_n^{\GL} $
for any DG algebra $A$.

\subsubsection{Trace maps}
\la{trmaps}
Recall that, for an augmented DG algebra $ R \in \DGA_{k/k} $, we denote by
$ \bar{R} \subset R $ the kernel of the augmentation map of $R$. Now, we set
$$
R_\n := \bar{R}/[\bar{R},\bar{R}] \,\cong\, R/(k + [R,R])\ ,
$$
where $ [\bar{R},\bar{R}] $ is the subcomplex of $ \bar{R} $ spanned by the
commutators in $\bar{R} $. This defines the functor
\begin{equation}
\la{cycf}
(\,\mbox{--}\,)_\n : \,\DGA_{k/k} \to \Com_{k}\ ,\quad R \mapsto R_{\n} \ ,
\end{equation}
which we call the (reduced) cyclic functor.

The next theorem, which is a well-known result due to
Feigin and Tsygan, justifies our terminology.
\begin{theorem}[\cite{FT}]
\la{ftt}
$ (a) $ The functor \eqref{cycf} has a total left derived
functor
$$
\L(\,\mbox{--}\,)_\n:\,\Ho(\DGA_{k/k}) \to \Ho(\Com_{k})\ ,\quad A \mapsto (QA)_\n \ ,
$$
where $ QA $ is a(ny) cofibrant resolution of $A$ in $ \DGA_{k/k} $.

$(b)$ For $ A \in \Alg_{k/k} $, there is a natural isomorphism
of graded vector spaces
$$
\H_\bullet[\L(A)_\n] \cong \rHC_\bullet(A)\ ,
$$
where $ \rHC_\bullet(A) $ denotes the (reduced) cyclic homology of
$A$.
\end{theorem}
\noindent

For a conceptual proof of Theorem~\ref{ftt} we refer to \cite{BKR}, Section~3.

\vspace{2ex}

Now, fix $ n \ge 1 $ and, for $ R \in \DGA_k $, consider the composite map
$$
R \xrightarrow{\pi_n} \M_n(R_n) \xrightarrow{\Tr} R_n
$$
where $ \pi_n $ is the universal representation of $R$ and $\,\Tr \,$ is the usual matrix trace. This map factors through $ R_\n $ and its image lies in $ R_n^{\GL} $. Hence, we get a morphism of complexes
\begin{equation}
\la{e2s1}
\Tr_n(R)_{\bullet}:\ R/[R,R] \to R_n^{\GL}\ ,
\end{equation}
which extends by multiplicativity to a map of graded commutative algebras
\begin{equation}
\la{trm}
\bSym\, \Tr_n(R)_{\bullet}:\ \bSym(R/[R,R]) \to R_n^{\GL}\ .
\end{equation}
If $ R \in \DGA_{k/k} $ is augmented, the natural inclusion $\,\bar{R} \into R \,$ induces
a morphism of complexes $\, R_\n \to R/[R,R] $. Composed with \eqref{e2s1} this
defines a morphism of functors that
extends to a morphism of the derived functors from $ \Ho(\DGA_{k/k}) $ to $ \Ho(\Com\,k) \,$:
\begin{equation}
\la{trm2}
\L\Tr_n:\, \L(\,\mbox{--}\,)_\n \to \L(\mbox{--}\,)_n^\GL\, .
\end{equation}
Now, for an ordinary $k$-algebra $ A \in \Alg_{k/k} $, applying \eqref{trm2} to a cofibrant resolution  $\, R = QA \,$  of $ A $ in $ \DGA_{k/k}$, taking homology and using the identification of Theorem~\ref{ftt}$(b)$, we get natural maps
\begin{equation}
\la{trm3}
\Tr_n(A)_{\bullet} :\,\rHC_{\bullet}(A) \to
\H_{\bullet}(A,n)^{\GL_n}\ ,\quad \forall\,n \ge 0 \ .
\end{equation}
In degree zero, $ \Tr_n(A)_0 $ is induced by the obvious linear map
$ A \to k[\Rep_n(A)]^{\GL_n} $ defined by taking characters of representations. Thus, the higher components of \eqref{trm3} may be thought of as derived (or higher) characters of $n$-dimensional representations of $A$. For each $ n \ge 1 $, these characters assemble to a single homomorphism of graded commutative algebras which we denote
$$
\bSym\, \Tr_n(A)_{\bullet}:\,
\bSym\,[\rHC_{\bullet}(A)] \to
\mathrm{H}_{\bullet}(A,n)^{\GL}\ .
$$
An explicit formula evaluating $\Tr_n(A)_{\bullet}$ on cyclic chains is given in~\cite[Section 4.3]{BKR}.

\section{Representation homology vs Lie (co)homology} \la{sec3}
The main result of this section (Theorem~\ref{main}) identifies the representation homology of a DG algebra $A$ with homology of the Lie coalgebra $\gl_n^{\ast}(C)$ defined over a Koszul dual DG coalgebra $C$. This crucial observation is the starting point for the present paper.

\subsection{Chevalley-Eilenberg complexes}
\la{CEcoalg}
First, we recall the definition of the classical Chevalley-Eilenberg complex (cf. \cite{Q2}, Appendix B). If $\g$ is a DG Lie algebra, the Chevalley-Eilenberg complex of
$\g$ with trivial coefficients is the (coaugmented, conilpotent) cocommutative DG coalgebra
$$
\C(\g;k)\,:=\, (\bSym^c(\g[1]),\, d_1+d_2)\ ,
$$
where $d_1$ is induced by the differential on $\g$ and $d_2$ is the coderivation whose corestriction to $\g[1]$ is given by the composite map
$$
\bSym^2(\g[1])\,\cong\, k[1] \otimes k[1] \otimes \wedge^2 \g
\xrightarrow{\mu_1 \otimes [\mbox{--}\,,\,\mbox{--}]} k[1] \otimes \g \,\cong\,\g[1] \,\text{.}
$$
Here, $\,\mu_1: k[1] \otimes k[1] \to k[1] \,$ is the natural map of degree $-1$
and $[\mbox{--}\,,\,\mbox{--}]$ is the Lie bracket on $\g$.
If $\h \subset \g$ is a DG Lie subalgebra, the relative Chevalley-Eilenberg complex
$ \C(\g,\h;k) $ is the DG coalgebra $ (\bSym^c[(\g/\h) [1]], d_1+d_2)_{\h} $, where
$ (\,\mbox{--}\,)_{\h} $ denotes the subcomplex of $\h$-coinvariants in
$ \bSym^c[(\g/\h) [1]] $. We denote the homology of the complex $\C(\g;k)$ (resp.,
$\,\C(\g,\h;k)$) by $\,\H_{\bullet}(\g;k)\,$ (resp., $\,\H_{\bullet}(\g,\h;k)\,$).

Now, dually, if $\mfc$ is a DG Lie coalgebra with Lie cobracket $]\, \mbox{--}\, [ \,:\,\mfc \rar \wedge^2\mfc$, the Chevalley-Eilenberg complex of $\mfc$ is defined to be the (augmented) commutative DG algebra
$$
\C^c(\mfc;k)\,:=\, (\bSym(\mfc[-1]),\, d_1+d_2)\ ,
$$
where $d_1$ is induced by the differential on $\mfc$ and $d_2|_{\mfc[-1]}$ is given by the composite map
\begin{equation} \la{diff2}
\begin{diagram}
\mfc[-1] \cong k[-1] \otimes \mfc &\rTo^{\Delta_{-1} \otimes ]\,\mbox{--}\,[\ } & k[-1] \otimes k[-1] \otimes \wedge^2\mfc \,\cong\, \bSym^2(\mfc[-1])\, \text{.}
\end{diagram}
\end{equation}
Here, $\Delta_{-1}\,:\,k[-1] \rar k[-1] \otimes k[-1]$ takes $1_{k[-1]}$ to $-1_{k[-1]} \otimes 1_{k[-1]}$. When one has a surjection $\mfc \twoheadrightarrow \mfd$
of DG Lie coalgebras, the commutative DG algebra
$ (\bSym(\Ker(\mfc \rar \mfd)[-1]), d_1+d_2)$ comes equipped with a coaction of $\mfd$.
The Chevalley-Eilenberg complex $\C^c(\mfc,\mfd;k)$ of the pair $(\mfc,\mfd)$ is the
DG subalgebra of $\mfd$-invariants of  $(\bSym(\Ker(\mfc \rar \mfd)[-1]), d_1+d_2)$.
Note that if $\h$ is a finite-dimensional DG Lie algebra and $\mfd =\mathfrak{h}^{\ast}$,
\[
\C^c(\mfc,\mfd;k) \,\cong\, (\bSym(\Ker(\mfc \rar \mfd)[-1]), d_1+d_2)^{\mathfrak{h}} \,\text{.}
\]
We denote the homology of the complex $\C^c(\mfc;k)$ (resp., $\,\C^c(\mfc,\mfd;k)$) by $\H_{\bullet}(\mfc;k)$ (resp., $\H_{\bullet}(\mfc,\mfd;k)$).
Note that if $\mfc$ is concentrated in homological degree $0$, then
$$\H_{-1}(\mfc;k) \,\cong\, \Ker(\,]\,\mbox{--}\,[\,:\,\mfc \rar \wedge^2\mfc)\,\text{.}$$
This is dual to the basic fact that $\H_1(\g;k)
\cong \g/[\g,\g]$ for a Lie algebra $\g$ concentrated in degree $0$.

We note that the Chevalley-Eilenberg complex $\CE(\g;k)$ of a DG Lie algebra $\g$ is the analogue of the bar construction of a DG algebra, while the Chevalley-Eilenberg complex $\CE^c(\mfc;k)$ of a DG Lie coalgebra is the analog of the cobar construction of a DG coalgebra (see Section~\ref{sec6.1}). We write $ \CE(\g;k) = \bB_{\mathtt{Lie}}(\g)$ and $ \CE^c(\mfc;k) = \cb_{\mathtt{Lie}}(\mfc)$ when
we want to emphasize these facts.

\subsection{Derived representation schemes and homology of Lie coalgebras}
For a DG coalgebra $M$ and an augmented DG coalgebra $ C \in \DGC_{k/k}\,$, define 
the tensor product $\,\dotimes\, $ by $\, M \,\dotimes\, C := 
k \oplus (M \varotimes \bar{C})\,$.
Note that $\,\overline{M \dotimes C}\,=\,M  \otimes \bar{C}\,$.
Recall from~\eqref{augmatrixfunctor} that $\M_n'(B)$ denotes the unitalization of the nonunital DG algebra $\M_n(\bar{B})$ for $B\,\in\,\cDGA_{k/k}\,$; further, recall from \eqref{root} the (noncommutative) representation functor $\,\sqrt[n]{\,\mbox{--}\,}\,$.
The following key observation relates $\,\sqrt[n]{\,\mbox{--}\,}\,$ to the cobar construction.

\bprop
\la{ncmain}
There is a natural isomorphism of functors from $\DGC_{k/k}$ to $\DGA_{k/k}$
\begin{equation} \la{cobarncrep}  \sqrt[n]{\,\mbox{--}\,} \circ \cb(\mbox{--})\,\cong\, \cb(\M_n^{\ast}(k) \dotimes \,\mbox{--}\,) \end{equation}
In particular, for any $C\,\in\,\DGC_{k/k}$, there is a natural isomorphism of DG algebras
$$ 
\sqrt[n]{\cb(C)}\,\cong\, \cb(\M_n^{\ast}(k) \dotimes C)\,\text{.}
$$
\eprop

\begin{proof}
Indeed, for any $B \in \DGA_{k/k}$, we have
\begin{eqnarray*}
\Hom_{\DGA_{k/k}}(\cb(\M_n(k)^{\ast} \dotimes C), B) & \cong & \Tw(\M_n^{\ast}(k) \dotimes C, B)\\
                                                                               & \stackrel{\mathtt{def}}{=} & \mathbf{MC}(\mathbf{Hom}(\M_n^{\ast}(k) \otimes \bar{C}, \bar{B}))\\
                                 & \cong & \mathbf{MC}(\mathbf{Hom}(\bar{C}, \mathbf{Hom}(\M_n^{\ast}(k), \bar{B}))\\
                                & = & \Tw(C, \mathbf{Hom}(\M_n^{\ast}(k), \bar{B})')\\
                                & = & \Tw(C, \M_n'(B))\\
                                & \cong & \Hom_{\DGA_{k/k}}(\cb(C), \M_n'(B))\\
                                & \cong & \Hom_{\DGA_{k/k}}(\!\sqrt[n]{\cb(C)} , B)
\end{eqnarray*}
The desired proposition now follows from Yoneda's lemma.
\end{proof}

\begin{remark} 
Proposition~\ref{ncmain} implies that $ \sqrt[n]{\,\mbox{--}\,}\, $ is a homotopy functor
on $ \DGA_{k/k} $ in the sense that it preserves weak equivalences: in fact, it induces a functor
on $ \Ho(\DGA_{k/k}) $ given by the formula
$$
\sqrt[n]{\,\mbox{--}\,}\,\cong\, \cb \circ \M_n^{\ast}(\,\mbox{--}\,) \circ \bar{\bB}\ ,
$$
where $ \bar{\bB} $ is the reduced bar construction.
\end{remark}

\vspace{1ex}

The following theorem makes precise the statement made in the Introduction
that representation homology is Koszul dual to Lie algebra homology.
This is one of the key results of the present paper. While it is an easy corollary of Proposition~\ref{ncmain}, we state it as a theorem in order to highlight its importance.
\begin{theorem}
\la{main}
Let $A \in \DGA_{k/k}$ and let $ C \in \DGC_{k/k} $ be a Koszul dual coalgebra of $ A $. Then, for any $n \ge 1$, there are isomorphisms in $\Ho(\cDGA_{k/k})\,$:
\begin{equation}
\la{F1}
 \DRep_n(A) \,\cong \, \C^c(\gl_n^{\ast}(\bar{C});k)\ ,\quad
 \DRep_n(A)^{\GL} \,\cong \, \C^c(\gl_n^{\ast}(C),\,\gl_n^{\ast}(k);k)\,\text{.}
\end{equation}
Consequently,
\begin{equation}\la{F2}
\H_{\bullet}(A,n) \, \cong\, \H_{\bullet}(\gl_n^{\ast}(\bar{C});k)\ ,\quad
\H_{\bullet}(A,n)^{\GL} \, \cong \, \H_{\bullet}(\gl_n^{\ast}(C),\,\gl_n^{\ast}(k);k)\,\text{.}
\end{equation}
\end{theorem}

\begin{proof}
Note that $\,\DRep_n(A) \,\cong \,\cb(C)_n\,=\,[\!\sqrt[n]{\cb(C)}]_{\nn} $.
By Proposition~\ref{ncmain},
$$
[\!\sqrt[n]{\cb(C)}]_{\nn}\,\cong\, \cb(\M_n^{\ast}(k) \dotimes C)_{\nn}\,\text{.}
$$
By Theorem~\ref{variouscbb} (in particular, see~\eqref{isofun3}),
$$
\cb(\M_n(k) \dotimes C)_{\nn}\,\cong\, \C^c({\mathcal Lie}^c(\M_n^{\ast}(k) \dotimes C);k) \,=\,
\C^c(\gl_n^{\ast}(\bar{C});k)\,\text{.}
$$
This proves that $\DRep_n(A)\,\cong\,\C^c(\gl_n^{\ast}(\bar{C});k)$. Now, the $\GL_n(k)$-invariant subalgebra of $\C^c(\gl_n^{\ast}(\bar{C});k)$ is indeed
$\C^c(\gl_n^{\ast}(C),\gl_n^{\ast}(k);k)$. This completes the proof of the desired theorem.
\end{proof}

\bcor
Let $A, C$ be as in Theorem~\ref{main}. Assume, in addition, that $ \dim_k(C) < \infty $, and
let $E\,:=\,C^{\ast}$ be the (linear) dual DG $k$-algebra. Then,
\begin{equation} \la{F8}
\H_{\bullet}(A,n)\,\cong\, \H^{-\bullet}(\gl_n(\bar{E});k)\,,\quad\,\H_{\bullet}(A,n)^{\GL} \, \cong \, \H^{-\bullet}(\gl_n(E),\gl_n(k);k) \ ,
\end{equation}
where $\H^{-\bullet}$ denotes Lie algebra {\rm cohomology} with negative grading.
\ecor

\bproof
If $C$ is finite-dimensional and $E\,:=\,C^{\ast}$, then $\H_{\bullet}(\gl_n^{\ast}(\bar{C});k) \,\cong\, \H^{-\bullet}(\gl_n(E);k)$ by~\eqref{duality2}. Hence, $\H_{\bullet}(\gl_n^{\ast}(C),\gl_n^{\ast}(k);k)\,\cong\,\H^{-\bullet}(\gl_n(E),\gl_n(k);k)$. The result is now immediate from Theorem~\ref{main}.
\eproof

\begin{remark}
Theorem~\ref{main} gives a natural interpretation to the Chevalley-Eilenberg Lie homology
of the matrix Lie coalgebra $ \gl_n^{\ast}(\bar{C}) $. Note that, in the first isomorphisms
in \eqref{F1} and \eqref{F2}, $\, \gl_n^\ast $ is defined
over $ \bar{C} = \mathrm{Coker}(\varepsilon_C)\,$, not $ C $ itself.
It is natural to ask whether the complex $ \C^c(\gl_n^{\ast}(C); k) $ for an arbitrary
(but not necessarily augmented) coalgebra $C$ can be identified with
the $ \DRep_n $ of some DG algebra $A$. The answer
is `yes'\,: the corresponding $A$ is given by the {\it extended} cobar
construction $ \cb^{\rm ext}(C) $ introduced in \cite{AJ}.
Precisely, for any counital coalgebra $C$, Anel and Joyal (see \cite[Sect. 5.3]{AJ})
define $ \cb^{\rm ext}(C) $ as the Sweedler product\footnote{See Appendix,
\eqref{swidpr}, for
the definition of the Sweedler product.} $ C \rhd {\rm MC} \,$ of $ C $
with the Maurer-Cartan algebra $ {\rm MC} $, which is the free DG algebra generated by
one element $ u $ of degree $-1$ with differential $ du = - u^2 $. The argument of
Proposition~\ref{ncmain} combined with results of \cite{AJ} gives a
natural isomorphism of functors $\,\DGC_{k} \to \DGA_{k} \,$:
\begin{equation*}
\sqrt[n]{\,\mbox{--}\,} \circ \cb^{\rm ext}(\mbox{--})\,\cong\,
\cb^{\rm ext}(\M_n^{\ast}(k)\otimes \,\mbox{--}\,) \ ,
\end{equation*}
and the result of Theorem~\ref{main} can thus be extended to arbitrary DG coalgebras as
$$
\C^c(\gl_n^{\ast}(C);k)\,\cong\,\DRep_n[\cb^{\rm ext}(C)]\ .
$$
\end{remark}

\subsubsection{Representation homology of bimodules}
In~\cite[Section 5]{BKR}, we defined representation homology of a DG algebra $A$ with coefficients
in an arbitrary DG bimodule over $A$. Proposition~\ref{VdBprop} below gives an interpretation
of this construction in terms of homology of Lie coalgebras.

Let $\mathtt{Bimod}(A)$ denote the category of DG bimodules over $A$.
For a fixed $ n \ge 1 $, the universal representation $\pi_n\,:\,A \rar \M_n(A_n)$ makes
$\M_n(A_n) $ a DG $A$-bimodule, which has also a natural DG $A_n$-module structure
compatible with the action of $A$. This allows one to define the functor
$$
(\,\mbox{--}\,)^{\mathrm{ab}}_n\,:\,
\mathtt{Bimod}(A) \rar \mathtt{Mod}(A_n)\,,\quad\,
M \mapsto M \otimes_{A^{\mathrm{e}}} \M_n(A_n)\,\text{.}
$$
which we call the van den Bergh functor (as it first appeared in \cite{VdB}).
In \cite[Section 5]{BFR}, it is shown that $\,(\,\mbox{--}\,)^{\mathrm{ab}}_n \,$
is the abelianization of the non-additive representation functor \eqref{rootab}
in the sense of model categories (cf. \cite[Section~II.5]{Q1}), whence its notation.

The derived functor of the van den Bergh functor is computed by the
formula (see \cite[Theorem~5.1]{BKR})
$$
\L (M)^{\mathrm{ab}}_n = [F(R,M)]^{\mathrm{ab}}_n\ ,
$$
where $\,R \stackrel{\sim}{\rar} A\,$ is a cofibrant resolution of $A$ in $\DGA_k$ and
$F(R,M)$ is a semi-free resolution of $M$ in $\mathtt{Bimod}\,R$.
This is independent of the choice of resolutions and allows one to define
the $n$-th representation homology of a DG bimodule $ M \in \Bimod(A) $ by
$$
\H_{\bullet}(M,n)\,:=\, \H_{\bullet}[F(R,M)^{\mathrm{ab}}_n]\ .
$$
Now, assume that $A \,\in\,\DGA_{k/k}$ and let $C\,\in\,\DGC_{k/k}$ be a Koszul dual coalgebra
to $A$.
\bprop \la{VdBprop}
Let $M \in \Bimod(A) $, and let $N$ be a DG $C$-bicomodule such that there is a quasi-isomorphism
$ \cb(C) \otimes_{\tau_C} N \otimes_{\tau_C} \cb(C) \stackrel{\sim}{\rar} M$ in $
\mathtt{Bimod}[\cb(C)] \,$, with $\tau_C\,:\,C \rar \cb(C)$ denoting the universal twisting 
cochain. Then, there is an isomorphism
$$ \H_{\bullet}(M, n)\,\cong\, \H_{\bullet}(\gl_n^{\ast}(\bar{C});\gl_n^{\ast}(N))\,,$$
where $\gl_n^{\ast}(N)$ denotes $\M_n^{\ast}(N)$ viewed as a DG Lie comodule over $\gl_n^{\ast}(\bar{C})$ via the coadjoint coaction.
\eprop

\bproof
Indeed, $\cb(C) \stackrel{\sim}{\rar} A$ is a cofibrant resolution of $A$ in $\DGA_{k/k}$ and $ \cb(C) \otimes_{\tau_C} N \otimes_{\tau_C} \cb(C) \stackrel{\sim}{\rar} M$ is a semifree resolution $M$ in $\mathtt{Bimod}\,\cb(C)$. Hence,
$$\L M_n^{\mathrm{ab}} \,\cong\, [\cb(C) \otimes_{\tau_C} N \otimes_{\tau_C} \cb(C)]^{\mathrm{ab}}_n \,\text{.}$$
Since $\cb(C) \otimes_{\tau_C} N \otimes_{\tau_C} \cb(C) \,\cong\, N \otimes \cb(C)^{e}$ as graded $\cb(C)$-bimodules, $[\cb(C) \otimes_{\tau_C} N \otimes_{\tau_C} \cb(C)]^{\mathrm{ab}}_n$ is generated as a graded $\cb(C)_n$-module by $\M_n^{\ast}(N)$. By Theorem~\ref{main}, $\cb(C)_n\,\cong\,\C^c(\gl_n^{\ast}(\bar{C});k)$. A direct computation then verifies that the differential on elements of $\M_n^{\ast}(N)$ in $[\cb(C) \otimes_{\tau_C} N \otimes_{\tau_C} \cb(C)]^{\mathrm{ab}}_n$ is the sum of a term induced by the coadjoint coaction of $\gl_n^{\ast}(\bar{C})$ on $\M_n^{\ast}(N)$ and another term induced by the differential of $N$. This proves that
$$ [\cb(C) \otimes_{\tau_C} N \otimes_{\tau_C} \cb(C)]^{\mathrm{ab}}_n\,\cong\, \C^c(\gl_n^{\ast}(\bar{C});\gl_n^{\ast}(N))\,\text{.}$$
The desired result is now immediate.
\eproof

\subsection{The Tsygan-Loday-Quillen theorem for coalgebras}
\la{LQTco}
For a coalgebra $ C \in \DGC_{k/k}$, let $\overline{\mathrm{CC}}(C)$ denote
the complex of reduced cyclic chains on $C$ (the definition of $\overline{\mathrm{CC}}(C)$  is formally dual to the definition of the reduced cyclic complex $\overline{\mathrm{CC}}(A)$ of an augmented DG algebra $A\,$, see~\cite[2.1.4 and 2.2.13]{L}).
For each $ n \ge 0 $, there is a generalized cotrace map (cf. \cite[2.2.10]{L}):
$$
\cTr_n(C)\,:\, \overline{\mathrm{CC}}(C) \rar \overline{\mathrm{CC}}
[\M_n^{\ast}(k)\, \dotimes\, C] \ ,
$$
where $\,\M_n^{\ast}(k) := \M_n(k)^* $ is the matrix coalgebra dual to $ \M_n(k) $.
Furthermore, dual to the construction of \cite[10.2.3]{L}, one has a natural morphism of
complexes
$$
\theta\,:\,\overline{\mathrm{CC}}(\M_n^{\ast}(C))[-1] \rar
\C^c(\gl_n^{\ast}(C),\gl_n^{\ast}(k);k)\,,
$$
which is induced by the projection
$$
(\alpha_0,\ldots,\alpha_n) \mapsto \alpha_0 \wedge \ldots \wedge \alpha_n \,\text{.}$$
Abusing notation, we denote the composite map $\,\theta \circ \cTr_n(C)\,$ by $\cTr_n(C)$.
Note that for any $n$, there is a canonical epimorphism in $\cDGA_{k/k}$:
$$
\mu_{n+1,n}\,:\,\C^c(\gl_{n+1}^{\ast}(C), \,\gl_{n+1}^{\ast}(k);k)
\twoheadrightarrow \C^c(\gl_n^{\ast}(C),\,\gl_n^{\ast}(k);k)
$$
Therefore, we can pass to the projective limit
$$
\C^c(\gl_{\infty}^{\ast}(C),\, \gl^{\ast}_{\infty}(k);k) \,:=\, \varprojlim_n \,\C^c(\gl_n^{\ast}(C),\,\gl_n^{\ast}(k);k)\,\text{.}$$
It is easy to check that $\mu_{n+1,n} \circ \cTr_{n+1}(C)\,=\,\cTr_n(C)\,$. Hence we have
a morphism of complexes
$$
\cTr_{\infty}(C)\,:\, \overline{\mathrm{CC}}(C)[-1] \rar
\C^c(\gl_{\infty}^{\ast}(C), \gl^{\ast}_{\infty}(k);k)\text{.}
$$
This extends by multiplicativity to a morphism of DG algebras
\begin{equation}
\la{symcotr}
\bSym(\cTr_{\infty}) :\, \bSym_k(\overline{\mathrm{CC}}(C)[-1]) \rar  \C^c(\gl_{\infty}^{\ast}(C), \gl^{\ast}_{\infty}(k);k)\text{.}
\end{equation}
Now, by a result of Quillen \cite{Q3}, there is a natural isomorphism
$$
\cb(C)_{\n} \,\cong\, \overline{\mathrm{CC}}(C)[-1]\ ,
$$
where $ (\,\mbox{--}\,)_\n $ is the cyclic functor defined by \eqref{cycf}.
On the other hand, by Theorem~\ref{main}, we have an isomorphism $\, \cb(C)_n^{\GL} \,\cong \,
\C^c(\gl_n^{\ast}(C),\,\gl_n^{\ast}(k);k)\,$. The following lemma is verified by an easy computation which we leave to the reader.
\blemma
\la{cotreqtrfin}
For any $n \ge 0 $, the following diagram commutes:
$$
\begin{diagram}[small, tight]
    \cb(C)_{\n} & \rTo_{\ \rm{Quillen}\ }^{\cong} & \overline{\mathrm{CC}}(C)[-1]\\
        \dTo^{\Tr_n}  & & \dTo_{\cTr_n}\\
      \cb(C)_n^{\GL} & \rTo_{\ \rm{Theorem~\ref{main}}\ }^{\cong} & \C^c(\gl_n^{\ast}(C),\gl_n^{\ast}(k);k)
 \end{diagram}
$$
\elemma
As a consequence, we get an isomorphism of (topological) DG algebras
\begin{equation}
\la{topiso}
\C^c(\gl_{\infty}^{\ast}(C), \gl^{\ast}_{\infty}(k);k)\,\cong\,
[\cb(C)]_{\infty}^{\GL_{\infty}}
\end{equation}
\bprop
\la{cotreqtr}
Under \eqref{topiso}, the morphism \eqref{symcotr} is identified with
$\,\bSym\,(\Tr_{\infty}) :\, \cb(C)_{\n} \rar \cb(C)_{\infty}^{\GL_{\infty}}$.
\eprop
Proposition~\ref{cotreqtr} combined with \cite[Lemma~4.1]{BR} implies that
the image of \eqref{symcotr} is a dense DG subalgebra of $\,
\C^c(\gl_{\infty}^{\ast}(C), \gl^{\ast}_{\infty}(k);k)\,$, which following \cite{BR}, we
denote by $\,\C^{\Tr}(\gl_{\infty}^{\ast}(C), \gl^{\ast}_{\infty}(k);k)$.
The main theorem of \cite{BR} (precisely, \cite[Theorem~4.4]{BR}) then
implies the following result, which is a (relative variant) of the Tsygan-Loday-Quillen
theorem for coalgebras.

\begin{theorem} \la{tlq}
For any $C \in \DGC_{k/k} $, there are natural isomorphisms of DG algebras
$$\bSym_k(\overline{\mathrm{CC}}(C)[-1]) \,\cong\, \bSym_k[\cb(C)_{\n}] \,
\cong\, \C^{\Tr}(\gl_{\infty}^{\ast}(C),\, \gl^{\ast}_{\infty}(k);k)\ .
$$
\end{theorem}

\vspace{1ex}

\noindent
{\bf Remarks.}
1. Theorem~\ref{tlq} can be also proven by dualizing the invariant-theoretic
arguments in~\cite[Ch. 10]{L}. With Theorem~\ref{main}, one can then presumably give
a different proof of~\cite[Theorem~4.4]{BR}. This, however, does not seem to answer the
questions of \cite{BR} regarding Koszul duality (see \cite[Question~5.2.1]{BR}).

2. If the DG coalgebra $C$ is bigraded (cf. Section~\ref{bigset} below),
then $\,\C^{\Tr}(\gl_{\infty}^{\ast}(C),\, \gl^{\ast}_{\infty}(k);k)$ coincides with
$\,\C^c(\gl_{\infty}^{\ast}(C),\, \gl^{\ast}_{\infty}(k);k)$. The bigraded dual of
Theorem~\ref{tlq} yields (see~\eqref{duality3}) the Tsygan-Loday-Quillen theorem
for the bigraded DG algebra $A\,:=\, C^*$. Hence, \cite[Proposition~7.5]{BR} implies
a version of the traditional Tsygan-Loday-Quillen theorem.

3. Let $E \in \DGA_{k/k}$ be finite-dimensional. If $C=E^*$, the homology of the complex
$\C^c(\gl_{\infty}^{\ast}(C),\gl_{\infty}^{\ast}(k);k)$ is isomorphic to $\H^{-\bullet} (\gl_{\infty}(E),\gl_{\infty}(k);k)$ by~\eqref{duality2}. By the cohomological version of the 
Tsygan-Loday-Quillen Theorem, the latter is isomorphic to
$\bSym_k [\rHC^{-(\bullet-1)}(E)]\,$. On the other hand, $\overline{\mathrm{CC}}(C)[-1]$ also has homology
$\rHC^{-(\bullet-1)}(E)$. It follows from Theorem~\ref{tlq} that
$$
\C^{\Tr}(\gl_{\infty}^{\ast}(C), \gl^{\ast}_{\infty}(k);k)\,\simeq\, \C^c(\gl_{\infty}^{\ast}(C), \gl^{\ast}_{\infty}(k);k)\text{.}
$$
It would be interesting to know if this is true in general. A nonrelative version of Theorem~\ref{tlq} was proven in~\cite{Kay} using a direct invariant-theoretic approach. It remains to be seen if such an approach can help settle the above question.

\vspace{1ex}

\subsection{Bigraded algebras}
\la{bigset}
Recall  from~\cite[Section 7]{BR}, that a {\it bigraded} DGA is a DG algebra equipped with homological as well as weight (polynomial) grading
that is finite-dimensional in each weight degree\footnote{Bigraded DGAs are also referred to as weight-graded DGAs or WDGAs. See~\cite[Sec. 1.5.11]{LV}.}. We remark that unlike in~\cite{BR}, we allow our DGA to have nonzero components in negative homological degree as well.
The differential is required to obey the graded Leibniz rule with respect to the
homological grading. Additionally, we require that a bigraded DGA be concentrated in nonnegative weight degrees and that
the component in weight degree $0$ be isomorphic to $k$. A bigraded DGA is therefore,  augmented as well.
Following~\cite[Section 7]{BR}, one can verify that the category $\bDGA_k$ of bigraded DGAs is a model category whose weak
equivalences are the quasi-isomorphisms and whose fibrations are the surjective morphisms.
The category $\bcDGA_k$ of bigraded commutative DGAs has an analogous model structure,

Similarly, a coalgebra $C \in \DGC_{k/k}$ will be called bigraded if $\bar{C}$ has a homological as well as weight grading,
is concentrated in positive weight degree and is finite-dimensional in each weight degree.
Note that in this case, $\cb(C)_{\nn}$ is a bigraded commutative DGA. In the bigraded setting,
we have the following result, the first part of which is Theorem~\ref{main} in the bigraded setting.

\begin{theorem} \la{mainbgr}
Let $ A \in \bDGA_k$, and let $ C $ be a Koszul dual coalgebra to $A$. Assume that
$ C $ is bigraded, and the quasi-isomorphism $\cb(C) \stackrel{\sim}{\twoheadrightarrow} A $ is
in $\bDGA_k$. Then,

$(a)$  In  $\,\Ho(\bcDGA_{k/k})$, there are isomorphisms
$$
\DRep_n(A)\,\cong\,\C^c({\gl}_n^{\ast}(\bar{C});k)\,,\,\,\,\,\,\DRep_n(A)^{\GL}\,\cong\,
\C^c(\gl_n^{\ast}(C),\gl_n^{\ast}(k);k)\,\text{.}
$$

$(b)$ If $E\,:=\,C^*$ is bigraded dual of $C$, then
$$ \H_{\bullet}(A,n) \,\cong\, \H^{-\bullet}(\gl_n(\bar{E});k)\,,\,\,\,\,
\H_{\bullet}(A,n)^{\GL}\,\cong\,\H^{-\bullet}(\gl_n(E),\gl_n(k);k)\ ,
$$
where $\,\H^{-\bullet}\,$ denotes (continuous) Lie algebra {\rm cohomology}.
\end{theorem}

\begin{proof}
The proof of Theorem~\ref{main} goes through word for word in the bigraded setting: this gives
$(a)$. Part $(b)$ follows immediately from $(a)$ and~\eqref{duality2}.
\end{proof}

Note that for $A \in \bDGA_k$, at least one $E$ as in Theorem~\ref{mainbgr} exists: namely, the bigraded dual of the bar construction $\boldsymbol{B}(A)$.

\section{Derived Harish-Chandra homomorphism}
\la{sec4}
The classical Harish-Chandra homomorphism is defined by restricting the characters
of representations (viewed as functions on a representation variety)
to the subvariety of diagonal representations. In this section, we construct a
derived version of this homomorphism.

\subsection{Basic construction}
\la{basic_hc}
Recall that, for a fixed
DG algebra $ A \in \DGA_k $ and an integer $ n \ge 1 $, the affine DG scheme
of $n$-dimensional representations of $ A $ is defined by the functor $ \Rep_n(A) $ 
which is represented by the commutative DG algebra $ A_n $ (see Section~\ref{S1.2.1}). 
Now, we introduce the functor of {\it diagonal representations}:
\begin{equation}
\la{diag}
{\rm Diag}_n(A):\, \cDGA_k \to \Sets\ ,\quad B \mapsto \Hom_{\DGA_k}(A,\,B^{\times n})\ ,
\end{equation}
where $\,B^{\times n} \,$ denotes the (direct) product of $n$ copies of $ B $ in the category
of commutative DG algebras. There are natural isomorphisms of sets
$$
\Hom_{\DGA_k}(A,\,B^{\times n})\, \cong\,
\Hom_{\DGA_k}(A,\,B)^{\times n}\, \cong\,
\Hom_{\cDGA_k}(A_{\nn},\,B)^{\times n}\, \cong\,
 \Hom_{\cDGA_k}((A_\nn)^{\otimes n},\,B)\ ,
$$
which show that \eqref{diag} is (co)represented by the commutative DG algebra
$ (A_\nn)^{\otimes n} $, the $n$-th tensor power of the abelianization of $A$.

Next, observe that the functor  $\,{\rm Diag}_n(A)\,$ comes together with a natural
transformation
\begin{equation}
\la{diagrep}
{\rm Diag}_n(A)\,\to \,\Rep_n(A)\ ,
\end{equation}
defined by the algebra maps $\, B^{\times n} \into \M_n(B) \,$
identifying $ B^{\times n} $ with the subalgebra of diagonal matrices in $ \M_n(B) $.
By Yoneda's Lemma, this natural transformation gives a homomorphism of commutative DG algebras
\begin{equation}
\la{HChom}
\Phi_n(A):\, A_n\,\to\, (A_\nn)^{\otimes n}\ ,
\end{equation}
which is obviously functorial in $ A \in \DGA_k $. In terms of the generators of
Lemma~\ref{exppr}, the  homomorphism $ \Phi_n(A) $ is given by
\begin{equation}
\la{HChomf}
a_{ij} \mapsto \delta_{ij}\,a_i\ ,
\end{equation}
where $\, a_i = 1 \otimes \ldots \otimes a \otimes \ldots \otimes 1 \,$ is
the image in $ (A_\nn)^{\otimes n} $ of an element $ a \in A $
sitting in the $i$-th tensor factor.

Now, for a commutative DG algebra $ B \in \cDGA_k $, let $ S^n(B) $ denote the $n$-th
symmetric power of $B$, i.e. $\, S^n(B) := [B^{\otimes n}]^{S_n} \,$. By definition,
$ S^n(B) $ is a subalgebra of $ B^{\otimes n} \,$;
since $ k $ has characteristic $0$, we may identify $ S^n(B) $ with
the image of the symmetrization map $\, \bSym^n(B) \into B^{\otimes n} \,$.
\blemma
\la{resthom}
The map \eqref{HChom} restricts to a homomorphism of commutative DG algebras
\begin{equation}
\la{HChom1}
\Phi_n(A):\, A^{\GL}_n\,\to\, S^n(A_\nn)\ .
\end{equation}
\elemma
\bproof
Consider the (right) action of $ S_n $ on $ \M_n(k) $ by permuting
the columns of matrices: i.e., if $\,M = \|M_1, \ldots, M_n\| \in \M_n(k) \,$
then $ M^{\sigma} := \|M_{\sigma^{-1}(1)}, \ldots, M_{\sigma^{-1}(n)}\|\,$ for
$ \sigma \in S_n $. Define $\,\iota: S_n \to \GL_n(k) \,$ by $ \iota(\sigma) = I_n^{\sigma} $,
where $ I_n $ is the identity matrix in $ \M_n(k) $. Then, for any $ M \in \M_n(k) $,
we have $\,M^\sigma = M\,\iota(\sigma)\,$. Hence $ \iota $ is an injective group homomorphism.
We claim that \eqref{HChom} is an $ S_n $-equivariant map provided $ S_n $ acts
on $ A_n $ by restricting the canonical $\GL_n$-action via $ \iota $.
To see this recall that the $\GL_n $-action on $A_n$ comes from the adjoint action
on $ \M_n(k) $, see \eqref{GLact}. Hence it suffices to check that the
following diagram commutes for all $ \sigma \in S_n \,$:
\begin{equation}
\la{AdS}
\begin{diagram}[small, tight]
\M_n(A_n)           &   \rTo^{\id \otimes \Phi_n}  &   \M_n[(A_\nn)^{\otimes n}] \\
\dTo^{\Ad_{\iota(\sigma)}} &                               & \dTo_{\id \otimes \sigma } \\
\M_n(A_n)           &   \rTo^{\id \otimes \Phi_n}  &   \M_n[(A_\nn)^{\otimes n}]
\end{diagram}
\end{equation}
In terms of generators of $A_n $, the left vertical map in \eqref{AdS} is given by
$$
\Ad_{\iota(\sigma)}\,\| a_{ij}\| =
\iota(\sigma)\,\| a_{ij}\|\,\iota(\sigma)^{-1} = \iota(\sigma)\,\| a_{ij}\|\,\iota(\sigma^{-1}) =
I_n^{\sigma}\,\| a_{ij}\|\,I_n^{\sigma^{-1}} = \|a_{\sigma(i)\sigma(j)}\|
$$
On the other hand, by \eqref{HChomf}, the horizontal maps are
$\, \|a_{ij}\| \mapsto \diag(a_1, a_2, \ldots, a_n) \,$. The commutativity
of \eqref{AdS} is therefore obvious. Since $ \Phi_n(A) $ is an equivariant map, it takes invariants to invariants.
\eproof

The functor $\,S^n(\,\mbox{--}\,)_\nn: \DGA_k \to \cDGA_k \,$ is not a homotopy functor in the 
sense that it does not preserve weak equivalences and hence does not descend to homotopy categories. However, just as the representation functor \eqref{rootab}, $\,S^n(\,\mbox{--}\,)_\nn\,$ admits a left derived functor that gives a functorial approximation to the induced functor at the level of homotopy categories.
\bprop
\la{derfu}
$(a)$ The functor $ S^n(\,\mbox{--}\,)_\nn $ has a total left derived functor given by
$$
\L S^n(\,\mbox{--}\,)_\nn :\, \Ho(\DGA_k) \to \Ho(\cDGA_k)\ , \quad
A \mapsto S^n[(QA)_\nn]\ ,
$$
where $ QA $ is a cofibrant resolution of $A$.

$(b)$ For any $ A \in \DGA_k $, there is a natural isomorphism of graded commutative algebras
$$
\H_\bullet[\L S^n(A)_\nn] \,\cong\, S^n[\H_\bullet(A,\,1)]\ .
$$
\eprop
\bproof
$(a)$ The abelianization $\, (\,\mbox{--}\,)_\nn :\, \DGA_k \to \cDGA_k $
is a left Quillen functor adjoint to the inclusion $\,\cDGA_k \into \DGA_k \,$.
Hence it maps acyclic cofibrations in $ \DGA_k $ to weak equivalences. On the other hand,
the functor $ S^n $ commutes with taking homology
(see, e.g., \cite[Appendix~B, Prop.~2.1]{Q2}), hence it maps weak equivalences
(= quasi-isomorphisms) in $\,\cDGA_k $ to weak equivalences. It follows that
$\,S^n(\,\mbox{--}\,)_\nn = S^n \circ (\,\mbox{--}\,)_\nn\,$ maps acyclic
cofibrations to weak equivalences. Hence $\,S^n(\,\mbox{--}\,)_\nn $ has a total left derived functor by Brown's Lemma (see, e.g., \cite[Lemma~A.2]{BKR}).

$(b)$ follows from the obvious isomorphism $\,(\,\mbox{--}\,)_\nn \cong (\,\mbox{--}\,)_{1}\,$
and the fact that $\,S^n \circ \H_{\bullet} \cong  \H_{\bullet} \circ S^n \,$.
\eproof

\noindent
{\bf Notation.} We will write the derived functor $\, \L S^n(\,\mbox{--}\,)_\nn  \,$ as
$\, S^n\DRep_1(\,\mbox{--}\,)\,$. This is justified by Proposition~\ref{derfu}: indeed,
for $ A \in \DGA_k $, we have $\, \DRep_1(A) \cong \L(A)_\nn \,$, so
$ S^n\DRep_1(A) $ is to be computed by the same formula as $ \L S^n(A)_\nn  $.

\vspace{1ex}

Now, by Lemma~\ref{resthom}, we have a natural transformation of functors
$\, \Phi_n: (\,\mbox{--}\,)_n^{\GL} \to S^n(\,\mbox{--}\,)_\nn \,$.
By Theorem~\ref{nS2.1t5} and Proposition~\ref{derfu}, both $ (\,\mbox{--}\,)_n^{\GL} $ and $ S^n(\,\mbox{--}\,)_\nn $ have left derived functors. Hence, $\, \Phi_n \,$ induces a
(unique) natural transformation at the level of homotopy categories:
\begin{equation}
\la{HCres}
\L\Phi_n :\, \L(\,\mbox{--}\,)_n^{\GL}\, \to\, \L S^n(\,\mbox{--}\,)_\nn\ .
\end{equation}
For a fixed $A \in \DGA_k $, this gives a canonical morphism in $ \Ho(\cDGA_k)\,$:
\begin{equation}
\la{ehc3}
\Phi_n(A):\ \DRep_n(A)^\GL \to S^n\DRep_1(A)\ ,
\end{equation}
which we call the {\it derived Harish-Chandra homomorphism}\footnote{This terminology will be
justified in the next section where we reinterpret the construction of $ \Phi_n(A) $
in Lie algebraic terms.}.

Next, we introduce natural maps combining \eqref{ehc3} with
the trace maps constructed in Section~\ref{trmaps}. To this end, we note that the functors $ (\,\mbox{--}\,)_n^{\GL} $, $\, S^n(\,\mbox{--}\,)_\nn $ and the natural transformation $ \Phi_n $ are well-defined on the category of augmented DG algebras ({\it cf.} Theorem~\ref{nS2.1t2}). Hence we can regard \eqref{HCres} as a morphism of functors from $ \Ho(\DGA_{k/k}) $ to $ \Ho(\cDGA_{k/k}) $.
Composing this morphism with the trace \eqref{trm2}, we get a morphism of functors
from $ \Ho(\DGA_{k/k}) $ to $ \Ho(\Com_k)$:
\begin{equation}
\la{redtr}
\L \TTr_n \,:\ \L(\,\mbox{--}\,)_\n \, \to \,  \L S^n(\,\mbox{--}\,)_\nn
\end{equation}
The next proposition shows that $ \L\TTr_n $ is essentially determined by $ \L\TTr_1 $.
\bprop
\la{prtrmain}
For each $ n \ge 1 $, the morphism $ \L\TTr_n $ factors as
$$
\L\TTr_n \cong S^n \circ \L\TTr_1 \ ,
$$
where $ S^n $ is the (additive) symmetrization morphism.
\eprop
\bproof
It suffices to check this on cofibrant objects. If $ A \in \DGA_{k/k} $
is cofibrant, then $ \L\TTr_n(A) = \TTr_n(A) $, and by Lemma~\ref{exppr}, the
composite map  $\, \TTr_n(A):\, A_\n \to A_n^{\GL} \to S^n(A_\nn) \,$
is given  by
$$\,
\bar{a} \mapsto \sum_{i=1}^n a_{ii} \mapsto \sum_{i=1}^n a_i \ .
$$
where $ a \in A $, $\,\bar{a} \in A_\n\,$ and $\,
a_i = 1 \otimes \ldots \otimes a \otimes \ldots 1 \in  (A_\nn)^{\otimes n}\,$
(with $ a \in A_\nn $ in the $i$-th position). Hence $ \TTr_n(A) $
factors through the canonical map $ A_\n \to A_\nn $ which is nothing but
$ \TTr_1(A) $. Note that the induced map is given by symmetrization:
$\, S^n: A_\nn \to S^n(A)_\nn\,$,$\ a \mapsto \sum_{i=1}^n a_i \,$.
\eproof

For any $A \in \DGA_{k/k} $, the morphism \eqref{redtr} induces on homology natural maps
\begin{equation}
\la{redttr}
\TTr_n(A)_{\bullet}:\, \rHC_\bullet(A) \to [\H_\bullet(A, \,1)^{\otimes n}]^{S_n}\ .
\end{equation}
We call \eqref{redttr} the {\it reduced} trace maps. By Proposition~\ref{prtrmain},
it suffices to compute these maps for $ n = 1 $. We will give an explicit
combinatorial formula for  $ \TTr_1(A)_{\bullet} $ in our subsequent paper \cite{BFPRW}.

\subsection{Interpretation in terms of Lie (co)homology}
We now reinterpret the derived Harish-Chandra homomorphism \eqref{ehc3} in Koszul dual terms of Lie coalgebras. Let $\h_n(k) $ denote the Cartan subalgebra of $\gl_n(k) $, which is the subalgebra
$ \D_n(k)  $ of diagonal matrices in $\M_n(k)$ viewed as a Lie algebra. Dually, one has a surjection of coalgebras $\M_n^{\ast}(k) \twoheadrightarrow \D_n^{\ast}(k)$ where $ \D_n^{\ast}(k)$ denotes the $k$-linear dual of $ \D_n(k)$.
Hence, for any $ C \,\in\, \DGC_{k/k}$, one has a morphism $\M_n^{\ast}(C) \twoheadrightarrow \D_n^{\ast}(C)$ of DG coalgebras that induces the
natural map $\M_n^{\ast}(k)  \twoheadrightarrow \D_n^{\ast}(k)$. Viewing $\M_n^{\ast}(C)$ and $ \D_n^{\ast}(C)$ as DG Lie coalgebras, one
obtains a morphism $\gl_n^{\ast}(C) \rar \h_n^{\ast}(C)$ of DG Lie coalgebras that induces the natural morphism $\gl_n^{\ast}(k) \rar \h_n^{\ast}(k)$ of Lie coalgebras.
As a result, by functoriality, we get a morphism of commutative DG algebras
\begin{equation}
\la{ehc1}
\Phi_n(C)\,:\ \C^c(\gl_n^{\ast}(C),\, \gl_n^{\ast}(k);k) \rar
\C^c(\h_n^{\ast}(C),\,\h_n^{\ast}(k);k)\ \text{.}
\end{equation}

Now, let $ A = \cb(C) $ be the cobar construction of $ C $, which is a cofibrant
DG algebra in $ \DGA_{k/k} $.  Since $\h_n$ is abelian and $\dim\,\h_n = n\, $,
we have canonical isomorphisms in $ \cDGA_{k/k} $:
\begin{equation}
\la{ehc2}
\C^c(\h_n^{\ast}({C}),\,\h_n^{\ast}(k);k) \,\cong\,
[\,\C^c(\gl_1^{\ast}(\bar{C})\,]^{\otimes n}\,\cong \, (A_{\nn})^{\otimes n}\ .
\end{equation}
On the other hand, by Theorem~\ref{main},
\begin{equation}
\la{ehc11}
\C^c(\gl_n^{\ast}(C),\, \gl_n^{\ast}(k); k) \cong  A_n^{\GL}\ .
\end{equation}

With these isomorphisms we can compare \eqref{ehc1} with the Harish-Chandra homomorphism \eqref{HChom1}. The following proposition is a direct consequence of Lemma~\ref{resthom}.
\bprop
\la{phc1}
With the identifications \eqref{ehc2} and \eqref{ehc11}, the map $ \Phi_n(C) $
agrees with  $ \Phi_n(A)$. The image of $ \Phi_n(C) $
is contained in
$\,\C^c(\h_n^{\ast}(C),\,\h_n^{\ast}(k);k)^{S_n}\,$, where the action
of $ S_n $ on the Chevalley-Eilenberg complex comes from the natural
action on $ \h_n $.
\eprop
In Section~\ref{sec7}, we will construct the derived Harish-Chandra homomorphism
for an arbitrary reductive Lie algebra $ \g \,$, generalizing the above construction
for $ \gl_n $. Proposition~\ref{phc1} then shows that the Harish-Chandra homomorphism for associative algebras is a special case of the one for Lie algebras.

\subsection{Symmetric algebras}
\la{sdhc1}

In this subsection, we take $A$ to be the symmetric algebra $ \Sym(V) $, where $V$ is a vector space
of dimension $ r \ge 1 $ concentrated in homological degree $0$.
This is a quadratic Koszul algebra defined by the quadratic data $(V,S) $ with
$ S \subset V\otimes V $ spanned by the vectors of the form
$v\otimes u - u \otimes v $ (cf.\, \cite[Example 3.2.4.2]{LV}).

\subsubsection{The minimal resolution} Recall that $A$ has a canonical (minimal)
semi-free resolution $ R = \Omega(A^\dual) $ given by the cobar construction of the Koszul
dual coalgebra $ A^\dual = C(sV, s^2S) $ (see \cite[Sect. 3.2.1]{LV}).
The algebra $ R $ is the tensor algebra generated by the vector space
$ s^{-1} \Lambda(sV) $, whose elements of degree $ k-1 $ we denote by
\begin{equation}
\la{lambda}
\lambda(v_1, v_2, \ldots, v_k) :=  s^{-1}(sv_1 \wedge sv_2 \wedge \ldots \wedge sv_k)
\,\in\, s^{-1} \wedge^k(sV) \ .
\end{equation}
With this notation, the differential on $ R $ satisfies
\begin{eqnarray}
&& d \lambda(v_1, v_2) = - [v_1, \,v_2] \ , \la{d1} \\
&& d \lambda(v_1, v_2, v_3) = - [v_1, \lambda(v_2, v_3)] -  [v_2, \lambda(v_3, v_1)] -
[v_3, \lambda(v_1, v_2)]\ . \la{d2}
\end{eqnarray}
In general, we have
\blemma
\la{diff_res}
The differential $d$ on the minimal resolution $R$ of $A=\Sym(V)$ is defined by
\begin{equation*}
d\lambda(v_1,\ldots,v_n)=\sum\limits_{\substack{p+q=n\\1\leq p\leq q}}(-1)^p\sum\limits_{\sigma\in\Sh(p,q)}
(-1)^{\sigma}\left[\lambda(v_{\sigma(1)},\ldots,v_{\sigma(p)}),\,\lambda(v_{\sigma(p+1)},\ldots,v_{\sigma(p+q)})\right]\ ,
\end{equation*}
where $\Sh(p,q)$ denotes the set of $(p,q)$-shuffles, i.e. $\sigma\in S_{p+q}$ of the form $\sigma=(i_1,\ldots,i_p;i_{p+1},\ldots,i_{p+q})$ with $i_1<\ldots<i_p$ and $i_{p+1}<\ldots<i_{p+q}$.
\elemma
\bproof
The proof is by direct computation using the definition of the differential on the cobar construction. We leave the details of this computation to the interested reader.
\eproof

It follows from Lemma~\ref{diff_res} that $\, R_{\ab} $ is isomorphic to
is the graded symmetric algebra of $ s^{-1}\Lambda(sV) $ equipped with zero differential.
Explicitly (omitting the shifts), we can write
\begin{equation}
\la{resab}
R_{\ab} =  \bSym(\wedge^1 V \oplus \wedge^2 V \oplus \wedge^3 V \oplus \ldots \oplus \wedge^r V) = \Sym(V) \otimes \bSym(\wedge^2 V \oplus \wedge^3 V \oplus \ldots \oplus \wedge^r V)
\end{equation}
with understanding that the elements of $\wedge^{k}V$ have (homological) degree $ k-1 $.

%
%

\subsubsection{Main conjecture for $\gl_n $}
For $A=\Sym(V)$, the homomorphism \eqref{ehc3} is given by
\begin{equation*}
\Phi_n :\,\DRep_n(A)^{\GL} \,\rar\, \bSym[\h_n^{\ast} \otimes
(\wedge^1V \oplus \wedge^2 V \oplus \wedge^3 V \oplus \ldots \oplus \wedge^r V)]^{S_n}\ ,
\end{equation*}
where  $\DRep_n(A)\,\cong\,\bSym[\M_n^{\ast}(k) \otimes  (\wedge^1V \oplus \wedge^2 V \oplus \wedge^3 V \oplus \ldots \oplus \wedge^r V)]$ as a graded commutative algebra. Explicitly, $\Phi_n$ is the restriction to $\DRep_n(A)^{\GL}$ of the map of commutative DG-algebras taking
each generator of the form $e^{ij} \otimes \alpha$ to $\delta_{ij}e^i \otimes \alpha$. Here, the $e^{ij}$'s are a basis of $\M_n^{\ast}(k)$ dual to the elementary matrices and the $e^i$'s form the basis of $\h_n^{\ast}$ dual to the standard basis and $\alpha$ is any element in $\wedge^1V \oplus \wedge^2 V \oplus \wedge^3 V \oplus \ldots \oplus \wedge^r V$.

On the $0$-th homologies, $\, \Phi_n $ induces an algebra map
$$
\H_0(\Phi_n):\,k[\Rep_n(A)]^{\GL} \,\rar\, \Sym(\h_n^{\ast} \otimes V)^{S_n} \ .
$$
By \cite[Theorem~4.1]{Do} (see also \cite[Theorem~3]{Va}), this map is known to be
an isomorphism for all $ V $ and all $ n \ge 1 $. It is therefore tempting to conjecture 
that $ \Phi_n $ is actually a quasi-isomorphism, i.e. induces isomorphisms on homology 
in {\it all}\, homological degrees. This is indeed the case when $ \dim(V) = 1 $ (since
$ \Sym(V) $ is a cofibrant DG algebra when $ \dim(V) = 1 \,$, and hence $ \DRep_n(A) \cong A_n $
has no higher homology). On the other hand, by evaluating Euler characteristics,
we will show in Section~\ref{otherex} that $ \Phi_n $ cannot be a quasi-isomorphism (for all $n$)
when $ \dim(V) \ge 3 $. In the case $ \dim(V) = 2 $, we believe that the conjecture
is still true.

To state our conjecture in more explicit terms, we choose a basis $\{x,y\}$ in $V$ and identify $ \Sym(V) = k[x,y]\,$. Denote by $ \theta $ the degree 1 element $ \lambda(x,y) \in \wedge^2 V $
and write $x_i$ (resp., $y_i$, $ \theta_i$) for the elements $ e^i \otimes x $
(resp., $e^i \otimes y, e^i \otimes \theta $) in $\, \h_n^{\ast}
\otimes (\wedge^1 V \oplus \wedge^2 V) $. Then
$\, \bSym[\h_n^{\ast} \otimes
(\wedge^1V \oplus \wedge^2 V)] \cong k[x_1,\ldots ,x_n,y_1, \ldots ,y_n,\theta_1, \ldots ,\theta_n]\,$, and we have

\begin{conjecture}
\la{conj1}
For any $ n \ge 1 $, the map
\begin{equation}
\la{phiV}
\Phi_n\,:\,\DRep_n(k[x,y])^{\GL} \,\rar\,k[x_1,\ldots ,x_n,y_1, \ldots ,y_n,\theta_1, \ldots ,\theta_n]^{S_n}
\end{equation}
is a quasi-isomorphism. Consequently, there is an isomorphism of graded commutative algebras
\begin{equation*}
\la{phind}
\H_\bullet(k[x,y], n)^{\GL}
\cong k[x_1, \ldots ,x_n,y_1, \ldots,y_n,\theta_1, \ldots ,\theta_n]^{S_n}\ .
\end{equation*}
\end{conjecture}

\vspace{1ex}

\begin{remark}
When $ \dim(V) = 2 $, the map $ \H_0(\Phi_n) $ coincides with the usual Harish-Chandra homomorphism for $ \gl_n $ as defined, for example, in \cite{J}; it is therefore natural to call $\Phi_n $ the derived Harish-Chandra homomorphism.
\end{remark}

\vspace{1ex}

As a first evidence for Conjecture~\ref{conj1}, we recall a vanishing theorem
for the representation homology of $ A = k[x,y] $ proved in \cite{BFR}
(see {\it op. cit.}, Theorem~27): for all $ n \ge 1 \,$,
$$
\H_i(k[x,y],\, n) = 0\ , \quad \forall\, i > n \ .
$$
This implies that $\, \H_i(k[x,y],\, n)^{\GL} = 0 \,$ for $ i > n $, and
hence $\, \Phi_n $ induce isomophisms (which are actually the zero maps)
in homological degrees $ i > n $.

For $ n = 1 $, Conjecture~\ref{conj1} is obvious. Furthermore, we have
\begin{theorem}
\la{conj1n2}
Conjecture~\ref{conj1} is true for $n=2$.
\end{theorem}
\begin{proof}
We will explicitly construct the inverse map to $ \Phi $ at the level of homology.
To simplify notation we set $\, A = k[x,y] \,$ and, using \cite[Theorem~2.5]{BKR},
identify $\, \H_0(A,2) = A_2 \,$, where $ A_2 $ is the coordinate ring of the
commuting scheme of $2\times 2$ matrices:
\[
A_2 = k[x_{11},x_{12},x_{21},x_{22},y_{11},y_{12},y_{21},y_{22}]/I\text{,}
\]
with $I$ generated by the relations
\begin{eqnarray}
\label{relations_xy}
&& x_{12}y_{21}-y_{12}x_{21} = 0\nonumber\\
&& x_{11}y_{12}+x_{12}y_{22}-y_{11}x_{12}-y_{12}x_{22}= 0\\
&& x_{21}y_{11}+x_{22}y_{21}-y_{21}x_{11}-y_{22}x_{21}= 0\nonumber
\end{eqnarray}

With this identification, we define an algebra homomorphism
\begin{equation*}
\Psi_0:\, k[x_1,x_2,y_1,y_2] \to \M_2(A_2)
\end{equation*}
by
\[
x_1 \mapsto X\,,\quad x_2\mapsto X^*\,,\quad y_1\mapsto Y\, ,\quad y_2\mapsto Y^*\ ,
\]
where
\[
X := \begin{pmatrix}
x_{11} & x_{12}\\
x_{21} & x_{22}
\end{pmatrix}
\quad ,\quad
X^* := \begin{pmatrix}
x_{22} & -x_{12}\\
-x_{21} & x_{11}
\end{pmatrix}
\]
and similarly for $Y$ and $Y^*$. Note that $X^*$ and $Y^*$ are the classical adjoints
of $X$ and $Y$ and the relations~\eqref{relations_xy} imply that $[X,Y]=0$. It follows
that the matrices $\,X,\,Y,\,X^*,\,Y^*\,$ pairwise commute, and the map $ \Psi_0 $ is thus
well-defined.

We claim that $ \Psi_0 $ restricts to an algebra \emph{isomorphism}
\begin{equation}
\la{psi0}
\Psi_0:\, k[x_1,x_2,y_1,y_2]^{S_2}\stackrel{\sim}{\to} A_2^{\GL_2}\ ,
\end{equation}
where $ A_2^{\GL} \subset A_2 $ is identified with a subalgebra of scalar matrices in
$ \M_2(A_2) $.
Indeed, the invariant subalgebra $k[x_1,x_2,y_1,y_2]^{S_2}$ is generated by the five
elements: $\,x_1+x_2$, $\,y_1+y_2$, $\,x_1x_2$, $\,y_1y_2\,$ and $\,x_1y_1+x_2y_2\,$,
which are mapped by \eqref{psi0} to
$ \Tr(X) $, $\, \Tr(Y) $, $\, \det(X) $, $\, \det(Y)$ and $\, \Tr(XY) $, respectively.
It is immediate to see that $ \H_0(\Phi)\circ \Psi_0 = \id $. On the other hand, as mentioned
above, the map $\, \H_0(\Phi) \,$ is known to be an isomorphism\footnote{In fact, the map
$ \H_0(\Phi) $ coincides with the isomorphism $ \Delta' $ constructed in \cite[Theorem~3]{Va}.}.
The map $ \Psi_0 $ is thus the inverse of $ \H_0(\Phi) $, and hence an isomorphism as well.
We will use the following notation for the generators of $\,k[x_1,x_2,y_1,y_2]^{S_2}\,$:
$$
\T(x) := x_1+x_2\,,\quad \T(y):= y_1+y_2\,,\quad \d(x) :=x_1x_2\,,\quad
\d(y):= y_1y_2\,,\quad \T(xy) := x_1y_1+x_2y_2 \ .
$$

Now, it is easy to check that the graded algebra
$ k[x_1,x_2,y_1,y_2,\theta_1,\theta_2]^{S_2} $ is generated by its degree zero subalgebra
$ k[x_1,x_2,y_1,y_2]^{S_2} $ and three extra elements of degree $1$, which we denote by
$$
\T(\theta) := \theta_1+\theta_2\,,\quad \T(x\theta) := x_1\theta_1+x_2\theta_2\,,\quad
\T(y\theta) := y_1\theta_1+y_2\theta_2\ .
$$
These elements satisfy the following relations
\begin{eqnarray*}
\label{relations_ABC}
&& \!\!\!\!\!\!\!\!\!
\T(y)\cdot \T(x\theta)\cdot \T(\theta) -
\T(x)\cdot \T(y\theta)\cdot \T(\theta) = 2\, \T(x\theta)\cdot \T(y\theta) \\
&&\!\!\!\!\!\!\! \!\!
(\T(x)^2-4 \d(x))\cdot \T(y\theta) -
(2\T(xy)-\T(x)\T(y))\cdot \T(x\theta) =
(\T(x)^2 \T(y)-2 \d(x) \T(y)-\T(x) \T(xy))\cdot \T(\theta) \nonumber\\
&& \!\!\!\!\!\!\!\!\!
(2\T(xy)-\T(x)\T(y))\cdot \T(y\theta) -
(\T(y)^2-4\d(y))\cdot \T(x\theta) = (\T(y)^2\T(x)-2\d(y)\T(x)-\T(y)\T(xy))\cdot
\T(\theta)\nonumber
\end{eqnarray*}
On the other hand, by \cite{BKR}[Example~4.2],
the full representation homology algebra $ \H_\bullet(A,2) $ is
generated by $ \H_0(A,2) = A_2 $ and three invariant elements
$\,\tau,\,\xi,\,\eta\,$ of degree $1$, satisfying
\begin{eqnarray*}
\label{relation_theta_eta_xi}
&& x_{12}\eta-y_{12}\xi = (x_{12}y_{11}-y_{12}x_{11})\tau \nonumber\\
&& x_{21}\eta-y_{21}\xi = (x_{21}y_{22}-y_{21}x_{22})\tau\\
&& (x_{11}-x_{22})\eta - (y_{11}-y_{22})\xi = (x_{11}y_{22}-y_{11}x_{22})\tau \nonumber\\
&& \xi\eta=y_{11}(\xi\tau)-x_{11}(\eta\tau) = y_{22}(\xi \tau)-x_{22}(\eta\tau)\nonumber
\end{eqnarray*}

We now extend the map \eqref{psi0} to the map
\[
\Psi_{\bullet}:\, k[x_1,x_2,y_1,y_2,\theta_1,\theta_2]^{S_2}\to \H_\bullet(A,2)
\]
by sending
$$
\T(\theta) \mapsto \tau \ ,\quad
\T(x\theta) \mapsto \xi \ ,\quad
\T(y\theta) \mapsto \eta \ .
$$

A straightforward (but tedious) calculation, using the above relations, shows
that $ \Psi_\bullet $ is a well-defined algebra map. Its image coincides with
$\, \H_\bullet(A,2)^{\GL} \,$, since $ \H_\bullet(A,2) $ is generated over
$ \H_0(A,2) $ by $\GL$-invariant elements, and we already know that $ \Psi_0 $ is
an isomorphism onto $ \H_0(A,2)^{\GL} $. On the other hand, it is again easy to
see that $ \H_\bullet(\Phi) \circ \Psi_\bullet = \id $. Hence $ \Psi_\bullet $ is injective and
is thus the inverse of $  \H_\bullet(\Phi) $.
\end{proof}

\subsection{Conjecture~\ref{conj1} in the limit}
In this section, we prove that Conjecture~\ref{conj1} holds in the limit $n \rar \infty$.
\subsubsection{Supersymmetric polynomials and power sums}
Let $\,x_1,\ldots,\,x_d\,$ be variables of homological degree $0$. Let the symmetric group $S_d$ act on the $x_i$'s by permutations. Consider the power sums $P_i\,:=\,\sum_{j=1}^d x_j^i$. It is a classical fact that the symmetric polynomials in $x_1,\ldots,x_d$ can be rewritten as polynomials of degree $\leq d$ in the power sums $P_i\,,\,i \geq 1$. Further, equip the variables $x_1,\ldots,x_d$ with weight $1$, making $k[x_1,\ldots,x_d]$ a weight graded (commutative) algebra. Let $k[x_1,\ldots,x_d,\ldots]^{S_{\infty}}\,:=\,\varprojlim_d k[x_1,\ldots,x_d]^{S_d}$ where the projective limit is taken in the category of weight graded algebras. Then, the homomorphism
$$k[q_1,\ldots,q_d,\ldots] \rar k[x_1,\ldots,x_d,\ldots]^{S_{\infty}}\,,\,\,\,\, q_i \mapsto \sum_{j=1}^{\infty} x_j^i $$
is an isomorphism.

We now generalize this classical fact. Consider the set of $\, N d \,$ variables $\{x_{\alpha, i}\,,\,\,1 \leq \alpha \leq N\,,\,1 \leq i \leq d\}$. Assume that the variables $x_{1,i},\ldots,x_{m,i}\,,\,1 \leq i \leq d$ have odd homological degree, with the remaining variables having even homological degree. In particular, in $k[x_{\alpha, i}\,|\,1 \leq \alpha \leq N\,,\,1 \leq i \leq d]$, we have $x_{\alpha, i}^2=0$ for $\alpha \leq m $. Let $S_d$ act on these variables, with a permutation $\sigma \in S_d$ taking $x_{\alpha,i}$ to $x_{\alpha, \sigma(i)}$. Let
$k[x_{\alpha, i}\,|\, \alpha \leq N \,,\,i \in \mathbb N]^{S_{\infty}}\,:=\,\varprojlim_d k[x_{\alpha,i}\,|\,1 \leq \alpha \leq N\,,\,1 \leq i \leq d]^{S_d}$ where the projective limit is taken in the category of bigraded DG algebras and the variables $x_{\alpha,i}$ have positive weight $d_{\alpha}$. Consider the power sums $P_{\mathbf{a}}\,:=\, \sum_{i \geq 1} \prod_{\alpha =1}^N x_{\alpha,i}^{a_{\alpha}}$ where $\mathbf{a}\,:=\,(a_1,\ldots, a_N)$ runs over $\{0,1\}^m \times \Z_{\geq 0}^{N-m}$.

\bprop
\la{sympowsum}
Let $V$ be a $k$-vector space generated by variables
$ \{q_{\mathbf{a}}\} $ where $\, \mathbf{a} \in \{0,1\}^m \times \Z_{\geq 0}^{N-m}$.

{\rm (i)} The homomorphism of bigraded (commutative) DG algebras
 \begin{equation}
\la{powsum} \bSym_k V \rar  k[x_{\alpha,i}\,|\,1 \leq \alpha \leq N\,,\,1 \leq i \leq d]^{S_d}\,,\,\,\,\, q_{\mathbf{a}} \mapsto P_{\mathbf{a}}\,\text{.}
\end{equation}
restricts to an isomorphism of bigraded $k$-vector spaces
$$
\bSym^{\leq d}_kV \,\cong\,  k[x_{\alpha,i}\,|\,1 \leq \alpha \leq N\,,\,1 \leq i \leq d]^{S_d}
$$
{\rm (ii)} The homomorphisms~\eqref{powsum} induce an isomorphism of bigraded (commutative) DG algebras
$$\bSym_k V \,\cong\, k[x_{\alpha, i}\,|\, \alpha \leq N \,,\,i \in \mathbb N]^{S_{\infty}}\,\text{.}$$
\eprop

\begin{proof}
Clearly, (i) implies (ii). We therefore, show (i). Any element of $k[x_{\alpha,i}\,|\,1 \leq \alpha \leq N\,,\,1 \leq i \leq d]^{S_d}$ is a $k$-linear combination of orbit sums of the form
$$O_{\bf{\mu}}\,:=\,\sum_{\sigma \in S_d} \sigma \left(\prod_{\alpha =1}^N \prod_{j=1}^{p} x_{\alpha,j}^{\mu_{\alpha,j}} \right)
$$ where $p \leq d$ and no column of the $N \times p$-matrix $\mathbf{\mu}:=(\mu_{\alpha,j})\,,\,\,1 \leq \alpha \leq N, 1 \leq j \leq p$ is identically zero. If $p=1$, $O_{\mathbf{\mu}}$ is precisely the power sum $P_{\mathbf{a}}$ where $\mathbf{a}=(\mu_{\alpha,1},\ldots,\mu_{\alpha,N})$.
Note that
$$O_{\mathbf{\mu}} \,=\, \pm \prod_{j=1}^p (P_{(\mu_{1,j},\ldots,\mu_{N,j})}) + \sum_{\mathbf{\nu}} c_{\mathbf{\nu}}O_{\mathbf{\nu}}$$
where $\mathbf{\nu}$ runs over matrices with less than $p$-columns (and where $c_{\mathbf{\nu}}\,\in\,k$). Hence, by induction on $p$, we see that $O_{\mathbf{\mu}}$ is represented as an element in the image of $\bSym_k^{\leq p}V$ under the homomorphism~\eqref{powsum}. Clearly, this representation is unique. This proves (i).
\end{proof}

\subsubsection{Proof of Conjecture~\ref{conj1} as $ n \to \infty $}\la{hcp2var}
The minimal free resolution of $ A = k[x,y] $ is given by $ R\,=\,k\langle x, y, \theta\rangle $ with differential $ d\theta =[x,y]$. Then $\,R_\nn \cong k[x, y, \theta]\,$ with $d \theta = 0 $ and
$$
S^n(R_\nn)\, \cong \,k[x_1, \ldots ,x_n,y_1, \ldots,y_n,\theta_1, \ldots ,\theta_n]^{S_n}\ , \ n \ge 1\ .
$$
By Theorem~\ref{ftt}, $\,\H_\bullet(R_\n) \cong \rHC_\bullet(A)\,$;
on the other hand, by \cite[Theorem~3.4.12]{L}, we can identify
$\,\rHC_{\bullet}(A) \cong \Omega^{\bullet}(A)/d\Omega^{\bullet -1}(A)\,$, where $ \Omega^{\bullet}(A) $ is the algebraic de Rham
complex of $A$. With these identifications, the reduced trace maps
\eqref{redttr} are given by the following formulas (see \cite{BFPRW})
\begin{flalign*}
& \TTr_n(A)_0\,:\,A \rar S^n(R_{\nn})\,,\,\,\,\, P(x,y) \mapsto \sum_{i=1}^n \,P(x_i,y_i) \ ,\\
&\TTr_n(A)_1\,:\,\Omega^1(A)/dA \rar S^n(R_{\nn})\,,\,\,\,\, [P(x,y)dx + Q(x,y)dy]
\mapsto \sum_{i=1}^n \,\left(P_y(x_i,y_i) - Q_x(x_i, y_i)\right)\theta_i \ \text{.}
\end{flalign*}
Using these formulas, it is easy to prove the following fact.
\blemma
\la{lempowsum}
The image of $\,\TTr_n(A)_{\bullet}\,$ in $\, S^n(R_{\nn})\,$ is spanned by the power sums
$$P_{abc}\,:=\,\sum_{i=1}^n x_i^ay_i^b\theta_i^c \ ,\quad (a,b,c) \in \Z_{\geq 0}^2 \times  \{0,1\} \,\text{.}$$
\elemma

Endow $R$ with the structure of a bigraded DG algebra, where $x,y\,\in\,R$ have weight $1$ and homological degree $0$ and $\theta\,\in\,R$ has weight $2$ and homological degree $1$. Note that there is a natural (surjective) homomorphism
$ S^n(R_{\nn}) \rar S^{n-1}(R_{\nn})$ for any $n \geq 2$. Let
$ S^{\infty}(R_{\nn})\,:=\,\varprojlim_n S^n(R_{\nn})$ where the projective limit is taken in the category of bigraded commutative (DG) algebras.
Using Proposition~\ref{sympowsum} and Lemma~\ref{lempowsum}, we can now establish
the following result, which can be viewed as a further evidence in favor of Conjecture~\ref{conj1}.

\begin{theorem}
\la{surjhc}
Let $ \Phi_n $ be the derived Harish-Chandra homomorphism for $ A = k[x,y]$, see \eqref{phiV}.

{\rm (i)}\, For any $ n \ge 1 $, the map
$\ \H_{\bullet}(\Phi_n)\,$ is degreewise surjective.

{\rm (ii)}\ The map $\, \H_{\bullet}(\Phi_{\infty})\,$ is an isomorphism of bigraded commutative algebras.
\end{theorem}

\begin{proof}
Note first that we have a commutative diagram
$$
\begin{diagram}
\bSym_k(\rHC_{\bullet}(A)) & & \\
  \dTo^{\bSym\,\Tr_n(A)_{\bullet}} & \rdTo^{\bSym\,\TTr_n(A)_{\bullet}} &\\
  \H_{\bullet}(A,n)^{\GL} & \rTo^{\H_{\bullet}(\Phi_n)} & S^n(R_{\nn})
\end{diagram}
$$
The map $\Phi_n$ being surjective follows from Lemma~\ref{lempowsum} and Proposition~\ref{sympowsum} (i). The map $\Phi_{\infty}$ being an isomorphism follows from Proposition~\ref{sympowsum} (ii), which says that $\bSym\,\TTr_{\infty}(A)_{\bullet}\,$ is an isomorphism and from~\cite[Theorem~4.4]{BR}, which says that $\bSym\,\Tr_{\infty}(A)_{\bullet} $ is an isomorphism.
\end{proof}

Note that part one of Theorem~\ref{surjhc} is the surjectivity part of
Conjecture~\ref{conj1}: thus, to prove Conjecture~\ref{conj1} it suffices
to prove that the map $\, \H_{\bullet}(\Phi_n)\,$ is injective.
As a corollary of Theorem~\ref{surjhc}, we obtain (in the $\gl_n$ case)
a classical result of A.~Joseph ({\it cf.} \cite{J}, Theorem 2.9).

\bcor
The restriction map
\begin{equation} \la{reshc} \Sym_k [\gl_n(k) \oplus \gl_n(k)]^{\GL} \rar \Sym_k [\h_n(k) \oplus
\h_n(k)]^{S_n} \,\cong \,k[x_1,\ldots,x_n,y_1,\ldots,y_n]^{S_n}
\end{equation}
is (graded) surjective.
\ecor

\begin{proof}
Note that \eqref{reshc} is precisely the composite map
$$\begin{diagram}k[\Rep_n(k\langle x,y\rangle)]^{\GL} & \rOnto & k[\Rep_n(k[x,y])]^{\GL} & \rTo^{\H_0(\Phi_n)} & S^n(k[x,y])\end{diagram}$$
The surjectivity of~\eqref{reshc} thus follows from that of $\H_0(\Phi_n)$.
\end{proof}

\subsubsection{The case of three variables}
\la{shcp3var}
One might expect that Conjecture~\ref{conj1} extends to polynomial algebras
of more than two variables. We now show that this is not the case. In fact,
Conjecture~\ref{conj1} fails already for polynomials of three variables.
Let $ A = k[x,y,z] $. As in~\cite[Section 6.3.2]{BFR},
we write the minimal resolution of $A$ in the form $\,R = k \langle x,y,z, \xi,\theta, \lambda, t\rangle\,$, where $\deg x=\deg y=\deg z=0$, $\deg \xi=\deg \theta=\deg\lambda=1$ and $\deg t=2$. The differential on $R$ is defined by
$$
d\xi = [y,z]\ ,\quad d\theta = [z,x]\ ,\quad d\lambda = [x,y]\ ,\quad dt = [x, \xi] + [y, \theta] + [z, \lambda]\ .
$$
Again, we equip $R$ with the structure of a bigraded DG algebra, with $x,y,z$ having weight $1$, $\xi,\theta,\lambda$ having weight $2$ and $t$ having weight $3$.

As in Section~\ref{hcp2var}, we can define $ S^{\infty}(R_{\nn})$ as a projective limit in the category of bigraded commutative (DG) algebras. Note that Proposition~\ref{sympowsum} (ii) implies that $ S^{\infty}(R_{\nn})$ can be naturally identified with the free bigraded commutative (DG) algebra generated by the power sums
\begin{equation}
\la{powersums}
P_{\mathbf{a}}\,:=\,\sum_{i=1}^{\infty} \xi_i^{a_1}\lambda_i^{a_2}\theta_i^{a_3}t_i^{a_4}x_i^{a_5}y_i^{a_6}z_i^{a_7} \ ,
\end{equation}
where the multi-index $ \mathbf{a}\,:=\,(a_1,\ldots,a_7)$ runs over $\{0,1\}^3 \times \Z_{\geq 0}^4$. 

Now, by Theorem~\ref{ftt} and \cite[Theorem~3.4.12]{L}, we can identify
$$
\H_\bullet(R_\n) \cong \rHC_\bullet(A) \cong \Omega^{\bullet}(A)/d\Omega^{\bullet -1}(A)\
$$
where $ \Omega^{\bullet}(A) $ is the algebraic de Rham complex of $A $.
This shows, in particular, that $ \rHC_i(A) $ vanish for $ i \ge 3 $. Representing the cyclic classes in $ \rHC_0(A) $, $\, \rHC_1(A) $ and $ \rHC_2(A) $
by differential forms
$$
\omega_0 = P\ ,\quad
\omega_1=Pdx+Qdy+Rdz\ ,\quad
\omega_2=Pdx\wedge dy+Qdy\wedge dz+R dz\wedge dx\ ,
$$
we have the following formulas for the (reduced) trace maps \eqref{redttr} derived in \cite{BFPRW}:
\begin{eqnarray*}
\TTr_{\infty}(A)_{0}\,[\omega_0]\,& = &\, \sum_{i=1}^{\infty}\,
P^i \ ,\\
\TTr_{\infty}(A)_{1}\,[\omega_1]\,& = &\,\sum_{i=1}^{\infty}\, (P_y^i - Q_x^i)\,\lambda_i
 + (Q_z^i -R_y^i)\,\xi_i+ (R_x^i - P_z^i)\,\theta_i\ ,\\
\TTr_{\infty}(A)_{2}\,[\omega_2]\,& = &\, \sum_{i=1}^{\infty}\, (P_{xz}^{i}+Q_{xx}^{i}+R_{xy}^{i})\,\theta_i\,\lambda_i + (P_{yz}^{i}+Q_{xy}^{i}+R_{yy}^{i})\,\lambda_i\,\xi_i +
(P_{zz}^{i}+Q_{xz}^{i}+R_{yz}^{i})\,\xi_i\,\theta_i\\
& + & \sum_{i=1}^{\infty} (P_z^{i}+Q_x^{i}+R_y^{i})\,t_i\ .
\end{eqnarray*}
Here, $ P_x, P_y, P_z, P_{xy},\, \ldots $ denote the
derivatives of a polynomial $ P \in k[x,y,z] $, and $P^{i}$ stand for the corresponding elements $ P(x_i, y_i, z_i) \in S^{\infty}(R_{\nn}) $.

The above formulas show that the image of the map $\,\TTr_{\infty}(A)_{\bullet} $ is spanned by
power sums \eqref{powersums}. However this map is {\it not} surjective: for example, it is easy to see that no power sum
$P_{\mathbf{a}}$ with $a_4 \geq 2$ appears in the image of $\TTr_{\infty}(A)_{\bullet}$.
By~\cite[Theorem 4.4]{BR}, we conclude
\bprop
\la{hcp3var} The map $\H_{\bullet}(\Phi_{\infty})\,:\,\H_{\bullet}(A,\infty)^{\GL} \rar S^{\infty}(R_{\nn})$ is injective but not surjective.
\eprop
Proposition~\ref{hcp3var} implies that if $ \dim(V) = 3 $, the derived Harish-Chandra homomorphism $ \Phi_n(A) $ for $ A = \Sym(V) \,$ cannot be a quasi-isomorphism for any $ n $.
More generally, in Section~\ref{otherex}, we will show that $ \Phi_n(A) $ is not a
quasi-isomorphism when $ \dim(V) \ge 3 $. Thus Conjecture~\ref{conj1} does not extend to higher dimensions. There is, however, another natural way to generalize this conjecture in
which case we can actually prove that it holds.

\subsection{Quantum polynomial algebras}
\la{qpoly}
In this section, we will show that (the analogue of) Conjecture~\ref{conj1}
holds for the $q$-polynomial algebra
\[
A_q := k\langle x,y\rangle/(xy-qyx) \ ,\quad q \in k^*\ .
\]
Precisely, we will prove
\begin{theorem}
\la{q-conj1}
If $q \in k^* $ is not a root of unity, the derived Harish-Chandra homomorphism
\[
\Phi_n(A_q):\,\DRep_n(A_q)^{\GL} \to S^n\DRep_1(A_q)
\]
is a quasi-isomorphism for all $\, n \ge 1 $.
\end{theorem}

We expect that the claim of Theorem~\ref{q-conj1} holds for an arbitrary value of $ q $ (except
possibly for $ q = 0 $); however, our proof relies on the following vanishing result that requires the restriction on $q$.
\blemma[\cite{BFR}]
\la{bfrvan}
If  $q \in k^* $ is not a root of unity, then, for all $\,n\ge 1\,$,
\[
\H_i(A_q,\,n) = 0\ ,\quad i > 0 \ .
\]
\elemma
For the proof of Lemma~\ref{bfrvan} we refer the reader to \cite[Section~6.2.2]{BFR}.

\begin{proof}[Proof of Theorem~\ref{q-conj1}]
Let us restate the claim of the theorem in explicit terms.
The minimal cofibrant resolution of $ A_q $ is given by the free DG algebra
$ R = \c\langle x,y, \theta \rangle $ with $ x, y $ of degree
$0$, $\,\theta $ of degree $1$, and the differential defined by
$\, d\theta = xy-qyx \,$. Hence, $ \DRep_n(A_q) $ and $ S^n\DRep_1(A_q) $ are
represented respectively by the following DG algebras
\begin{equation}
\la{algB}
B_n := k[x_{ij}, y_{ij},\theta_{ij}\,:\, i, j = 1,2, \ldots, n]^{\GL}
\ ,\quad  d\theta_{ij} = \sum^n_{k=1}\, (x_{ik} y_{kj}\,-\,q y_{ik} x_{kj})\ ,
\end{equation}
\begin{equation}
\la{algE}
E_n := k[x_{i}, y_{i},\theta_{i}\,:\, i=1,2,\ldots, n]^{S_n}
\ ,\quad  d\theta_{i} = (1-q) x_i y_i\ .
\end{equation}
The derived Harish-Chandra homomorphism \eqref{ehc3} is explicitly given by
\begin{equation}
\la{HChom2}
\Phi_n :\, B_n \to E_n\ ,\quad
x_{ij} \mapsto \delta_{ij} x_i\ ,\quad
y_{ij} \mapsto \delta_{ij} y_i\ , \quad
\theta_{ij} \mapsto \delta_{ij} \theta_i\ .
\end{equation}
We need to show that \eqref{HChom2} is a quasi-isomorphism.

First, we put an augmentation on the algebra $ A_q $ letting
$ \varepsilon(x) = \varepsilon(y) = 0 $.
The DG algebra $ R $, and hence $ B_n $ and $ E_n $, then become augmented in
a natural way, and the map \eqref{HChom2} preserves the augmentations. Next,
using the notation of Section~\ref{trmaps}, we set $\, R_\n := \bar{R}/[\bar{R}, \bar{R}] \,$
and define $ W $ to be $ \bSym(R_\n) $.
Note that $ W $ is filtered by the graded vector spaces
$\, W^{\leq n} := \bSym^{\leq n}(R_\n) \,$, each of which carries a
differential induced from $ R $. We may think of $ R_\n $ as the space of
cyclic words in letters $\, x,\,y,\,\theta \,$; then, $ W^{\leq n} $
is spanned by products of at most $ n $ such words.

Now, consider the composition of morphisms of complexes
\begin{equation}
\la{qqism}
W^{\leq n} \xrightarrow{\Tr_n} B_n \xrightarrow{\Phi_n} E_n\ .
\end{equation}

The first map in \eqref{qqism} is defined by the trace morphism \eqref{trm2}: explicitly,
it sends a cyclic word to the corresponding trace expression (e.g., $\,xy\theta \mapsto \sum_{i,j,k} x_{ij} y_{jk} \theta_{ki}$). The second map is the Harish-Chandra homomorphism \eqref{HChom2}.

We claim that each arrow in \eqref{qqism} is actually a quasi-isomorphism.
To see this we first show that the composition is a quasi-isomorphism.
By Theorem~\ref{ftt}, the homology of $ R_\n $
can be identified with the reduced cyclic homology $ \rHC_\bullet(A_q) $.
When $q$ is not a root of unity, the latter is known to vanish in all
positive degrees while spanned by the cyclic words $ \{x^p\} $ and $ \{y^p\} $
in degree $0$ (see \cite[Th\'eor\`eme~2.1]{Wa}): i.e.,
$$
\rHC_\bullet(A_q) = \rHC_0(A_q) \cong \,\bigoplus_{p \geq 1}\, (k x^p \oplus k y^p)
$$
This implies
\begin{equation}
\la{iss1}
\H_\bullet(W^{\leq n}) =
\H_0(W^{\leq n})\cong
\Sym^{\leq n}[\,\H_0(R_\n)\,] \cong
\Sym^{\leq n}[\,\rHC_0(A_q)\,] \cong
k[X_p,\,Y_p\,:\,p=1,2,\ldots]^{\leq n}\, ,
\end{equation}
where the variables $ X_p $ and $ Y_p $ correspond to the cyclic words $ x^p $ and $y^p $.

Now, by Proposition~\ref{sympowsum}(i), we can identify
$$
E_n \cong \bSym^{\leq n}(\bar{R}_\nn)\ ,
$$
where $ R_\nn \cong k[x,y,\theta] $ with differential $ d\theta = (1-q)xy $.
In this case, $ \H_\bullet(R_\nn) $ is concentrated in degree 0  and
$ \H_0(R_\nn) $ is spanned (as a vector space) by the classes of
$x^p $ and $y^p $ with $\, p\ge 1 $. Hence the natural projection
$ R_\n \onto \bar{R}_\nn $ induces an isomorphism of graded vector spaces
$ \H_\bullet(R_\n) \stackrel{\sim}{\to} \H_\bullet(\bar{R}_\nn) $, which gives
\begin{equation}
\la{iss2}
\H_\bullet(W^{\leq n}) \stackrel{\sim}{\to} \H_\bullet(E_n)\ .
\end{equation}
It is easy to see that the isomorphism \eqref{iss2} is actually induced by the composite
map \eqref{qqism}: thus \eqref{qqism} is a quasi-isomorphism. Now, by
\cite[Theorem~3.1]{BR}, we know that the trace map $ \Tr_n: W^{\leq n} \to R_n^{\GL} $
is (degreewise) surjective. Since both $ W^{\leq n} $ and $ R_n^{\GL} $ are non-negatively
graded, the map induced by $ \Tr_n $ on the zero homology is also surjective:
\begin{equation}
\la{surj1}
\H_0(\Tr_n):\, \H_0(W^{\leq n}) \onto \H_0(B_n)
\end{equation}
On the other hand, by Lemma~\ref{bfrvan}, the homology of $ R_n $ is concentrated in degree $0$. Since $ \GL_n(k) $ is reductive, this implies
\begin{equation}
\la{iss3}
\H_\bullet(B_n) \cong \H_\bullet(R_n)^{\GL} \cong \H_0(R_n)^{\GL}
\cong \H_0(B_n)\ .
\end{equation}
Combining \eqref{iss2}, \eqref{surj1} and \eqref{iss3}, we now see that \eqref{qqism}
induces an isomorphism
\begin{equation}
\la{iss4}
\H_\bullet(W^{\leq n}) \onto \H_\bullet(B_n) \to  \H_\bullet(E_n) \ ,
\end{equation}
where the first arrow is also surjective. It follows that both maps in \eqref{iss4}
are isomorphisms. In particular, $ \Phi_n $ is a quasi-isomorphism.
\end{proof}

\begin{remark}
Note that \eqref{iss1} gives isomorphisms
\begin{equation}
\la{iss5}
\H_\bullet(B_n) \cong \H_\bullet(E_n) \cong k[X_p,\,Y_p\,:\, p =1,2,\ldots]^{\leq n}\ ,
\end{equation}
where the variables $ X_p $ and $ Y_p $ have homological degree $0$.
\end{remark}

\section{Euler characteristics and constant term identities}
\la{sec5}

In this section, we compare the Euler characteristics of both sides of Conjecture~\ref{conj1} and prove the resulting identities for all $n$. Throughout, we assume that $k =\c$.

\subsection{A constant term identity for $ \gl_n $}
\la{secid}
The target of the Harish-Chandra homomorphism~\eqref{phiV} is the graded commutative
algebra $ E_n := \c[x_1,\dots,x_n,y_1,\dots,y_n,\theta_1,\dots,\theta_n]^{S_n}$
with $x_i,y_i$ of
homological degree $0$ and $ \theta_i $ of degree $1$. It has an additional
$\mathbb Z^2$ grading, which we call weight,
such that wt$(x_i)=(1,0)$, wt$(y_i)=(0,1)$,
wt$(\theta_i)=(1,1)$.

Let $ G_n(q,t,s)$ be the generating function
\[
G_n(q,t,s)=\sum\mathrm{dim}(E_n)_{d,a,b}s^dq^at^b
\]
of the dimensions of
the homogeneous components $ (E_n)_{d,a,b} $ of degree $d$ and weight $(a,b)$

\blemma
\la{lemAA}
$\, G_n(q,t,s)$ is equal to the coefficient of $v^n$ in the
Taylor expansion at $v=0$ of
\[
\prod_{a,b\geq0}\frac{1+q^{a+1}t^{b+1}sv}{1-q^at^bv}
\]
\elemma
\begin{proof}
A generating set is given by the $S_n$ orbit
sums of monomials $\theta^\nu x^\lambda y^\mu$ with exponents $\lambda,\mu\in
\mathbb Z_{\geq0}^n$, $\nu\in \{0,1\}^n$. Such an orbit sum
contributes $q^{|\lambda|+|\nu|}t^{|\mu|+|\nu|}s^{|\nu|}$ to the generating function,
where $|\lambda|=\sum\lambda_i$.
To get a basis we take the subset
of such orbit sums whose exponents obey
 (i) $\nu$ is a partition: $\nu_i\geq \nu_{i+1}$
for all $i=1,\dots,n-1$, (ii) $\lambda_i\geq\lambda_{i+1}$ whenever $\nu_i=\nu_{i+1}$, (iii) $\mu_i\geq\mu_{i+1}$
whenever $\nu_i=\nu_{i+1}=0$ and $\lambda_i=\lambda_{i+1}$ and (iv) $\mu_i>\mu_{i+1}$ whenever $\nu_i=\nu_{i+1}=1$ and $\lambda_i=\lambda_{i+1}$. Such data are
in one-to-one correspondence with families $(m_{dab})$ of
non-negative integers labeled by triples $(d,a,b)\in\{0,1\}\times\mathbb Z_{\geq0}^2$ such that $|m|=\sum_{d,a,b}m_{dab}=n$ and $m_{1ab}\leq 1$:
for $(\nu,\lambda,\mu)$ obeying (i)--(iv),
$m_{dab}$ is the number of parts of size $b$ of the partition $(\mu_j)_{j\in J}$
where $j\in J$ iff $\nu_j=d$ and $\lambda_j=a$. This gives
\begin{align*}
\sum_{n=0}^\infty G_n(q,t,s)v^n&=
\prod_{a,b\geq 0}
\sum_{m=0}^\infty q^{ma}t^{mb}v^m
\prod_{a,b\geq 0}
\sum_{m=0}^1 q^{m(a+1)}t^{m(b+1)}s^{m}v^m
\\
&=\prod_{a,b\geq0}\frac{1+q^{a+1}t^{b+1}sv}{1-q^at^bv}
\end{align*}
\end{proof}
The (weighted)
Euler characteristic of $ E_n $ is the value of $ G_n$ at $s=-1$, where
the formula simplifies considerably. Let us use the standard notation
$(v;q)_\infty :=\prod_{j=0}^\infty(1-q^jv)$, $(v;q)_n :=(v;q)_\infty/(q^nv;q)_{\infty}$.
\begin{corollary}
\la{corA}
The Euler characteristic $\chi(E_n,q,t)= G_n(q,t,s=-1)$ is the coefficient of
$v^n$ in the Taylor expansion at $v=0$ of the function
\[
\frac{1-v}{(v;q)_\infty\,(v;t)_\infty}
\]
More explicitly,
\[
\chi(E_n,q,t)=\sum_{j=0}^n\frac {t^j}{(q;q)_{n-j}(t;t)_j}.
\]
\end{corollary}
\begin{proof}
The first statement is immediate from Lemma~\ref{lemAA};
the second follows from the well-known Rothe identity
\[
\frac1{(v;q)_\infty}=\sum_{n=0}^\infty \frac {v^n}{(q;q)_n}\ .
\]
\end{proof}

Let us compare this Euler characteristic with the Euler characteristic
of the graded commutative algebra $ B_n := R_n^{\GL}$ underlying
the invariant part of the complex computing representation homology. Here
\[
R_n= \c[x_{ij}, y_{ij},\theta_{ij} \,:\,i,j=1,\dots,n].
\]
Again, we assign degree $0$, $0$, $1$ and weight $(1,0)$, $(0,1)$ and $(1,1)$
to $x_{ij}$, $y_{ij}$ and $\theta_{ij}$, respectively. The weighted Euler
characteristic of $B_n$ can be extracted from the character valued
weighted Euler characteristic
\[
\chi(R_n,q,t,u)=\sum \mathrm{trace}(u|(R_n)_{d,a,b})(-1)^dq^at^b,\quad u=\mathrm{diag}(u_1,\dots,u_n)\in \GL_n
\]
of the $\GL_n$-module $ R_n$.
\blemma
\[
\chi(R_n,q,t,u)=\prod_{1\leq i,j\leq n}\frac{1-q t u_i/u_j}
{(1-q u_i/u_j)(1-t u_i/u_j)}
\]
\elemma
\begin{proof}
We use the basis of monomials, noticing that
$u$ acts on $x_{i,j},y_{i,j},\theta_{i,j}$ by multiplication by
$u_i/u_j$.
\end{proof}
To find the weighted Euler characteristic of $B_n$ we need to select the
coefficient of $1$ in the expansion of the coefficients of the
character valued Euler characteristic in Schur polynomials. This is
obtained by taking the $1/n!$ times the
constant term of the product with the Weyl denominator:

\begin{corollary}
\la{corB}
Let $\mathrm{CT}\colon\mathbb Z[u_1^{\pm1},\dots,u_n^{\pm1}][[q,t]]\to\mathbb Z[[q,t]]$ be the constant term map. Then
\[
\chi(B_n,q,t)=\frac1{n!}\,\frac{(1-qt)^n}{(1-q)^n(1-t)^n}\ \
\mathrm{CT}
\prod_{1\leq i\neq j\leq n}
\frac{(1-qt u_i/u_j)(1-u_i/u_j)}{(1-q u_i/u_j)(1-t u_i/u_j)}.
\]
\end{corollary}

Now, Conjecture~\ref{conj1} implies the equality
\begin{equation}
\la{eqeuler}
\chi(E_n,q,t)=\chi(B_n,q,t)\ ,
\end{equation}
which, using Corollary~\ref{corA} and Corollary~\ref{corB}, we can write as the following
identity
\begin{equation}
\label{conj2}
\frac1{n!}\,\frac{(1-qt)^n}{(1-q)^n(1-t)^n}\ \
\mathrm{CT}
\prod_{1\leq i\neq j\leq n}
\frac{(1-qt u_i/u_j)(1-u_i/u_j)}{(1-q u_i/u_j)(1-t u_i/u_j)}\
=\ \sum_{j=0}^n\frac {t^j}{(q;q)_{n-j}(t;t)_j}\ .
\end{equation}

The main result of this section is
\begin{theorem}
\la{Tconj2}
The identity \eqref{conj2} holds for all $ n \ge 1 $.
\end{theorem}
\bproof
This is an immediate consequence of Theorem~\ref{q-conj1}. Indeed, if we forget
differentials, the commutative algebras $ B_n $ and $ E_n $
introduced in this section coincide with the algebras \eqref{algB}
and \eqref{algE} introduced in Section~\ref{qpoly}.
The differentials in \eqref{algB} and \eqref{algE} respect the weight
gradings and so does the Harish-Chandra homomorphism \eqref{HChom2}.
Theorem~\ref{q-conj1} thus implies the equality \eqref{eqeuler} which is
equivalent to \eqref{conj2}.
\eproof

\vspace{1ex}

\noindent
{\bf Remarks.}
1. The last isomorphism in \eqref{iss5} allows one to compute the right-hand side of
the identity \eqref{conj2} directly, without
using Lemma~\ref{lemAA}. Indeed, $ \c[X_p,\,Y_p\,:\,p=1,2,\ldots]^{\leq n} $
consists of polynomials in $X_1,Y_1, X_2,Y_2, \ldots\, $ of degree
$ \leq n $. The number of monomials of degree exactly
$ n $ and of bidegree $(a,b)$ is the coefficient of $v^n q^at^b$ in
the generating function $\,\prod_{k \ge 1}(1-vq^k)^{-1} (1-vt^k)^{-1} $.
The number of polynomials of degree $ \leq n $ can be obtained by multiplying this function
by $1/(1-v)$. The result of Corollary~\ref{corA} follows then easily
from \eqref{iss5}.

2. The identity \eqref{conj2} can be rewritten equivalently in the following form
\[
\int_{U(n)}\,
\frac{\det(1\,-\,qt\,\mathrm{Ad}(g))}
{\det(1\,-\,q\,\mathrm{Ad}(g))\,\det(1\,-\,t\,\mathrm{Ad}(g))}\
d g\ =\
\frac{1}{n!} \sum_{\sigma \in S_n} \frac{\mathrm{det}(1-qt\sigma)}{\mathrm{det}(1-q\sigma)\mathrm{det}(1-t\sigma)}\ ,
\]
where the integration on the left is taken with respect to the normalized Haar measure
over the $n$-th unitary group $ U(n) $ and the determinants on the
right are taken with respect to the natural action of $S_n$ on $\c^n$.
In this form, the identity \eqref{conj2} extends to an arbitrary reductive Lie
algebra (see Section~\ref{constid}).

3. The left-hand side of \eqref{conj2} is the expansion at $\,q=t=0\,$ of
an integral over the product of unit circles $|u_i|=1$ defined for
$\,|q|,\,|t|<1\,$ and can be computed by iterated residues, with the
following result: let $\,Z_n(q,t)\,$ be the left-hand side of
\eqref{conj2} and $\,Z(v,q,t)=1+\sum_{n=1}^\infty Z_n(q,t)v^n$ the
generating function. Then
\[
Z(v,q,t)=\exp\left(\sum_{\lambda}\frac {v^{|\lambda|}}{|\lambda|}W_\lambda(q,t)\right).
\]
The sum is over all partitions (non-increasing integer sequences converging to $0$) $\lambda=(\lambda_1\geq\lambda_2\geq \cdots\geq 0)$ of positive size
 $|\lambda|=\sum\lambda_i$. The coefficient $W_\lambda(q,t)$ is a
regularized product over the boxes $(i,j)$ of the Young diagram
\[
Y(\lambda)=\{(i,j)\in\mathbb Z^2, 1\leq j\leq \lambda_i\}=
\{(i,j)\in\mathbb Z^2, 1\leq i\leq \lambda_j'\}
\]
of the partition $\lambda$, with conjugate partition $\lambda'\,$:
\[
W_\lambda(q,t)=
\prod_{(i,j)\in Y(\lambda)}\hspace{-0.5cm}{}'\hspace{0.5cm}\frac{(1-q^jt^{i})(1-q^{-j+1}t^{-i+1})}
{(1-q^{\lambda_i-j+1}t^{-\lambda'_j+i})(1-q^{-\lambda_i+j}t^{\lambda'_j-i+1})}
\]
In the regularized product $\Pi'$ we omit the factor $1-q^0t^0$ appearing
at $(i,j)=(1,1)$. Using the product formula of Corollary \ref{corA}
for the generating function of the right-hand side of \eqref{conj2} we
see that the identity \eqref{conj2} is equivalent to
\[
\sum_{|\lambda|=n}W_\lambda(q,t)=\frac{1-q^nt^n}{(1-q^n)(1-t^n)}, \quad
n=1,\,2,\,\dots
\]
This calculation is a trigonometric version of a similar calculation
occurring in supersymmetric gauge theory \cite{Nekrasov}:
the Nekrasov instanton partition function of
$\mathcal N=2$ super Yang--Mills theory on
$\mathbb{R}^4$ with $U(N)$ gauge group has both an
integral representation and an expression as a sum over collections of
partitions (see \cite{Nekrasov}, (3.10) and (1.6), respectively).
They are related by iterated residues.

\subsection{Other examples}
\la{otherex}
As mentioned earlier, the Harish-Chandra homomorphism
$ \Phi_n(A) $ induces an isomorphism on the 0-th homology
for any commutative algebra. In Section~\ref{shcp3var}, we have shown that
this isomorphism  cannot be extended to higher homologies for
 $ A = k[x,y,z] $ . We now give a general argument showing that
$ \Phi_n $ is not a quasi-isomorphism for $ A = \Sym(V) $ when
$ \dim(V) \ge 3 $.  We will also look the algebra of dual numbers,
in which case the analog of Conjecture~\ref{conj1} also fails.

\subsubsection{Euler characteristics for arbitrary symmetric algebras}
\la{scounter}
Our aim is to prove
\bprop \la{counter}
The  derived Harish-Chandra homomorphism
$$
\Phi_n\,:\,\DRep_n[\Sym(V)]^{\GL} \,\rar\, \bSym(\h_n^{\ast} \otimes (\wedge^1V \oplus \wedge^2 V \oplus \wedge^3 V \oplus \ldots \oplus \wedge^r V))^{S_n}
$$
is not a quasi-isomorphism when the dimension $r$ of $V$ is at least $3$ (and when $n \geq 2$).
\eprop

\begin{proof}
Suppose that $r := \dim_k V = 3$. Choose a basis $\{x,y,z\}$ of $V$. Note that both $\DRep_n(\Sym(V))^{\GL}$ and $\bSym(\h_n^{\ast} \otimes (\wedge^1V \oplus \wedge^2 V \oplus \wedge^3 V))^{S_n}$ are $\Z^3$-graded, with $x$ having weight $(1,0,0)$, $y$ having weight $(0,1,0)$ and $z$ having weight $(0,0,1)$ (these weights uniquely determine the weight of each generator of $\DRep_n(\Sym(V))$ as well as the weight of each generator of $\bSym(\h_n^{\ast} \otimes
(\wedge^1V \oplus \wedge^2 V \oplus \wedge^3 V))$). Further note that by its very construction, the derived Harish-Chandra homomorphism is weight preserving.
We then claim that
\blemma \la{checkcounter}
The subcomplexes of weight $(1,1,1)$ in $\DRep_n(\Sym(V))^{\GL}$ and $\bSym(\h_n^{\ast} \otimes (\wedge^1V \oplus \wedge^2 V \oplus \wedge^3 V))^{S_n}$ have different Euler characteristics.
\elemma
This shows that $\Phi_n$ cannot be a quasi-isomorphism when $r=3$ and $n \geq 2$.
Further, when $r >3$, one can choose a basis $\{x_1,\ldots,x_r\}$ of $V$, and equip both
$\DRep_n(\Sym(V))^{\GL}$ and
$\bSym(\h_n^{\ast} \otimes (\wedge^1V \oplus \wedge^2 V \oplus \wedge^3 V \oplus \ldots \oplus \wedge^r V))^{S_n}$
with the $\Z^r$-grading determined by fixing the weight of $x_i$ to be $(0, \ldots ,0,1,0,\ldots,0)$
(with $i$-th coordinate $1$). Again, the derived Harish-Chandra homomorphism is weight preserving
by its very construction. Further, it is easy to see that the subcomplexes of weight $(1,1,1,0, \ldots ,0)$
in $\DRep_n(\Sym(V))^{\GL}$ and $\bSym(\h_n^{\ast} \otimes
(\wedge^1V \oplus \wedge^2 V \oplus \wedge^3 V \oplus \ldots \oplus \wedge^r V))^{S_n}$
are isomorphic to their counterparts for the case when $r=3$.
Thus, Lemma~\ref{checkcounter} implies the desired proposition in general.

It remains to verify Lemma~\ref{checkcounter}. For notational brevity, we write
$\zeta \,:=\, -\lambda(x,y)$, $\eta\,:=\,-\lambda(z,x)$, $\xi\,:=\,-\lambda(y,z)$ and
$t\,:=\,-\lambda(x,y,z)$. Then, the minimal resolution $R$ is generated by
$x,y,z,\zeta, \eta, \xi,t$ with $d\zeta \,=\,[x,y]$,
$d\eta = [z,x]$, $d \xi = [y,z]$ and $dt = [\xi, x]+[\eta, y]+[\zeta, z]$.
Thus, $\DRep_n(k[x,y,z])$ is generated by the variables
$x_{ij}, y_{ij}, z_{ij}, \zeta_{ij}, \eta_{ij}, \xi_{ij}, t_{ij}$ for $ 1 \leq i,j \leq n$
where $x_{ij} \,:=\,e^{ij} \otimes x$, etc. The subcomplex of
$\DRep_n(k[x,y,z])^{\GL}$ of weight $(1,1,1)$ is of the form $0 \rar C_2 \rar C_1 \rar C_0 \rar 0$ where
\begin{eqnarray*}
C_2\,& := &\, \mathrm{Span}\{ \mathrm{Tr}(t) \} \\
C_1\,& := &\,\mathrm{Span}\{ \mathrm{Tr}(\xi x),\Tr(\eta y), \Tr(\zeta z), \Tr(\xi)\Tr(x), \Tr(\eta)\Tr(y),  \Tr(\zeta)\Tr(z) \}\\
C_0\,& := &\, \mathrm{Span}\{\mathrm{Tr}(xyz), \Tr(xzy), \Tr(xy)\Tr(z), \Tr(yz)\Tr(x), \Tr(zx)\Tr(y), \Tr(x)\Tr(y)\Tr(z) \}
\end{eqnarray*}
For $n \geq 3$ the above elements are linearly independent. When $n=2$, there is one relation:
$$\Tr(xyz)+\Tr(xzy)-\Tr(xy)\Tr(z)-\Tr(yz)\Tr(x)-\Tr(zx)\Tr(y)+\Tr(x)\Tr(y)\Tr(z)=0\,\text{.}$$
Hence, $\chi(\DRep_n(k[x,y,z])^{\GL}_{(1,1,1)})\,=\,1$ when $n \geq 3$ and $0$ for $n=2$.

On the other hand, writing $\,x_i\,:=\,e^i \otimes x\,$, etc., we identify
$\,\bSym[\h_n^{\ast} \otimes (\wedge^1V \oplus \wedge^2 V \oplus \wedge^3 V)]^{S_n}\,$ with
$$
A := k[x_1,\ldots,x_n,y_1,\ldots,y_n,z_1,\ldots,z_n,\xi_1,\ldots,\xi_n,\eta_1,\ldots,\eta_n,\zeta_1,\ldots,\zeta_n,t_1,\ldots,t_n]^{S_n}\,\text{.}
$$
In homological degree $0$, we have
$$A_0(1,1,1)\,=\,\mathrm{Span}(\sum_{(a,b,c) \in O} x_ay_bz_c) $$
where $O$ runs over the orbits of $S_n$ in $\{1,\ldots ,n\}^3$. For $n \geq 3$ there are five such orbits: those of $(1,1,1),(1,1,2),(1,2,1),(2,1,1)$ and $(1,2,3)$. In homological degree $1$,
$$A_1(1,1,1)\,=\,\mathrm{Span}(\sum_{(a,b) \in O}\xi_ax_b  , \sum_{(a,b) \in O} \eta_ay_b, \sum_{(a,b) \in O} \zeta_az_b ) $$
where $O$ runs over the orbits of $S_n$ in $\{1, \ldots ,n\}^2$. There are two such orbits: those of $(1,1)$ and $(1,2)$.
Finally, $$A_2(1,1,1) = \mathrm{Span}(t_1+\ldots+t_n)\,\text{.}$$
Thus, the Euler characteristic of $A(1,1,1)$ is $0$ when $n \geq 3$ (and $-1$ when $n=2$). This verifies Lemma~\ref{checkcounter}, thereby completing the proof of Proposition~\ref{counter}.
\end{proof}

\subsubsection{Dual numbers}
Let $A\,:=\,\c[x]/(x^2)\,$, the algebra of dual numbers. In this case, $C\,=\,T^c(V)$ where $V$ is a $1$-dimensional vector space in
homological degree $1$. Explicitly, $\cb(C)$ is the free DG algebra $R\,:=\,\c\langle x,t_1,t_2,\ldots.\rangle$ with $x$ having homological degree $0$
and $t_i$ having homological degree $i$ for all $i \in \mathbb N$. The differential on $\cb(C)$ is given by
$$dt_p \,=\,xt_{p-1}-t_1t_{p-2} + \ldots +(-1)^{p-1}t_{p-1}x \,,\,\,\,p \geq 1\,\text{.}$$
Note that as DG algebras, both $A$ and $R$ are {\it weight graded} as well, with the weight of $x$ being $1$ and that of $t_p$ being $p+1$ for
all $p \geq 1$. Note that the derived Harish-Chandra homomorphism is weight preserving by its very construction. Also, all the complexes involved
have finite total dimension in each weight degree. To see that the derived Harish-Chandra homomorphism is not a quasi-isomorphism in this case,
it suffices to check that the weighted Euler characteristics
$\chi(\DRep_2(A)^{\GL},q)$ and $\chi((\DRep_1(A)^{\otimes 2})^{S_2},q)$ differ. Here, for $V$ a weight graded complex of $\c$-vector spaces that has
finite total dimension in each weight degree, we define the weighted Euler characteristic of $V$ by
$$\chi(V,q)\,:=\,\sum_{i,j} (-1)^i\dim_{\c}V_i(j)q^j$$
where $V_i(j)$ is the component of $V_i$ of weight $j$.

Note that
$$\chi(S^2\DRep_1(A), q) \,=\,\frac{1}{2}[\chi(\DRep_1(A),q)^2+\chi(\DRep_1(A),q^2)]\,\text{.}$$
Since $\chi(\DRep_1(A),q)\,=\,\prod_{k=1}^{\infty}\,
(1-q^{2k+2}) (1-q^{2k+1})^{-1}  = \sum_{j=0}^{\infty}\, q^{j(j+1)/2}\,$,
\begin{equation} \la{ehc4} \chi(S^2\DRep_1(A),q) \,=\,1+q+q^2+q^3+q^4+2q^6+ \ldots
\end{equation}
On the other hand (see~\cite[Section 7.6]{BR}),
$$\chi(\DRep_2(A)^{\GL},q)\,=\, \int_{U(2)} \,
\prod_{i=1}^{\infty} \mathrm{det}(1\,- \,q^i\,\mathrm{Ad}(g))^{(-1)^i}d g \,\text{.}
$$
The above integral is taken over the unitary group $ U(2) \subset \GL_2(\c)$,
where the Haar measure $ d g $ is normalized so that the volume of
$U(2)$ is $1$. The determinant is taken in the adjoint representation of $\GL_2(\c)$ on $\M_2(\c)$. The above integral can be computed directly, giving
\begin{equation} \la{ehc5}  \chi(\DRep_2(A)^{\GL},q) \,=\,\frac{1}{1-q} \,\text{.}\end{equation}
Comparing~\eqref{ehc4} and~\eqref{ehc5}, one sees that the derived Harish-Chandra homomrphism
$$
\Phi_2\,:\,\DRep_2(\c[x]/(x^2))^{\GL_2} \rar S^2\DRep_1[\c[x]/(x^2)]
$$
is not an isomorphism in $\Ho(\cDGA_{k/k})$.

\section{Derived representation schemes of Lie algebras}
\la{sec6}

In this section, we construct a (derived) representation functor on the category of DG Lie algebras. We prove a basic result (see Theorem~\ref{drepg}) relating the derived representation functor of  Lie algebra with the homology of a Lie coalgebra defined over a Koszul dual cocommutative coalgebra.

\subsection{Classical representation schemes}

Let $\g$ be a finite-dimensional Lie algebra. For any Lie algebra $\mathfrak{a}$, there is an affine scheme $\Rep_{\g}(\mathfrak{a})$ parametrizing representations of $\mathfrak{a}$ in $\g$. More precisely, the functor
$$ \g(\,\mbox{--}\,)\,:\,\cAlg_k \rar \LieAlg_k\,,\,\,\,\, B \mapsto \g(B)\,:=\,\g \otimes B $$
has a left adjoint
$$ (\,\mbox{--}\,)_{\g}\,:\,\LieAlg_k \rar \cAlg_k\,,\,\,\,\,\mathfrak{a} \mapsto \mathfrak{a}_{\g} \,\text{.}$$ In particular, for a fixed Lie algebra $\mathfrak{a}$, the commutative algebra $\mathfrak{a}_{\g}$ represents the functor\footnote{This functor in a more general operadic setting is briefly discussed in~\cite[Section 6]{Gi}. }
$$
\Rep_{\g}(\mathfrak{a})\,:\,\cAlg_k \rar \mathtt{Sets}\,,\,\,\,\, B \mapsto \Hom_{\LieAlg_k}(\mathfrak{a},\g(B))\,\text{.}
$$
By definition,
$\,\mathfrak{a}_{\g}$ is the coordinate ring of the affine scheme parametrizing the representations of $\mathfrak{a}$ in $\g$.
For example, if $\mathfrak{a}$ is the abelian two dimensional Lie algebra over $k$ and $\g$ is reductive, then $\Rep_{\g}(\mathfrak{a})$ is the classical commuting scheme  of the Lie algebra $\g$.

In this section, we extend the functor $(\,\mbox{--}\,)_{\g}$ to the category $\DGL_k$ of DG Lie algebras and derive it using the natural model structure on $\DGL_k $. We then
define the representation homology $ \H_\bullet(\mathfrak{a}, \g) $ as the homology of the
corresponding derived functor.

\subsection{Quillen equivalences for Lie (co)algebras}
\la{sec6.1}
We begin with a brief review of the bar/cobar formalism in the Lie setting.

\subsubsection{Basic functors} Recall that $\cDGA_{k/k}$ denotes the category of commutative DG algebras augmented over $k$. Let $\DGL_k$ denote the category of DG Lie algebras over $k$. Let $\cDGC_{k/k}$ denote the category of DG cocommutative conilpotent coalgebras co-augmented over $k$. Similarly, let $\DGLC_k$ denote the category of conilpotent DG Lie coalgebras.

There is a pair of adjoint functors
\begin{equation}
 \la{cbbccom}
\cb_{\mathtt{Comm}}\,:\,
\mathtt{DGCC}_{k/k} \rightleftarrows \DGL_k \,:\,\bB_{\mathtt{Lie}} \ ,
\end{equation}
where $ \bB_{\mathtt{Lie}} $ is defined by the classical Chevalley-Eilenberg complex
of a DG Lie algebra (see Section~\ref{CEcoalg}) and $ \cb_{\mathtt{Comm}} $ is the functor 
assigning to a cocommutative coalgebra $C$ the free graded Lie 
algebra
on the vector space $\bar{C}[-1]$. The differential on 
$ \cb_{\mathtt{Comm}}(C) $ is given by $\,d_1+d_2\,$, where
$d_1$ is induced by the inner differential on $ C $ and $d_2$ is the lift of the 
linear map
\begin{diagram} k[-1] \otimes \bar{C} & \rTo^{\ \Delta_{-1} \otimes \Delta\ } & k[-1] \otimes k[-1] \otimes \bSym^2(\bar{C})\,\cong\, \wedge^2 (\bar{C}[-1])\ ,
\end{diagram}
where $ \Delta_{-1}\,:\,k[-1] \rar k[-1] \otimes k[-1]$ takes $1_{k[-1]}$ to $-1_{k[-1]} \otimes 1_{k[-1]}$ and $\Delta $ is the coproduct on $\bar{C}$.

Dually, there is a pair of adjoint functors
\begin{equation}
\la{cbbcom}
\cb_{\mathtt{Lie}}: \DGLC_k \rightleftarrows \cDGA_{k/k} : \bB_{\mathtt{Comm}} \ ,
\end{equation}
where $ \cb_{\mathtt{Lie}} $ is defined by the Chevalley-Eilenberg complex of a
DG Lie coalgebra (see Section~\ref{CEcoalg}), and
$ \bB_{\mathtt{Comm}} $ is the functor taking $R \in \cDGA_{k/k}$ to the cofree DG Lie coalgebra
$\mathfrak L(\bar{R}[1])$ equipped with (co)differential $d_1+d_2$ where $d_1$ is induced by the differential on $R$ and $d_2$ is determined by the linear map
\begin{diagram} \wedge^2(\bar{R}[1]) \,\cong\, k[1] \otimes k[1] \otimes \bSym^2(\bar{R}) & \rTo^{\ \mu_{1} \otimes \mu\ } &
 k[1] \otimes \bar{R}\,\cong\, \bar{R}[1]
\end{diagram}
(Here, $\mu_{1}$ identifies $1_{k[1]} \otimes 1_{k[1]}$ with $1_{k[1]}$ and 
$\mu: \bSym^2(\bar{R}) \to \bar{R} $ is induced by the multiplication map on $ R$).

\vspace{2ex}

\noindent
{\bf Notation.}\ If there is no danger of confusion, we will use the notation
$\, \CE := \bB_{\mathtt{Lie}} \,$ and $\, \CE^c := \cb_{\mathtt{Lie}} $ for the
Chevalley-Eilenberg functors on Lie algebras and Lie coalgebras, respectively.

\subsubsection{Model structures and Quillen theorems}
The following theorem collects basic facts about the model structures and
Quillen equivalences for Lie (co)algebras: part (i) is well known
(essentially, due to Quillen \cite{Q2}); for part (ii) and (iii),
see, for example, \cite[Theorems 3.1 and 3.2]{Hi} and \cite[Corollary 4.15]{SW}.
\begin{theorem}
\la{quilleneq}
$\mathrm{(i)}$ The categories $\cDGA_{k/k}$ and\, $\DGL_k$ have model structures where the weak-equivalences are the quasi-isomorphisms and the fibrations are the degreewise surjective maps.

$\mathrm{(ii)}$ The category $\DGLC_k$ $($resp., $\cDGC_{k/k}$$)$ admits a model structure, where the weak equivalences are the maps $f$ such that $\cb_{\mathtt{Comm}}(f)$ $($resp., $\cb_{\mathtt{Lie}}(f)$$)$ is a quasi-isomorphism and the cofibrations are degreewise monomorphisms.

$\mathrm{(iii)}$ For the above model structures, the pairs of functors~\eqref{cbbccom} and~\eqref{cbbcom} are Quillen equivalences.
\end{theorem}

Note that part (iii) says that the functors~\eqref{cbbccom} and~\eqref{cbbcom} induce derived
equivalences
\begin{eqnarray*}
\L \cb_{\mathtt{Comm}}\,:\,\Ho(\cDGC_{k/k}) \rightleftarrows \Ho(\DGL_k) \,:\, \R \bB_{\mathtt{Lie}}\,, \\
\L \cb_{\mathtt{Lie}}\,:\, \Ho(\DGLC_k) \rightleftarrows \Ho(\cDGA_{k/k}) \,:\,\R \bB_{\mathtt{Comm}} \,\text{.}
\end{eqnarray*}

\subsubsection{Relation to DG algebras and linear duality}
We now introduce the following functors on coalgebras: the co-abelianization functor
$\,(\,\mbox{--}\,)^{\nn} :\,\DGC_{k/k} \rar \cDGC_{k/k}\,$ assigning to a DG coalgebra $ C $
its maximal cocommutative DG subcoalgebra $ C^{\nn} \subseteq C $ (this functor is
dual to the abelianization functor \eqref{ab}); the universal co-enveloping coalgebra
functor $\,{\mathcal U}^c\,:\,\DGLC_k \rar \DGC_{k/k}\,$ dual to the universal
enveloping algebra functor $\,\mathcal U\,:\,\DGL_k \rar \DGA_{k/k}\,$; the Lie coalgebra functor
$\, {\mathcal Lie}^c\,:\,\DGC_{k/k} \rar \DGLC_k\,$ assigning to each $C$ the co-augmentation coideal $\bar{C}$ viewed as a DG Lie coalgebra (this functor is dual to the Lie algebra functor
$\, {\mathcal Lie}\,:\,\DGA_{k/k} \rar \DGL_k$ assigning to $A \in \DGA_{k/k}$ the augmentation ideal $\bar{A}$ viewed as a DG Lie algebra.
Their relationship to Quillen equivalences~\eqref{cbbccom} and~\eqref{cbbcom} is
summarized by the following theorem.

\begin{theorem}
\la{variouscbb}
In each of the following diagrams the square subdiagrams obtained by starting at any corner and mapping to the opposite corner commute (up to isomorphism).
\begin{equation} \la{cbb1}
\begin{diagram}[tight]
\DGC_{k/k} & \pile{\rTo^{\cb} \\ \lTo_{\bB}} & \DGA_{k/k} & & \DGC_{k/k} & \pile{\rTo^{\cb} \\ \lTo_{\bB}} & \DGA_{k/k}\\
  \uInto^{\mathit{in}} \dTo_{(\mbox{--})^{\nn}} &  & \uTo^{\mathcal U} \dTo_{{\mathcal Lie}} & & \dTo^{{\mathcal Lie}^c} \uTo_{{\mathcal U}^c} & & \dTo^{(\mbox{--})_{\nn}} \uInto_{\mathit{in}} \\
 \cDGC_{k/k} &  \pile{\rTo^{\cb_{\mathtt{Comm}}} \\ \lTo_{\bB_{\mathtt{Lie}}}}& \DGL_k &  & \DGLC_k & \pile{\rTo^{\cb_{\mathtt{Lie}}} \\ \lTo_{\bB_{\mathtt{Comm}}}}& \cDGA_{k/k}
\end{diagram}
\end{equation}
\end{theorem}

\begin{proof}
Let $C \in \cDGC_{k/k}$. Let $\mathfrak L\,V$ denote the free Lie algebra generated by a graded vector space $V$. Since $T_k\,V \,\cong\, {\mathcal U}(\mathfrak L V)$ as graded $k$-algebras,
we have an isomorphism of graded algebras
$$
\cb(C) \,\cong\, \mathcal U[\cb_{\mathtt{Comm}}(C)]\ .
$$
The fact that this isomorphism commutes with differentials follows from the fact that the coalgebra
$C$ is cocommutative. Hence, on the category of cocommutative DG coalgebras, we have an isomorphism of functors
\begin{equation}
\la{isofun1} \cb  \,\cong\, \mathcal U \circ \cb_{\mathtt{Comm}}\,\text{.}
\end{equation}
By adjunction, this gives an isomorphism
\begin{equation}
\la{isofun2}\bB_{\mathtt{Lie}} \circ {\mathcal Lie} \,\cong\, (\mbox{--})^{\nn} \circ \bB \,\text{.}
\end{equation}
which proves the result for the first diagram in~\eqref{cbb1}.
A similar argument shows that
\begin{equation}
\la{isofun3} (\mbox{--})_{\nn} \circ \cb \,\cong\, \cb_{\mathtt{Lie}} \circ {\mathcal Lie}^c\,\text{.}\end{equation}
which by adjunction (together with \eqref{isofun3}) gives an isomorphism of functors
on commutative DG algebras:
\begin{equation} \la{isofun4}
\bB \,\cong\, {\mathcal U}^c \circ \bB_{\mathtt{Comm}} \,\text{.}
\end{equation}
This establishes the desired result for the second diagram in~\eqref{cbb1}.
\end{proof}

The following theorem is an immediate consequence of~\cite[Theorem 4.17]{SW}, which explains the canonical linear dualities relating Lie algebraic and coalgebraic Quillen functors.
\begin{theorem}
\la{cbbdual}
$(\mathrm{i})$ For $\g \,\in\,\DGL_k$, there is a natural isomorphism
\begin{equation} \la{duality2}
\cb_{\mathtt{Lie}}(\g^{\ast})\,\cong\,\bB_{\mathtt{Lie}}(\g)^{\ast}\,\text{.}
\end{equation}
$(\mathrm{ii})$ For $\mfc \in \DGLC_k$, we have a natural isomorphism
\begin{equation} \la{duality3} \bB_{\mathtt{Lie}}(\mfc^{\ast})\,\cong\, \cb_{\mathtt{Lie}}(\mfc)^{\ast}\,\text{.} \end{equation}
\end{theorem}

\noindent

\textbf{Remark.} Note that $\bB_{\mathtt{Lie}}(\g)^{\ast}$ is precisely the complex of Lie {\it cochains} of $\g$ with trivial coefficients. In particular, $\H_{\bullet}[\bB_{\mathtt{Lie}}(\g)^{\ast}]\,=\, \H^{-\bullet}(\g;k)$ (where $\H^{\bullet}$ denotes Lie {\it cohomology}). Thus, the isomorphism~\ref{duality2} explicitly relates the homological and cohomological Chevalley-Eilenberg complexes.    In addition, Theorem~\ref{variouscbb} and Theorem~\ref{cbbdual} hold in the bigraded setting.
In that setting, $\, (\,\mbox{--}\,)^{\ast} $ means taking the bigraded dual.

\subsection{Derived representation schemes}

\subsubsection{Convolution Lie algebras}
For a fixed $\mfc \in \DGLC_k$ and $A \in \cDGA_k$, we define a Lie bracket on
$ \mathbf{Hom}(\mfc, A) $ by
$$
[f,\,g] := m\, \circ\, (f \otimes g) \,\circ\, ]\,\mbox{--}\,[
$$
where $\, m: A \otimes A \to A\,$ is the multiplication map on $A$, and
$\,]\,\mbox{--}\,[\,:\,\mfc \rar \mfc \otimes \mfc$ is the Lie cobracket on $\mfc$.
For $\,\mfc\in \DGLC_k\,$ fixed, this gives a functor
\begin{equation} \la{lieconv}
\mathbf{Hom}(\mfc,\, \mbox{--}\,):\,\cDGA_k \rar \DGL_k \,\text{.}
\end{equation}
which we call a convolution functor.

\subsubsection{The left adjoint functor} \la{leftadjlie}

For arbitrary elements $\xi, \eta$ in a DG Lie algebra $\mathfrak{a}$ and for $x \in \mfc$, let $( ]x[, \xi, \eta)$ denote the image of
$ x \otimes \xi \otimes \eta $ under the composite map
\begin{equation}
\la{cobrack}
\begin{diagram}
 \mfc  \otimes \mathfrak{a} \otimes \mathfrak{a} & \rTo^{\,\,\,]\mbox{--}[\,\,\,\,\,} & \mfc \otimes \mfc \otimes \mathfrak{a} \otimes \mathfrak{a} & \rTo^{\cong\,\,} & (\mfc \otimes \mathfrak{a} )^{\otimes 2} & \rOnto & \bSym^2(\mfc \otimes \mathfrak{a} ) & \rInto & \bSym_k(\mfc \otimes \mathfrak{a} ) \end{diagram}
\end{equation}
The following proposition describes the left adjoint for the Lie convolution algebra
functor $\mathbf{Hom}(\mfc, \mbox{--})$.

\bprop \la{lierep}
The functor~\eqref{lieconv} has a left adjoint
$\,(\,\mbox{--}\,)_{\mfc} \,:\,\DGL_k \rar \cDGA_k \,$, which is given by
$$
\mathfrak{a} \mapsto \mathfrak{a}_{\mfc} \,:=\, \bSym_k(\mfc \otimes \mathfrak{a}) / \langle x \otimes [\xi, \eta]  - (]x[, \xi, \eta) \rangle\ ,
$$
where $\,(]x[, \xi, \eta)\,$ is defined by \eqref{cobrack}.
\eprop
\begin{proof}
For $B \in \cDGA_k$ consider natural maps
\[\xymatrix@=0.6cm{
	\Hom_{\DGL_k}(\mathfrak{a}, \mathbf{Hom}(\mfc, B )) \ar@{^{(}->}[r] \ar[dd]^-{\simeq} &  \Hom_{\Com_k}(\mathfrak{a}, \mathbf{Hom}(\mfc, B)) \ar[d]^-{=} \\
	& \Hom_{\Com_k}(\mfc \otimes \mathfrak{a}, B)\ar[d]^-{\simeq}\\
	\Hom_{\cDGA_k}(\mathfrak{a}_{\mfc},\, B) \ar@{^{(}->}[r] & \Hom_{\cDGA_k}(\bSym_k(\mfc \otimes \mathfrak{a}), B)
}
\]
The map
\begin{equation} \la{homnatlie} \Hom_{\DGA_k}(\mathfrak{a}, \mathbf{Hom}(\mfc, B ))  \hookrightarrow \Hom_{\cDGA_k}(\bSym_k(\mfc \otimes \mathfrak{a}), B) \end{equation}
obtained by following the upper right part of the above diagram is explicitly given by
$$ (f\,:\,\mathfrak{a} \rar \mathbf{Hom}(\mfc, B)) \mapsto [\hat{f}\,:\,\bSym_k(\mfc \otimes_k \mathfrak{a}) \rar B\,,\,\,\,\, x \otimes \xi \mapsto {(-1)}^{|x||\xi|} f(\xi)(x)] \,\text{.}$$
Further, by a straightforward calculation,
$f$ is a DG Lie algebra homomorphism iff $f([\,\xi\,,\,\eta\,]) = [\,f(\xi)\,,\, f(\eta)]$ iff for all $\xi, \eta \in \mathfrak{a}$ and $x \in \mfc$,
$\hat{f}\left( x \otimes [\xi\,,\,\eta]-(]x[, \xi, \eta) \right)=0$.
This shows that the image of the map~\eqref{homnatlie} is precisely $\Hom_{\cDGA_k}(\mathfrak{a}_{\mfc}, B)$.
This proves the desired proposition.
\end{proof}
%

%
%

As a consequence of Proposition~\ref{lierep}, we obtain:

\begin{theorem} \la{dlierep}
The pair of functors $(\mbox{--})_{\mfc}\,:\,\DGL_k \rightleftarrows \cDGA_k\,:\, \mathbf{Hom}(\mfc,\mbox{--})$ is a Quillen pair. As a result, the functor $(\mbox{--})_{\mfc}$ has a (total) left derived functor
$$ \L (\mbox{--})_{\mfc}\,:\,\Ho(\DGL_k) \rar \Ho(\cDGA_k)\,,\,\,\,\, \mathfrak{a} \mapsto (Q\mathfrak{a})_{\mfc} $$
where $Q\mathfrak{a} \stackrel{\sim}{\rar} \mathfrak{a}$ is any cofibrant resolution in $\DGL_k$.
\end{theorem}

\begin{proof}
By Proposition~\ref{lierep} and~\cite[Remark 9.8]{DS}, it suffices to check that $\mathbf{Hom}(\mfc,\mbox{--})$ preserves degree-wise surjections and quasi-isomorphisms. This is obvious.
\end{proof}

The functor $\mathbf{Hom}(\mfc,\mbox{--})$ can be modified naturally to give a functor on {\it augmented} commutative DG algebras
$$ \mathbf{Hom}(\mfc,\mbox{--})\,:\,\cDGA_{k/k} \rar \DGL_k \,,\,\,\,\, A \mapsto \mathbf{Hom}(\mfc,\bar{A})\,\text{.}$$
The left adjoint $(\mbox{--})_{\mfc}$ of $\mathbf{Hom}(\mfc,\mbox{--})\,:\,\cDGA_{k/k} \rar \DGL_k$ is the functor assigning to each $\mathfrak{a} \in \DGL_k$ the commutative DG algebra $\mathfrak{a}_{\mfc}$ equipped with the canonical augmentation
$$\varepsilon\,:\,\mathfrak{a}_{\mfc} \rar k $$
corresponding to the $0 \in \Hom_{\DGL_k}(\mathfrak{a}, \mathbf{Hom}(\mfc,k))$ under the adjunction~\eqref{lieadj}.
As in Theorem~\ref{dlierep}, it is easy to verify that the pair of functors $(\mbox{--})_{\mfc}\,:\,\DGL_k \rightleftarrows \cDGA_{k/k} \,:\, \mathbf{Hom}(\mfc,\mbox{--})$ is Quillen. Hence, $(\mbox{--})_{\mfc}$ has a left derived functor
$$ \L (\mbox{--})_{\mfc}\,:\,\Ho(\DGL_k) \rar \Ho(\cDGA_{k/k})\,,$$
which, after applying the forgetful functor $\Ho(\cDGA_{k/k}) \rar \Ho(\cDGA_k)$ coincides with $\L (\mbox{--})_{\mfc}\,:\,\Ho(\DGL_k) \rar \Ho(\cDGA_{k})$.

\subsubsection{Derived representation schemes}

Let $\g$ be a finite-dimensional Lie algebra over $k$. Let $\mfc \,:=\,\g^{\ast}$, the dual Lie coalgebra. Then,
$$\mathbf{Hom}(\mfc, A) \,=\, \mathbf{Hom}(\g^{\ast}, A)\,\cong\, \g \otimes A \,=:\,\g(A)\,\text{.}$$
Thus, the commutative DG algebra $\mathfrak{a}_{\g}\,:=\,\mathfrak{a}_{\mfc}$ represents the functor
$$ \Rep_{\g}(\mathfrak{a})\,:\,\cDGA_k \rar \mathtt{Sets}\,,\,\,\,\, A \mapsto \Hom_{\DGL_k}(\mathfrak{a}, \g(A))\ ,
$$
that is, there is a natural isomorphism of sets
\begin{equation} \la{lieadj}
\Hom_{\cDGA_k}(\mathfrak{a}_{\g}, A) \,\cong\, \Hom_{\DGL_k}(\mathfrak{a}, \g(A))\,\text{.}
\end{equation}

As in the associative case, we now define
$$
\DRep_{\g}(\mbox{--})\,:=\,\L (\mbox{--})_{\g}\,:\,\Ho(\DGL_k) \rar \Ho(\cDGA_k)\,\text{.}$$
We call $\DRep_{\g}(\mathfrak{a})$ {\it the derived representation scheme parametrizing representations of $\mathfrak{a}$ in $\g$}. Further, if $G$ is a Lie group whose Lie algebra is $\g$, $G$ acts (via the adjunction~\eqref{lieadj}) on $\mathfrak{a}_{\g}$ by automorphisms for any $\mathfrak{a} \in \DGL_k$. One can therefore, form the subfunctor
$$(\mbox{--})_{\g}^G\,:\,\DGL_k \rar \cDGA_k\,,\,\,\,\, \mathfrak{a} \mapsto \mathfrak{a}_{\g}^G\, $$
of $(\mbox{--})_{\g}$. An argument using Brown's lemma similar to the proof of~\cite[Theorem 2.6]{BKR} shows that the functor
$(\mbox{--})_{\g}^G$ has a total left derived functor $$\DRep_{\g}(\mbox{--})^G\,:=\,\L (\mbox{--})_{\g}^G\,:\,\Ho(\DGL_k) \rar \Ho(\cDGA_k)\,\text{.}$$
We define the representation homologies
\begin{eqnarray*}
\H_{\bullet}(\mathfrak{a},\g) \,:= \,\H_{\bullet}(\DRep_{\g}(\mathfrak{a}))\,\,\,,\,\,\,
 \H_{\bullet}(\mathfrak{a},\g)^G \,:=\, \H_{\bullet}(\DRep_{\g}(\mathfrak{a})^G)\,\text{.}
\end{eqnarray*}
More generally, $\g$ acts (via  the adjunction~\eqref{lieadj}) on  $\mathfrak{a}_{\g}$ by derivations. One can therefore, form the functor
$$ (\mbox{--})_{\g}^{\ad \,\g} \,:\,\DGL_k \rar \cDGA_k\,,\,\,\,\, \mathfrak{a} \mapsto \mathfrak{a}_{\g}^{\ad\,\g}\,$$
Again, an argument paralleling the proof of ~\cite[Theorem 2.6]{BKR} shows that the functor $ (\mbox{--})_{\g}^{\ad\, \g} $  has a total left derived functor $$ \DRep_{\g}(\mbox{--})^{\ad\,\g}\,:=\, \L  (\mbox{--})_{\g}^{\ad\, \g} \,:\,\Ho(\DGL_k) \rar \Ho(\cDGA_k)\,\text{.}$$
We define
 $$  \H_{\bullet}(\mathfrak{a},\g)^{\ad\,\g} \,:=\,
\H_{\bullet}[\DRep_{\g}^{\ad\,\g}(\mathfrak{a})]\,\text{.}$$
Note that if $\g$ is the Lie algebra of a reductive Lie group $G$,  the functors $(\mbox{--})_{\g}^G$ and $(\mbox{--})_{\g}^{\ad\,\g}$ coincide. Hence, in this situation, their derived functors coincide as well.

The discussion in Section~\ref{leftadjlie} points out that $\DRep_{\g}(\mathfrak{a})$ can be viewed as an object in $\Ho(\cDGA_{k/k})$ (rather than $\Ho(\cDGA_k)$. In the same way, $\DRep_{\g}(\mathfrak{a})^{\ad\,\g}$ can be viewed as an object in $\Ho(\cDGA_{k/k})$. Similarly, if $G$ is a Lie group whose Lie algebra is $\g$, one can consider
$\DRep_{\g}(\mathfrak{a})^G$ as an object in $\Ho(\cDGA_{k/k})$. In this case, $\DRep_{\g}(\mathfrak{a})^G\,\cong\,\DRep_{\g}(\mathfrak{a})^{\ad\,\g}$.

\subsubsection{Representation homology and Lie cohomology}

The following proposition leads to the main result of this section (Theorem~\ref{drepg}).

\bprop
\la{lierepcobar}
For any $\mfc \in \DGLC_k$, the following diagram commutes (upto isomorphism of functors):
$$
\begin{diagram}[small, tight]
\cDGC_{k/k} & \pile{ \rTo^{\cb_{\mathtt{Comm}}} \\ \lTo_{\bB_{\mathtt{Lie}}} }  & \DGL_k\\
  \dTo^{\mfc \dotimes \mbox{--}} & & \dTo^{(\mbox{--})_{\mfc}} \uTo_{\mathbf{Hom}(\mfc, \mbox{--})}\\
  \DGLC_k & \pile{ \rTo^{\cb_{\mathtt{Lie}}} \\ \lTo_{\bB_{\mathtt{Comm}}}  }& \cDGA_{k/k}
\end{diagram}$$
where $\mfc \dotimes C\,:=\,\mfc \otimes \bar{C}$ for any $C \in \cDGC_{k/k}$. Thus, there is an isomorphism of functors
\begin{equation} \la{isofunrep} (\mbox{--})_{\mfc} \circ \cb_{\mathtt{Comm}}\,\cong\, \cb_{\mathtt{Lie}} \circ (\mfc \dotimes \mbox{--}) \,=\, \CE^c \circ (\mfc \dotimes \mbox{--})\,\text{.} \end{equation}
\eprop

\begin{proof}
For any $C \in \cDGC_{k/k}$ and $\mathfrak L \in \DGL_k$, let $\mathtt{Tw}(C, \mathfrak L)$ denote the set of Maurer-Cartan elements (i.e, elements satisfying $d\alpha+\frac{1}{2}[\alpha,\alpha]=0$) in the DG Lie algebra $\mathbf{Hom}(\bar{C}, L)$. Similarly, for $\mathfrak L^c \in \DGLC_k$ and $A \in \DGA_{k/k}$, $\mathtt{Tw}(\mathfrak L^c, A)$ shall denote the set of Maurer-Cartan elements in the DG Lie algebra $\mathbf{Hom}(\mathfrak L^c, \bar{A})$. Now, for any $A \in \cDGA_{k/k}$ and $C \in \cDGC_{k/k}$, we have
\begin{eqnarray*}
 \Hom_{\cDGA_{k/k}}(\cb_{\mathtt{Comm}}(C)_{\mfc}, A) &\cong & \Hom_{\DGL_k}(\cb_{\mathtt{Comm}}(C), \mathbf{Hom}(\mfc, \bar{A}))\\
 & \cong & \mathtt{Tw}( C \,,\,\mathbf{Hom}(\mfc, \bar{A}))\\
 & \cong & \mathtt{Tw} ( \mfc \dotimes C, A)\\
 & \cong & \Hom_{\cDGA_{k/k}}(\CE^c(\mfc \dotimes C), A)\,\text{.}
\end{eqnarray*}
The first isomorphism above is Proposition~\ref{lierep}, the second is from~\cite{Hi}, the third is because the DG Lie algebras $\mathbf{Hom}(\bar{C}, \mathbf{Hom}(\mfc, A))$ and $\mathbf{Hom}(\mfc \dotimes C, \bar{A})$ are isomorphic by the standard hom-tensor duality and the fourth is from arguments similar to those in~\cite{Hi} proving the second. ~\eqref{isofunrep} follows from this by Yoneda's lemma. The rest of the desired proposition follows from~\eqref{isofunrep} by adjunction.
\end{proof}

\bthm
\la{drepg}
$(a)$ Suppose that $\mfc \,=\,\g^{\ast}$ for some Lie algebra $\g$. If $\cb_{\mathtt{Comm}}(C) \stackrel{\sim}{\rar} \mathfrak{a}$ is a quasi-isomorphism for some $C \in \cDGC_{k/k}$, then
$$
\DRep_{\g}(\mathfrak{a}) \, \cong\ \, \CE^c(\g^{\ast}(\bar{C}); k)\,, \,\,\,\,
\DRep_{\g}(\mathfrak{a})^{\ad\,\g} \,\cong \, \CE^c(\g^{\ast}(C), \g^{\ast}; k) \,\text{.}
$$
In particular,
\begin{eqnarray*}
\H_{\bullet}(\mathfrak{a}, \g) \,\cong \,  \H_{\bullet}(\g^{\ast}(\bar{C}) ; k) \,,\,\,\,\,
 \H_{\bullet}(\mathfrak{a}, \g)^{\ad\,\g} \, \cong \, \H_{\bullet}(\g^{\ast}(C), \g^{\ast}; k)\,\text{.}
\end{eqnarray*}
$(b)$ Suppose, in addition, that $\g$ and $C$ are finite-dimensional or bigraded. Let $A = C^{\ast}$ (with bigraded duals being taken in the bigraded setting). Then,
 \begin{eqnarray*}
\DRep_{\g}(\mathfrak{a}) \, \cong\ \, \CE(\g(\bar{A}); k)^{\ast}\,, \,\,\,\,
\DRep_{\g}(\mathfrak{a})^{\ad\,\g} \, \cong \, \CE(\g(A), \g; k)^{\ast} \,\text{.}
\end{eqnarray*}
In particular,
\begin{eqnarray*}
\H_{\bullet}(\mathfrak{a}, \g) \, \cong \, \H^{-\bullet}(\g(\bar{A}) ; k) \,,\,\,\,\,
 \H_{\bullet}(\mathfrak{a}, \g)^{\ad\,\g} \,\cong \, \H^{-\bullet}(\g(A), \g; k)\,\text{.}
\end{eqnarray*}
Here, in the bigraded setting, ``Lie cohomology" means ``continuous Lie cohomology".
\ethm

\begin{proof}
That $\,\DRep_{\g}(\mathfrak{a}) \cong \CE^c(\g^{\ast}(\bar{C}); k)\,$ is immediate from
Proposition~\ref{lierepcobar}. That
$\,\DRep_{\g}(\mathfrak{a})^{\ad\,\g} \cong  \CE^c(\g^{\ast}(C), \g^{\ast}; k)$ then follows from the fact that $\CE^c(\g^{\ast}(C), \g^{\ast}; k) = \CE^c(\g^{\ast}(\bar{C}); k)^{\ad\,\g}$. This proves $(a)$. Part $(b)$ follows from $(a)$ and~\eqref{duality2}.
\end{proof}

\noindent
\textbf{Remark.} Recall that, by~\eqref{isofun1},
$\, \cb(C)\,\cong\,\mathcal U(\cb_{\mathtt{Comm}}(C))\,$ for any $ C \in \cDGC_{k/k}$.

\vspace{1ex}

Now, take $\g\,=\,\gl_n(k)$ and $\mfc\,=\,\g^{\ast} \in \DGLC_k$ and $M \,=\,\M_n(k)^{\ast} \in \DGC_k$. The following proposition clarifies the relation between derived representation schemes of Lie algebras and their associative counterparts.

\bprop \la{dreplieas} Let $C \in \cDGC_{k/k}$. Then there is a natural isomorphism of commutative
DG algebras
\begin{equation}
\la{drla1} \cb(C)_n \,\cong\, \cb_{\mathtt{Comm}}(C)_{\gl_n} \,\text{.}
\end{equation}
Hence, for any DG Lie algebra $\,\mathfrak{a}\,$,
\begin{equation}\la{drla2}  \DRep_n(\mathcal U(\mathfrak{a}))\,\cong\, \DRep_{\gl_n}(\mathfrak{a})\,\text{.}\end{equation}
\eprop

\begin{proof}
For any $B \in \cDGA_k$, we have
\begin{eqnarray*}
\Hom_{\cDGA_k}(\cb(C)_n,B) & \cong & \Hom_{\DGA_k}(\cb(C), \M_n(B)) \\
                                             & \cong  &  \Hom_{\DGA_k} (\mathcal U(\cb_{\mathtt{Comm}}(C)), \M_n(B))\\
                                              & \cong & \Hom_{\DGL_k}(\cb_{\mathtt{Comm}}(C), \gl_n(B))\\
                                               & \cong & \Hom_{\cDGA_k}(\cb_{\mathtt{Comm}}(C)_{\gl_n}, B)
\end{eqnarray*}
The isomorphism~\eqref{drla1} now follows from Yoneda's lemma. To prove~\eqref{drla2}, note that for any $\mathfrak{a} \in \DGL_k$, one can find $C \in \cDGC_{k/k}$ such that $\cb_{\mathtt{Comm}}(C) \rar \mathfrak{a}$ is a cofibrant resolution in $\DGL_k$ (for example, $C= \bB_{\mathtt{Lie}}(\mathfrak{a})$). Then, $\cb(C)\,\cong\,\mathcal U(\cb_{\mathtt{Comm}}(C)) \rar \mathcal U(\mathfrak{a})$ is a cofibrant resolution of $\mathcal U(\mathfrak{a})$ in $\DGA_k$. Hence,
$$ \DRep_n(\mathcal U(\mathfrak{a})) \,\cong\, \cb(C)_n \,\cong\, \cb_{\mathtt{Comm}}(C)_{\gl_n} \,\cong\, \DRep_{\gl_n}(\mathfrak{a})\,\text{.}
$$
This completes the proof of the proposition.
\end{proof}

\section{Derived  Harish-Chandra homomorphism and Drinfeld traces}
\la{sec7}
The aim of this section is to construct the derived Harish-Chandra
homomorphism and trace maps for representation schemes of Lie
algebras. Our starting point is the observation of Section 4.2
interpreting the derived HC homomorphism for associative algebras
in terms of Chevalley-Eilenberg complexes.

Throughout this section,
$\g$ will denote a finite-dimensional reductive Lie algebra,
$ \h \subset \g $ its Cartan subalgebra and $ W $ the Weyl
group of $\g$. Recall that in this case, if $V$ is any $k$-vector space with an action of the reductive group $G$ whose Lie algebra is $\g$,
$V^G \,\cong\,V^{\ad\,\g}$.

\subsection{The derived Harish-Chandra homomorphism}
The natural inclusion $ \h \into \g $ defines a homomorphism of Lie coalgebras
$ \g^{\ast} \onto \h^{\ast} $ and hence, for any $ C \in \mathtt{DGCC}_{k/k} $, a morphism of commutative
DG algebras
\begin{equation}
\la{derHCg}
\Phi_{\g}(C) \,:\,\C^c(\g^{\ast}(C),\g^{\ast};k) \rar \C^c(\h^{\ast}(C),\h^{\ast};k)\,\text{.}
\end{equation}
The following proposition is a generalization of Lemma~\ref{resthom}.
\bprop
\la{Winv}
The image of $\Phi_{\g}(C) $ lies in the DG subalgebra $\C^c(\h^{\ast}(C),\h^{\ast};k)^W$ of chains that are invariant under the action of the Weyl group $W$ of $\g$.
\eprop

\begin{proof}
As $\h $ is an abelian Lie algebra, we have an isomorphism
$ \C^c(\h^{\ast}(C),\h^{\ast}(k);k) \,\cong\, \C^c(\h^{\ast}(\bar{C});k)$.
The map  $\Phi_{\g}(C) $ is thus the restriction of the natural map
$ \C^c(\g^{\ast}(C); k) \rar \C^c(\h^{\ast}(\bar{C});k)$ to
$ \C^c(\g^{\ast}(C), \g^{\ast}; k) \,\cong\,
\C^c(\g^{\ast}(\bar{C});k)^{{\rm ad}\,\g} \cong \C^c(\g^{\ast}(\bar{C});k)^{G} $, where
$ G $ is the Lie group attached to $ \g $. Now, let $N$ denote the normalizer
of $\h $ in $ G $, so that there is a surjective group homomorphism $ N \onto W $.
Since $ W $ acts naturally on $ \h^* $, so does $N$. Thus, $N$ acts on
$ \h^{\ast}(\bar{C})$ as well, making $\h^{\ast}(\bar{C})$ a DG-Lie coalgebra with $N$-action.
This, in turn, induces an $N$-action on the commutative DG algebra $\C^c(\h^{\ast}(\bar{C});k)$. On the other hand, the adjoint action of $ G $ on $ \g $ makes $ g^{\ast}(\bar{C})$
a DG Lie coalgebra with $ G $-action (and hence, $N$) action. Thus, the commutative DG algebra
$\C^c(g^{\ast}(\bar{C});k)$ acquires
a $G$ (and hence, $N$) action. Since the map $ \g^{\ast} \onto \h^{\ast} $ is $N$-equivariant, the map
$\g^{\ast}(\bar{C}) \twoheadrightarrow \h^{\ast}(\bar{C})$ is $N$-equivariant as well. Therefore, the map $\C^c(\g^{\ast}(\bar{C});k)
\rar \C^c(\h_n^{\ast}(\bar{C});k)$ is $N$-equivariant. Since any element of $\C^c(\g^{\ast}({C}),\g^{\ast}(k);k)$ is $G$-invariant
(and hence, $N$-invariant), any element in the image of $\Phi_\g(C) $ is $N$-invariant (and hence, $W$-invariant).
\end{proof}

Thus, we have a morphism of commutative DG algebras
$$
\Phi_{\g}(C)\,:\,\C^c(\g^{\ast}(C),\g^{\ast};k) \rar \C^c(\h^{\ast}(C),\h^{\ast};k)^W\,,
$$
which we call the {\it derived Harish-Chandra homomorphism}.
Suppose that there exists a quasi-isomorphism $\cb_{\mathtt{Comm}}(C) \stackrel{\sim}{\twoheadrightarrow} \mathfrak{a}$ for some Lie algebra $\mathfrak{a}$. By Theorem~\ref{drepg}, the derived Harish-Chandra homomorphism can be viewed as a map in $\Ho(\cDGA_{k/k})$
$$\Phi_{\g}(C)\,:\, \DRep_{\g}(\mathfrak{a})^G \rar \C^c(\h^{\ast}(C),\h^{\ast};k)^W\,\text{.}$$
It follows that the derived Harish-Chandra homomorphism induces the map
$$
\H_{\bullet}(\Phi_{\g}) \,:\,\H_{\bullet}(\mathfrak{a}, \g)^{G} \rar \H_{\bullet}(\h^{\ast}(C),\h^{\ast};k)^W \,\text{.}
$$
If $C$ has zero differential, then $\C^c(\h^{\ast}(C),\h^{\ast};k)$ has also zero differential, and
hence in this case,
$\H_{\bullet}(\Phi_{\g})\,$ maps $\,\H_{\bullet}(\mathfrak{a},\g)^G $ to
$\bSym_k(\h^{\ast} \otimes \bar{C}[-1])^W$.

\begin{example} \la{1dliehc}
Let $k=\c$ and let $\mathfrak{a}\,:=\,\c.x$ be the one-dimensional Lie algebra
with $x$ having weight $1$ and homological degree $0$. The cocommutative coalgebra
$C\,:=\,\bSym^c(\mathfrak{a}[1]) = \Lambda^c(sx)$ is Koszul dual to $\mathfrak{a}$. In this case,
$\g^{\ast}(C)\,\cong\,\g^{\ast}.(sx) \oplus \g^{\ast}$ and
$\C^c(\g^{\ast}(C), \g^{\ast}; \c) \,=\, \Sym(\g^{\ast})^{G}$.
In this case, the derived Harish-Chandra homomorphism becomes the Chevalley isomorphism
$$
\Sym(\g^{\ast})^{G} \stackrel{\sim}{\to} \Sym(\h^{\ast})^W\,\text{.}
$$
\end{example}
\subsection{Drinfeld homology}
We now proceed with constructing the analogues of the trace maps~\eqref{trm3} in the case of Lie algebras. As the first step, we introduce
an appropriate version of cyclic homology for Lie algebras that
will relate to representation homology via the trace maps.

\subsubsection{ }
In \cite{Dr}, Drinfeld introduced the functor
$$
\lambda :\,\DGL_k \to \Com_k\ ,\quad
\mathfrak{a} \mapsto
\Sym^2(\mathfrak{a})/\langle [x,y]\cdot z -x \cdot [y,z]\,:\,x,y,z \in \mathfrak{a}\rangle
$$
that assigns to a Lie algebra $ \mathfrak{a} $ (the target of) the universal invariant bilinear
form on $ \mathfrak{a} $. As shown in \cite{GK}, this functor plays a role of the cyclic functor \eqref{cycf} on the category of Lie algebras: its left derived
$\, \mathfrak L \lambda \,$ exists and defines the analogue of cyclic homology for Lie algebras
({\it cf.} \cite[Theorem (5.3)]{GK}).

More generally, extending Drinfeld's construction, for an integer $ d \ge 1 $, we define
$$
\lambda^{(d)}\,:\,\DGL_k \rar \Com_k \ ,\quad
\mathfrak{a} \mapsto \Sym^d(\mathfrak{a})/[\mathfrak{a}, \Sym^d(\mathfrak{a})]\ .
$$
This functor assigns to a Lie algebra $ \mathfrak{a} $ (the target of) the universal invariant multilinear form of degree $d$ on $ \mathfrak{a} \,$; in particular, for $d=2$, we have
$\,\lambda^{(2)} = \lambda\,$.

Note that the symmetric invariant $d$-multilinear forms $\g \times \ldots \times \g \rar k$ are in one-to-one correspondence with linear maps $\lambda^{(d)}(\g) \rar k$. To be precise, the nondegenerate pairing
$$ \Sym^d(\g) \times \Sym^d(\g^{\ast}) \rar k $$
induces a nondegenerate pairing
$$ \lambda^{(d)}(\g) \times \Sym^d(\g^{\ast})^{\ad\,\g} \rar k \,\text{.}$$

The next theorem generalizes the result of \cite[Theorem~(5.3)]{GK} in the case of
the Lie operad.
\begin{theorem} \la{dlambda}
For each $d \ge 1 $, the functor $\,\lambda^{(d)}\,$ has a (total) left derived functor given by
$$ \L \lambda^{(d)}:\,\Ho(\DGL_k) \rar \Ho(\Com_k)\,,\,\,\,\, \mathfrak{a} \mapsto \lambda^{(d)}(\mathfrak L) \ ,
$$
where $\mathfrak L \stackrel{\sim}{\rar} \mathfrak{a}$ is a cofibrant resolution of
$ \mathfrak{a}$ in $\DGL_k$.
\end{theorem}
\begin{proof}
Suppose that $\mathfrak L \in \DGL_k$ is cofibrant and that $f,g\,:\,\mathfrak L \rar \mathfrak{a}$ are homotopic. Then, there exists $h\,:\,\mathfrak L \rar \mathfrak{a} \otimes k[t,dt]$ such that $h(0)=f$ and $h(1)=g$. Here, $\mathrm{deg}(t)=0$ and $d(t)\,:=\,dt$ and $h(a)$ denotes postcomposition of $h$ with the map $\id_{\mathfrak{a}}\otimes \mathrm{ev}_a$ where $\mathrm{ev}_a\,:\,k[t,dt] \rar k$ is the map taking $t$ to $a$ and $dt$ to $0$ for any $a \in k$. Note that for any $B \in \cDGA_k$, one has natural maps in $\Com_k $
\begin{equation} \la{mnatcalg} \lambda^{(d)}(\mathfrak{a} \otimes B) \rar \lambda^{(d)}(\mathfrak{a}) \otimes B\,\text{.}
\end{equation}
One therefore, has a map
$H\,:\,\lambda^{(d)}(\mathfrak L) \rar \lambda^{(d)}(\mathfrak{a}) \otimes k[t,dt]$
given by the composition
$$\begin{diagram} \lambda^{(d)}(\mathfrak L) & \rTo^{\lambda^{(d)}(h)} & \lambda^{(d)}(\mathfrak{a} \otimes k[t,dt]) & \rTo^{\eqref{mnatcalg}} &  \lambda^{(d)}(\mathfrak{a}) \otimes k[t,dt] \end{diagram}\,\text{.}$$
Clearly, $H(0)\,=\,\lambda^{(d)}(f)$ and $H(1)\,=\,\lambda^{(d)}(g)$. It follows that the maps $\lambda^{(d)}(f)$ and $\lambda^{(d)}(g)$ are homotopic in $\Com_k$.

Now, if $f\,:\,\mathfrak L \rar \mathfrak L'$ is a weak equivalence between cofibrant objects in $\DGL_k$, there exists a $g\,:\,\mathfrak L' \rar \mathfrak L$ in $\DGL_k$ such that
$fg$ and $gf$ are homotopic to the respective identities. It follows that $\lambda^{(d)}(fg)$ and $\lambda^{(d)}(gf)$ are homotopic to the respective identities as well. Hence,
$\lambda^{(d)}(f)$ is a quasi-isomorphism. In other words, the functors $\lambda^{(d)}$ take weak equivalences between cofibrant objects to weak equivalences. The desired theorem now follows from
Brown's lemma ({\it cf.}~\cite[Lemma 9.9]{DS}).
\end{proof}

We let  $\,\HC^{(d)}_{\bullet}(\mathtt{Lie},\mathfrak{a})\,$ denote the homology
$\,\H_{\bullet}[\L \lambda^{(d)}(\mathfrak{a})] \,$ and refer to it as {\it Drinfeld homology}. 
Note that $ \lambda^{(1)} $ is just the abelianization functor: $\,\mathfrak{a} \mapsto \mathfrak{a}/[\mathfrak{a},\mathfrak{a}]\,$,
and hence $\,\HC^{(1)}_{\bullet}(\mathtt{Lie},\mathfrak{a}) \cong \H_{\bullet + 1}(\mathfrak{a}; k)\,$ 
for any Lie algebra $ \mathfrak{a} $ (see, e.g., \cite[Example~1, Sect.2.6]{BFR}). For $d=2$,
 $\,\HC^{(2)}_{\bullet}(\mathtt{Lie},\mathfrak{a})\,$ is precisely the Lie cyclic
homology introduced in \cite{GK} and denoted $ \mathrm{HA}_{\bullet}(\mathtt{Lie}, \mathfrak{a}) $ in that paper. In general,
the meaning of the homology groups $\,\HC^{(d)}_{\bullet}(\mathtt{Lie},\mathfrak{a})\,$  is clarified by the following theorem which is one of the main results of
this section.
\begin{theorem} \la{hodgeuel}
Let $\mathfrak{a} \in \DGL_k$. The reduced cyclic homology of the universal
enveloping algebra $ \mathcal U(\mathfrak{a}) $ of the Lie algebra $\mathfrak{a}$ has a natural Hodge-type decomposition
\begin{equation} \la{hodged1} \rHC_{\bullet}[\mathcal U(\mathfrak{a})]\,\cong\, \bigoplus_{d= 1}^{\infty} \HC^{(d)}_{\bullet}(\mathtt{Lie},\mathfrak{a}) \,\text{.}\end{equation}
\end{theorem}

\begin{proof}
Let $ C \in \cDGC_{k/k} $ be a coalgebra Koszul dual to the Lie algebra
$\mathfrak{a}$ (for example, $\,C = \bB_{\mathtt{Lie}}(\mathfrak{a})$). Then we have
a cofibrant resolution  $\,\cb_{\mathtt{Comm}}(C) \stackrel{\sim}{\rar} \mathfrak{a}\,$
in $\DGL_k$.
For a graded $k$-vector space $V$, there are natural isomorphisms of graded vector spaces
\begin{eqnarray*}
T_k(V)_{\n} \, \cong \, T_k(V)/(k+[V, T_k(V)]) \cong
                      T_k(V)/(k+[L_k(V), T_k(V)])\ ,
\end{eqnarray*}
where $ L_k(V) \subset T_k(V)$ is the free (graded) Lie algebra generated by $V$.
It follows that
\begin{equation} \la{cyclie}
\cb(C)_{\n}\,\cong\, \cb(C)/[\cb_{\mathtt{Comm}}(C), \cb(C)]
\end{equation}
as complexes of $k$-vector spaces. By~\eqref{isofun1}, $\cb(C)\,\cong\,\mathcal{U}(\cb_{\mathtt{Comm}}(C))$. On the other hand, since $\,\cb_{\mathtt{Comm}}(C)\,$ is a DG Lie algebra,
we have an isomorphism of DG $\cb_{\mathtt{Comm}}(C)$-modules
\begin{equation} \la{symmap}
\bSym_k[\cb_{\mathtt{Comm}}(C)] \stackrel{\sim}{\rar} \mathcal{U}_k[\cb_{\mathtt{Comm}}(C)]
\end{equation}
given by the symmetrization map. Therefore, writing
$$
\Sym^d[\cb_{\mathtt{Comm}}(C)]_\n \,:=\,
\frac{\Sym^d(\cb_{\mathtt{Comm}}(C))}{[\cb_{\mathtt{Comm}}(C), \Sym^d(\cb_{\mathtt{Comm}}(C))]}\ ,
$$
we get the following decomposition
\begin{equation}
\la{decomp}
\cb(C)_{\n} \,\cong\, \bigoplus_{d=1}^{\infty} \,
\Sym^d[\cb_{\mathtt{Comm}}(C)]_\n\ .\end{equation}
Note that $\cb(C) \stackrel{\sim}{\rar} \mathcal U(\mathfrak{a})$ is a cofibrant resolution in $\DGA_{k/k}$. By \cite[Proposition 3.1]{BKR}), we have
$$
\H_{\bullet}[\cb(C)_{\n}] \,\cong\, \rHC_{\bullet}(\mathcal U(\mathfrak{a}))\,\text{.}$$
On the other hand,
$$ \H_{\bullet}[\Sym^d(\cb_{\mathtt{Comm}}(C))_\n]\,\cong\, \HC^{(d)}_{\bullet}(\mathtt{Lie},\mathfrak{a}) $$
since $C$ is Koszul dual to $\mathfrak{a}$. This proves the desired theorem.
\end{proof}

\subsubsection{Lie-Hodge decomposition}
\la{LieHodge}
Let $\mathfrak{a} \in \DGL_k$ and let $C \in \cDGC_{k/k}$ be Koszul dual to $\mathfrak{a}$. By~\eqref{isofun1}, $\cb(C)\,\cong\,\mathcal U(\cb_{\mathtt{Comm}}(C))$. From this isomorphism, $\cb(C)$ acquires the structure of a (primitively generated)  cocommutative DG Hopf algebra whose DG Lie algebra of primitives is $\cb_{\mathtt{Comm}}(C)$. Let $m_p\,:\,\cb(C)^{\otimes p} \rar \cb(C)$ denote the $p$-fold product and let $\Delta^p\,:\,\cb(C) \rar \cb(C)^{\otimes p}$ denote the $p$-fold coproduct. For each $p \geq 2$, define the Adams operation
$$ \psi^p\,:= \, m_p \circ \Delta^p\,:\,\cb(C) \rar \cb(C)\,\text{.}$$
Note that $\psi^p \circ \psi^q\,=\,\psi^{pq}$. The following proposition is dual to~\cite[Propositions 5.3.4-5.3.6]{FT1}.

\bprop
\la{adamscyc}
The Adams operations $\psi^p\,,\,\,p \geq 2$ descend to Adams operations
$$\psi^p\,:\,\cb(C)_{\n} \rar \cb(C)_{\n}  \,,\,\,p \geq 2 \,\text{.}$$
\eprop

It is verified without difficulty that on the image of $\Sym^d(\cb_{\mathtt{Comm}}(C))$ in $\cb(C)$ under the symmetrization map~\eqref{symmap}, $\psi^p$ coincides with multiplication by $p^d$.
Therefore,
\bprop
\la{eigenvalue}
 $\psi^p$ acts on the direct summand
$ \Sym^d[\cb_{\mathtt{Comm}}(C)]_\n $ of \eqref{decomp} by multiplication by $p^d$.
\eprop
\begin{corollary} \la{adamseival}
There are Adams operations
$$\psi^p\,:\,\rHC_{\bullet}\left[\mathcal U(\mathfrak{a})\right] \rar \rHC_{\bullet}\left[\mathcal U(\mathfrak{a})\right]\,,\,\,p \geq 2 \,\text{.}$$
Further, $\HC^{(d)}_{\bullet}(\mathtt{Lie},\mathfrak{a})$ is precisely the (graded) eigenspace corresponding to the eigenvalue $p^d$ of $\psi^p$ for each $p \geq 2$.
\end{corollary}

Corollary~\ref{adamseival} justifies referring to \eqref{hodged1} as a Hodge decomposition. When $C$ is the $k$-linear dual of a smooth commutative algebra $A$, the decomposition~\eqref{hodged1} can be related to the Hodge decomposition of the cyclic cohomology $\rHC^{\bullet}(A)$ by the following proposition, which is an immediate consequence of~\cite[Corollary 6.5.1]{FT1}.

\bprop \la{phodgedual}
Let $ \mathfrak{a} \in \DGL_k $, and let $\, C \in \cDGC_{k/k} $ be a cocommutative coalgebra Koszul dual to $ \mathfrak{a} $. Assume that $ C \, \cong \,A^{\ast}$ for some {\rm smooth} commutative $k$-algebra $A$. Then
$$ \HC^{(d)}_{\bullet}(\mathtt{Lie},\mathfrak{a})\,\cong\, \rHC^{-\bullet}_{(d-1)}(A)[-1]\,\text{.}$$
\eprop

In particular, $\HC^{(2)}_{\bullet}(\mathtt{Lie}, \mathfrak{a})$ is isomorphic (up to a shift) to the Harrison cohomology of $A$. More generally, for $C \in \cDGC_{k/k}$, one can define a Hodge decomposition for $\rHC_{\bullet}(C)$ dual to the decomposition defined in~\cite[Theorem 4.6.7]{L}. When $C\,=\,A^{\ast}$, this Hodge decomposition coincides with that on $\rHC^{-\bullet}(A)$ (after the obvious identification is made). Proposition~\ref{phodgedual} therefore shows that the homology of the direct summand $\,\Sym^d(\cb_{\mathtt{Comm}}(C))_\n[1]\,$ of $ \cb(C)_{\n}[1]$ should be denoted by $\rHC_{\bullet}^{(d-1)}(C)$. Thus,
\begin{equation} \la{drhodge} \HC^{(d)}_{\bullet}(\mathtt{Lie},\mathfrak{a})\,\cong\, \rHC_{\bullet}^{(d-1)}(C)[-1]\,\text{.} \end{equation}

\noindent \textbf{Remark.} Proposition~\ref{phodgedual} holds in the bigraded situation as well. In this case, $\rHC^{\bullet}(A)$ should be interpreted as the cohomology of the bigraded dual of the reduced cyclic chain complex of $A$.

\subsection{Drinfeld trace maps}
\la{secdrinfeldtr}
Let $\g$ be a finite-dimensional reductive Lie algebra over $k$. The adjunction~\eqref{lieadj} gives a universal representation
$$ \pi_{\g}\,:\,\mathfrak{a} \rar \g \otimes \mathfrak{a}_{\g}  $$
for any $\mathfrak{a}\,\in\,\DGL_{k}$. Let $\mathfrak{L} \stackrel{\sim}{\rar} \mathfrak{a}$ be a cofibrant resolution in $\DGL_{k}$. For any $d \geq 1$, consider the composite map
 \begin{equation} \la{trdg} \begin{diagram} \lambda^{(d)}(\mathfrak L) & \rTo^{\lambda^{(d)}(\pi_{\g})} & \lambda^{(d)}(\g \otimes \mathfrak L_{\g}) & \rTo^{\eqref{mnatcalg}} & \lambda^{(d)}(\g) \otimes \mathfrak L_{\g} \end{diagram}\,\,\text{.}\end{equation}
\bprop \la{trdginv}
The image of the composite map~\eqref{trdg} lies in $ \lambda^{(d)}(\g) \otimes \mathfrak L_{\g}^{\ad\,\g} $.
\eprop

\begin{proof}
Equip $\mathfrak L$ with the trivial $\g$-action. Then, $\pi_{\g}\,:\,\mathfrak L \rar \g \otimes \mathfrak L_{\g} $ is $\g$-equivariant. It follows that $\lambda^{(d)}(\pi_{\g})$ is $\g$-equivariant as well. On the other hand, it is easy to verify that the map~\eqref{mnatcalg} from $\lambda^{(d)}(\g \otimes \mathfrak L_{\g} )$ to $\lambda^{(d)}(\g) \otimes \mathfrak L_{\g}$ is also $\g$-equivariant. Since $\g$ acts trivially on $\lambda^{(d)}(\g)$, the desired proposition follows.
\end{proof}
One therefore obtains the maps
$$ \Tr_{\g}\,:\,\lambda^{(d)}(\mathfrak L) \rar  \lambda^{(d)}(\g) \otimes  \mathfrak L_{\g}^{\ad\,\g}  \,\text{.}$$

For $d=2$, the vector space $ \lambda^{(d)}(\g)$ is one-dimensional: indeed, there is a unique (up to a scalar factor) invariant bilinear form on $\g$ (the Cartan-Killing form). Hence, at the level of homology, the map
$\Tr_{\g}$ induces a {\it canonical} trace map
\begin{equation} \la{canlietrace} \Tr_{\g}\,:\,\HC_{\bullet}(\mathtt{Lie},\mathfrak{a}) \rar \H_{\bullet}(\mathfrak{a},\g)^{\ad\,\g}\,\text{.}\end{equation}
More generally, let $P \,\in \,\Sym^d(\g^{\ast})^{\ad\,\g}$.  Then,  one has the trace map
$$ \begin{diagram} \Tr_{\g}^P \,:\, \lambda^{(d)}(\mathfrak L) & \rTo^{\Tr_{\g}} &  \lambda^{(d)}(\g) \otimes  \mathfrak L_{\g}^{\ad\,\g}  & \rTo^{ P(\mbox{--}) \otimes \id } & \mathfrak L_{\g}^{\ad\,\g} \end{diagram}\,\text{.}$$
Recall that we have the Chevalley isomorphism
$$ \Sym(\g^{\ast})^{\ad\,\g} \,\cong\,\Sym(\h^{\ast})^W \,\cong\, k[\bar{P}_1, \ldots,\bar{P}_l] $$
where $\mathrm{deg}(\bar{P}_i)\,=\,d_i$ for $1 \leq i \leq l$ and the $d_i$ are the fundamental degrees of $\g$. Choosing $P_i \in \Sym(\g^{\ast})^{\ad\,\g}$ corresponding to $\bar{P}_i$ under Chevalley's isomorphism, we get a family of trace maps
$$ \Tr^{(d_i)}_{\g}\,:=\,\Tr_{\g}^{P_i}\,:\,\lambda^{(d_i)}(\mathfrak L) \rar \mathfrak L_{\g}^{\ad\,\g}\,,$$
which yields a homomorphism of commutative DG algebras
\begin{equation} \la{lietraces}
\bSym_k[\Tr_{\bullet}(\mathfrak L)]\,:\,\bSym_k[\oplus_{i=1}^l \lambda^{(d_i)}(\mathfrak L)] \rar \mathfrak L_{\g}^{\ad\,\g} \,\text{.} \end{equation}
 We refer to \eqref{lietraces} as the {\it Drinfeld trace map}. We shall sometimes abuse this terminology and use it for closely related maps as well. At the level of homology,~\eqref{lietraces} gives
\begin{equation} \la{lietracesh}
\bSym_k[\Tr_{\bullet}(\mathfrak{a})]\,:\,
\bSym_k[\oplus_{i=1}^l \HC_{\bullet}^{(d_i)}(\mathtt{Lie},\mathfrak{a})]
 \rar \H_{\bullet}(\mathfrak{a}, \mathfrak{g})^{\ad\,\g} \,\text{.} \end{equation}
In particular, if $\mathfrak L\,=\,\cb_{\mathtt{Comm}}(C)$, then by~\eqref{drhodge},
the above map becomes
\begin{equation} \la{hlietraces}
\bSym_k[\Tr_{\bullet}(\mathfrak{a})]\,:\,\bSym_k[\oplus_{i=1}^l \rHC_{\bullet}^{(d_i-1)}(C)[-1]] \rar \H_{\bullet}(\mathfrak{a}, \mathfrak{g})^{\ad\,\g}\,\text{.} \end{equation}

\noindent \textbf{Remark.} The Drinfeld trace map depends on the choice of the $P_i\,,\,\,1 \leq i \leq l$. This choice in turn depends precisely on the choice of an isomorphism
$$\Sym(\h^{\ast})^W \,\cong\, k[\bar{P}_1, \ldots,\bar{P}_l] \,\text{.}$$
\begin{example}
\la{1dieliedt}
Let $\mathfrak{a}\,:=\,k.x$ be a one-dimensional Lie algebra over $k$ with generator $x$ having weight $1$ and homological degree $0$. Note that $\mathfrak{a}$ is a free (and therefore, cofibrant) DG Lie algebra. Since $\mathfrak{a}$ is also abelian,
$$ \L \lambda^{(d)}(\mathfrak{a})\,\cong\, \Sym^d(\mathfrak{a})\,=\,k.x^d \,\text{.}$$
In this case, $\mathfrak{a}_{\g}^{\ad\,\g}\,=\,\Sym(\g^{\ast})^{\ad\,\g}$ and the map
$$ \lambda^{(d)}(\mathfrak{a}) \rar \mathfrak{a}_{\g}^{\ad\,\g} \otimes \lambda^{(d)}(\g) $$
becomes the map dual to the nondegenerate pairing
$$ \lambda^{(d)}(\g) \otimes \Sym^d(\g^{\ast})^{\ad\,\g} \rar k\,\text{.}$$
It follows that for a fixed choice of isomorphism
$$\Sym(\h^{\ast})^W \,\cong\, k[\bar{P}_1, \ldots,\bar{P}_l] \,,$$
the Drinfeld trace becomes the map
$$ k[\bar{P}_1, \ldots,\bar{P}_l] \rar \Sym(\g^{\ast})^{\ad\,\g}\,,\,\,\,\,\, \bar{P}_i \mapsto P_i $$
where the variable $\bar{P_i}$ is identified with $x^{d_i}$. Since $P_i$ corresponds to $\bar{P}_i$ under the Chevalley restriction isomorphism, the Drinfeld trace is indeed a generalization of the map inverse to the Chevalley restriction isomorphism. Combining this observation with Example~\ref{1dliehc}, we conclude that when $\mathfrak{a}$ is a one-dimensional Lie algebra, the derived Harish-Chandra homomorphism and the Drinfeld trace are mutually inverse (quasi-)isomorphisms.
\end{example}

\section{Derived commuting schemes}
\la{sec8}
In this section, we turn to our main example: the derived commuting
scheme associated to a finite-dimensional reductive Lie algebra $\g$. Our goal is
to state a general version of Conjecture 1 for $\g$, deduce the corresponding constant term
identity and present some evidence in favor of this conjecture.
The next section will explain the relation to the famous Macdonald
conjectures proved in~\cite{Ch} and~\cite{FGT}.

\subsection{Main conjecture}
Let $ \mathfrak{a} $ be an abelian Lie algebra of dimension $ r \ge 1 $.
In this case $ \, \mathcal{U}\mathfrak{a} = \Sym(\mathfrak{a})\,$ and
the graded coalgebra $C:=\bSym^c(\mathfrak{a}[1])$ is Koszul dual to $\mathfrak{a}$. Explicitly, $ C  =  k \oplus  \mathfrak{a} \oplus \wedge^2 \mathfrak{a}
\oplus \ldots \oplus \wedge^r \mathfrak{a} $,  where $\wedge^i \mathfrak{a} $ has homological degree $i$. For a reductive Lie algebra $ \mathfrak{g} $, the
derived Harish-Chandra homomorphism \eqref{derHCg} then becomes
$$
\Phi_{\g}\,:\,\DRep_{\g}(\mathfrak{a})^{G} \rar \bSym[\h^{\ast} \otimes (\mathfrak{a} \oplus \wedge^2  \mathfrak{a} \oplus \ldots \oplus \wedge^r  \mathfrak{a})]^W \,,
$$
where $\wedge^i  \mathfrak{a} $ has homological degree $i-1$.
If $\dim\,\mathfrak{a}\,=\,2$, $ \Phi_{\g} $ induces on the $0$-th homology the map
$$
k[\Rep_{\g}(\mathfrak{a})]^G \rar \Sym[\h^{\ast} \otimes  \mathfrak{a}]^W \,,
$$
which is known to be an isomorphism, at least when $\g$ is complex semisimple and $\Rep_{\g}(\mathfrak{a})/\!/G $ is reduced
(see, e.g., \cite{Hai}, Sect.~6.2).
It is therefore, reasonable to make the following conjecture extending
Conjecture~\ref{conj1}.

\begin{conjecture}
\la{conj3}
Let $ \dim(\mathfrak{a}) = 2$. Then, for any reductive Lie algebra $\g$ over $k$,
$$
\Phi_{\g}\,:\, \DRep_{\g}(\mathfrak{a})^G \rar \bSym(\h^{\ast} \oplus \h^{\ast} \oplus \h^{\ast}[1])^W$$ is a quasi-isomorphism (at least when the quotient commuting scheme
$\Rep_{\g}(\mathfrak{a})/\!/G $ is reduced).
\end{conjecture}

\begin{remark} The scheme $\Rep_{\g}(\mathfrak{a})$ in Conjecture~\ref{conj3} is precisely the classical commuting scheme of the Lie algebra $\g$. It is known that the underlying variety of $\Rep_{\g}(\mathfrak{a})$ is irreducible for any semisimple complex Lie algebra $\g$ (see~\cite{R}).
However, the question of whether $\Rep_{\g}(\mathfrak{a})$ (or even
$ \Rep_{\g}(\mathfrak{a})^G $) is a reduced scheme remains open in general.
\end{remark}

\vspace{1ex}

Conjecture~\ref{conj3} can be restated in elementary terms, without using the language of derived representation schemes. To this end,  define the DG algebra
$ \left(k[\g \times \g] \otimes \wedge \g^{\ast}, d \right)$, with $\g^{\ast}$ being
in homological degree $1$ and with differential $d\,:\,\g^{\ast}\rar k[\g \times \g]$ given by the formula
$$
d\varphi(\xi,\eta)\,:=\, \varphi([\xi,\,\eta])\ ,\quad
\forall\,\,\xi,\eta\,\in\,\g \, ,\ \varphi \in \g^*\ .
$$
The adjoint action of $G$ on $\g$ induces the diagonal $G$-action on $\left(k[\g \times \g] \otimes \wedge \g^{\ast}, d \right)$, and this last action commutes with differential.
We may therefore consider the invariant DG algebra
$\left(k[\g \times \g] \otimes \wedge \g^{\ast}, d \right)^G$. Since the functions on $\g$ restrict to $\h$, there is a natural DG algebra homomorphism
$$\Phi_{\g}\,:\, \left(k[\g \times \g] \otimes \wedge \g^{\ast}, d \right) \rar k[\h \times \h] \otimes \wedge \h^{\ast} \,,$$
where the right-hand side has zero differential. It is easy to see that the image of $\left(k[\g \times \g] \otimes \wedge \g^{\ast}, d \right)^G$ under $\Phi_{\g}$ lies in $\left(k[\h \times \h] \otimes \wedge \h^{\ast}\right)^W$, and we have the following proposition.
\bprop \la{conj5}
Conjecture~\ref{conj3} is equivalent to the following statement: the DG algebra map
\begin{equation} \la{hcwithoutdrep} \Phi_{\g}\,:\, \left(k[\g \times \g] \otimes \wedge \g^{\ast}, d \right)^G \rar \left(k[\h \times \h] \otimes \wedge \h^{\ast} \right)^W \end{equation}
is a quasi-isomorphism.
\eprop

\bproof
By Theorem~\ref{drepg}, there is an isomorphism
$$\DRep_{\g}(\mathfrak{a})^G\,\cong\,\CE^c(\g^{\ast}(C),\g^{\ast};k) $$
where $C\,=\,\bSym^c(\mathfrak{a}[1])$. Choose any basis $\{x,y\}$ of $\mathfrak{a}$ over $k$. Then, as graded algebras,
$$\CE^c(\g^{\ast}(C),\g^{\ast};k)\,\cong\,\bSym(\g^{\ast}.x \oplus \g^{\ast}.y \oplus \g^{\ast}.\theta) $$
where $\g^{\ast}.x\,:=\,\g^{\ast} \otimes x$, etc., and $x,y$ have homological degree $0$ and $\theta\,:=\,s^{-1}(sx \wedge sy)$ has homological degree $1$. Hence,
$$ \CE^c(\g^{\ast}(C),\g^{\ast};k) \,\cong\, k[\g \times \g] \otimes \wedge \g^{\ast} \,\text{.}$$
In particular, $\bSym(\g^{\ast}.x \oplus \g^{\ast}.y)$ can be identified with $k[\g \times \g]$. A direct computation using~\eqref{diff2} then shows that for any $ \varphi \,\in\,\g^{\ast}$, the differential of the generator $\frac{1}{2}\,\varphi.\theta$ in $\CE^c(\g^{\ast}(C),\g^{\ast};k)$ is equal to the function $ d\varphi \,\in\,k[\g \times \g]$ satisfying $d\varphi (\xi, \eta)\,=\,\varphi([\xi, \eta])$. This identifies $\CE^c(\g^{\ast}(C),\g^{\ast};k) $ with $\left(k[\g \times \g] \otimes \wedge \g^{\ast}, d \right)^G$. The identification of $\CE^c(\h^{\ast}(C),\h^{\ast};k)$ with $k[\h \times\h] \otimes \wedge \h^{\ast}$ is obvious. Since the derived Harish-Chandra homomorphism $\Phi_{\g}\,:\,\CE^c(\g^{\ast}(C),\g^{\ast};k) \rar \CE^c(\h^{\ast}(C),\h^{\ast};k)^W$ is indeed restriction of cocycles from $\g$ to $\h$, it coincides with the map $\Phi_{\g}\,:\, \left(k[\g \times \g] \otimes \wedge \g^{\ast}, d \right)^G \rar \left(k[\h \times \h] \otimes \wedge \h^{\ast} \right)^W$ defined above. This proves the desired proposition.
\eproof

\subsection{A constant term identity}
\la{constid}
We now assume that $ k = \c $ and compare the Euler characteristics of DG algebras in Conjecture~\ref{conj3}.
Write $ \mathfrak{a} = \c x \oplus \c y $ with $\,[x,y] = 0\,$, denote
$ B_{\g} := \DRep_{\g}(\mathfrak{a})$. Recall that as graded algebras,
\begin{equation} \la{bgasgralg}
B_{\g} \,\cong\, \bSym_{\c}(\g^*.x \oplus \g^*.y \oplus \g^*.\theta)\,\cong\,
\bSym_{\c}(\g.x \oplus \g.y \oplus \g.\theta)\ ,
\end{equation}
where $x,y$ have homological degree $0$ and $\theta$ has homological degree $1$. The DG algebra $B_{\g}$ is equipped with an additional $\Z^2$-weight grading, with the subspaces $\g.x$ and $\g.y$ having weights $(1,0)$ and $(0,1)$ and $\g.\theta$ having weight $(1,1)$.

Now, for $ h \in \h $, define $ \chi(B_{\g}, q,t,e^h)$ to be the character-valued Euler
characteristic
$$
\sum_{a,b \ge 0}\,\sum_{i \in \Z}\, (-1)^i\,\Tr(e^h\,|_{(B_{\g})_{i,a,b}})q^at^b\,,
$$
where $V_{a,b}$ denotes the component of a $\Z^2$-graded vector space $V$ of weight $(a,b)$.
Then
\blemma
\la{eulie1}
$$
\chi(B_{\g},q,t,e^h) \,=\, \frac{(1-qt)^{l}}{(1-q)^{l}(1-t)^{l}} \,\prod_{\alpha \in R} \frac{(1-qte^{\alpha(h)})}{(1-qe^{\alpha(h)})(1-te^{\alpha(h)})}\ ,
$$
where $\, l = \dim_{\c} \h \,$ is the rank of $\g $ and $R$ is the
root system associated to $\g$.
\elemma
\begin{proof}
Since $\, \g = \h \oplus (\oplus_{\alpha \in R} \,\g_{\alpha})\,$, it follows from \eqref{bgasgralg} that
$$
B_{\g} \,\cong\, B_{\h} \otimes (\otimes_{\alpha \,\in\,R} B_{\g_{\alpha}}) \ ,
$$
where $B_{\h}\,:=\, \bSym(\h.x \oplus \h.y \oplus \h.\theta) $ and $
B_{\g_{\alpha}}\,:=\,\bSym(\g_{\alpha} .x \oplus \g_{\alpha}.y \oplus \g_{\alpha}.\theta) $.  Since $e^h $ acts as the identity on $\h$ and by multiplication by $e^{\alpha(h)}$ on the root space $ \g_{\alpha} $, we have
\begin{equation} \chi(B_{\h}, q,t,e^h)\,=\,  \frac{(1-qt)^{l}}{(1-q)^{l}(1-t)^{l}} \,,\,\,\quad\,\, \chi(B_{\g_{\alpha}}, q,t,e^h)\,=\, \frac{(1-qte^{\alpha(h)})}{(1-qe^{\alpha(h)})(1-te^{\alpha(h)})}\,\text{.} \end{equation}
The desired lemma now follows from the multiplicativity of the Euler characteristic.
\end{proof}
Let $ Q = Q(R) $ be the root lattice of $\g$, and let $ \Z[Q] $ denote the group ring of $Q$.
For each $ \alpha \in R \subset Q(R)\,$, write $ e^{\alpha} \in \Z[Q] $ for the
corresponding element in $\Z[Q]$, and denote by $ {\rm CT}: \Z[Q] \to \Z $  the
map assigning to a polynomial in $ \Z[Q] $ its constant term which does not involve
any $ e^{\alpha} $. This constant term map naturally extends to the ring $ \Z[Q][[q,t]] $ of
formal power series over $ \Z[Q] \,$: specifically, we define ${\rm CT}\,:\, \Z[Q][[q,t]] \rar \Z[[q,t]]$ by
\begin{equation}
\la{cterm}
{\rm CT}\,\sum_{a,b \ge 0} P_{a,b}(e^\alpha) \,q^a t^b\, :=\,
\sum_{a,b \ge 0} {\rm CT}\,[P_{a,b}(e^\alpha)]\,q^a t^b\ .
\end{equation}
As corollary of Lemma~\ref{eulie1}, we now get
\bcor \la{corlie1} The weighted Euler characteristic of $B_{\g}^G$ is given by
$$
\chi(B_{\g}^G,q,t) \,=\, \frac{1}{|W|}\ \frac{(1-qt)^{l}}{(1-q)^{l}(1-t)^{l}}\ \mathrm{CT}\, \left\{ \prod_{\alpha \in R}  \frac{(1-qte^{\alpha})(1-e^{\alpha})}{(1-qe^{\alpha})(1-te^{\alpha})}
\right\}\ .
$$
\ecor
\begin{proof}
It suffices to verify the above formula for the compact real form of the group $ G $, which
we also denote $ G $.  We have
$$
\chi(B_{\g}^G,q,t) \,=\, \int_{G} \chi(B_{\g},q,t, \,\mathrm{Ad}\,g)\, d g\ ,
$$
where $ d g $ is the Haar measure on $ G $ normalized by $\, \int_{G} d g = 1 \,$.
Now, let $ T $ be the maximal torus of $ G $ corresponding to $ R $. For $ \alpha \in R \,$,
regard $ e^{\alpha} $ as a character of $ T $. Then, by Weyl's integration formula,
$$
\int_{G} \chi(B_{\g},q,t, \,\mathrm{Ad}\,g)\, d g \ = \ \frac{1}{|W|}\,\int_{T} \chi(B_{\g},q,t,  \,{\rm Ad}\,\tau)\,
\prod_{\alpha \in R}(1 - e^{\alpha}(\tau)) \,d \tau
$$
where $ d \tau $ is the normalized Haar measure on $ T $. The result follows immediately
from Lemma~\ref{eulie1} and  the fact that $\,\int_{T} e^{\alpha}(\tau)\,d \tau = 0 \,$ for any nonzero
root $ \alpha $.
\end{proof}

Next, we look at the right-hand side of Conjecture~\ref{conj3}.
We write
$$
A_{\g} := \bSym_\c(\h^{\ast} \oplus \h^{\ast} \oplus \h^{\ast}[1]) \,\cong\,
\bSym_{\c}(\h.x \oplus \h.y \oplus \h.\theta) \ ,
$$
where $ x,y $ are of homological degree
$0$ and $\theta$ of homological degee $1$. Again, the algebra $A_{\g}$ has an additional $\Z^2$-grading,  with $\h.x$ and $\h.y$ having weights $(1,0)$ and $(0,1)$ and $\h.\theta$ having
weight $(1,1)$.
For an element $ w \in W$, let $\, \{\lambda_1,\ldots, \lambda_l\} \,$ be the eigenvalues $ w $
under the natural action of $W$ on $ \h $. Then
$$
\chi(A_{\g},q,t,w)\,=\, \prod_{i=1}^l \frac{(1-qt\lambda_i)}{(1-q\lambda_i)(1-t\lambda_i)}\,=
\,\frac{\det(1 - qt w)}{\det(1-q w)\,\det(1-t w)}\,,
$$
where `$\,\det \,$'  is taken on $ \End\,\h\,$. By the classical Molien formula, we get
$$
\chi(A_{\g}^W,q,t)\,=\,\frac{1}{|W|} \sum_{w \in W}\,\frac{\det(1 - qt w)}{\det(1-q w)\,\det(1-t w)}\ .
$$
Conjecture~\ref{conj3} therefore implies the following constant term identity
generalizing \eqref{conj2}:
\begin{conjecture}
\la{conj4}
The following identity holds:
$$
\frac{(1-qt)^{l}}{(1-q)^{l}(1-t)^{l}}\ \mathrm{CT}\, \left\{ \prod_{\alpha \in R}  \frac{(1-qte^{\alpha})(1-e^{\alpha})}{(1-qe^{\alpha})(1-te^{\alpha})}
\right\}\, = \,
\sum_{w \in W}\,\frac{\det(1 - qt\,w)}{\det(1-q\,w)\,\det(1-t\,w)}\ .
$$
\end{conjecture}

\vspace{1ex}

Note that Conjecture~\ref{conj4} can be equivalently rewritten as
\begin{equation}
\la{eqkt}
\int_{G}\, \frac{\det(1\,-\,qt\,\mathrm{Ad}\,g)}
{\det(1\,-\,q\,\mathrm{Ad}\,g)\,\det(1\,-\,t\,\mathrm{Ad}\,g)}\,
d g\ = \ \frac{1}{|W|}\, \sum_{w \in W}\,\frac{\det(1 - qt w)}{\det(1-q w)\,\det(1-t w)}\ .
\end{equation}
For $ t = 0 $, this simplifies to the well-known identity
\begin{equation}\la{eqk0}
\int_{G}\, \frac{d g}{\det(1\,-\,q\,\mathrm{Ad}\,g)} \ = \
\frac{1}{|W|}\, \sum_{w \in W}\,\frac{1}{\det(1 - q w)}\ = \ \prod_{i=1}^l \frac{1}{1-q^{d_i}}\ ,
\end{equation}
which is obtained by equating the Poincar\'{e} series of both sides of the Chevalley isomorphism
$$
\Sym(\g^{\ast})^G \cong \Sym(\h^{\ast})^W .
$$
In \eqref{eqk0}, the numbers $ d_i $ are the fundamental degrees of $ W $, i.e. the degrees of algebraically
independent elements generating $ \Sym(\h^{\ast})^W $. If we expand both sides
of \eqref{eqkt} as power series in $ t $ and compare
the corresponding Taylor coefficients under $ t^k $, we get a sequence of identities
extending \eqref{eqk0} for $ k \ge 1 $. Thus, Conjecture~\ref{conj4} answers a question of I.~Macdonald posed in his original
paper \cite{M} (see {\it op. cit.}, Remark~2, p.~997).

We now provide some evidence for Conjecture~\ref{conj3} and Conjecture~\ref{conj4}.
We recall that we have already proved Conjecture~\ref{conj3} for $ \mathfrak{gl}_2 $
(Theorem~\ref{conj1n2}) and Conjecture~\ref{conj4} for $ \gl_n $ for an arbitrary $ n $
(Theorem~\ref{Tconj2}).

\subsection{Lower order terms}
First, we show that the Taylor expansion in $q,t$ of both sides of Conjecture~\ref{conj4} agree up to degree $2$. In fact, in their normalized form
\eqref{eqkt} the quadratic terms are independent of the root system.
\bprop
\la{2terms}
The first terms of both sides of \eqref{eqkt} (viewed as power series in $q,t$) are equal to
\[
1+q^2+qt+t^2 +\,\ldots\ ,
\]
where the dots stand for terms of degree at least 3 in $q,t$.
\eprop
Proposition~\ref{2terms} follows from Lemma~\ref{l2terms} below that evaluates both sides of~\eqref{eqkt} as integrals of the same ratio of
determinants over a compact Lie group for a real irreducible
nontrivial representation  (we formally think of $W$ as a compact Lie group of dimension $0$) .
\begin{lemma} \la{l2terms}
Let $\rho\colon G\to \mathrm{GL}(V)$ be a nontrivial irreducible representation
of a compact Lie group $G$ over $\mathbb R$ with invariant measure $dg$ and volume $|G|$.
Then
\[
\frac1{|G|}\int_G\frac{\det(1-qt \rho(g))}{\det(1-q\rho(g))\det(1-t\rho(g))}dg=1+q^2+qt+t^2+\ldots
\]
\end{lemma}
\begin{proof}
For any endomorphism $a$ of a finite-dimensional vector space,
$$
\frac{\det(1-qt a)}{\det(1-qa)\det(1-ta)}
=1+(q+t-qt)\mathrm{tr}\,a+\frac12(q+t)^2(\mathrm{tr}\,a)^2+ \frac12(q^2+t^2)\mathrm{tr}(a^2)+\ldots
$$
If $\rho$ is irreducible and nontrivial then
its character has norm 1 and is orthogonal
to the trivial character. Thus
\[
\frac1{|G|}\int_G\mathrm{tr}\,\rho(g)\,dg=0,\quad \frac1{|G|}\int_{G}|\mathrm{tr}\,\rho(g)|^2dg=1.
\]
Now real representations have real-valued characters and
the Frobenius-Schur theorem holds:
\[
\frac1{|G|}\int_G\mathrm{tr}\,\rho(g^2)\,dg=1.
\]
The average over the group is then
\[
1+\frac12(q+t)^2
+
\frac12(q^2+t^2)+\ldots=1+q^2+qt+t^2+\ldots\ .
\]
This finishes the proof of the lemma and Proposition~\ref{2terms}.
\end{proof}

\subsection{The case of $\mathfrak{sl}_n$} Our next result is
\begin{theorem}
\la{glnvssln}
For any $ n \ge 2 $, Conjecture~\ref{conj3}\, for $\sll_n$ is equivalent to
Conjecture~\ref{conj3}\, for $\mathfrak{gl}_n$.
\end{theorem}
\begin{proof}
Let $\mathfrak{t}_n$ denote the Cartan subalgebra of $\mathfrak{sl}_n$ comprising diagonal matrices with trace $0$. The short exact sequence of Lie algebras
$$
0 \to \mathfrak{sl}_n \to  \gl_n \xrightarrow{\mathrm{Tr}}  k \to 0
$$
has a canonical splitting, giving an isomorphism of Lie algebras:
\begin{equation}
\la{gltosl}
\gl_n \stackrel{\sim}{\rar} \mathfrak{sl}_n \oplus k\,,\,\,\,\, M \mapsto (M -\frac{1}{n}\,\mathrm{Tr}(M)\, \mathrm{Id}_n, \mathrm{Tr}(M)) \,\text{.}
\end{equation}
This restricts to an isomorphism of diagonal Cartan subalgebras $\h_n \,\cong\, \mathfrak{t}_n
\oplus k $, which is equivariant under the action of the Weyl group $S_n$ (with $S_n$ acting trivially on $k$).

Let $C:=\bSym^c(\mathfrak{a}[1])$ where $\mathfrak{a}:=k.x \oplus k.y$ is the two dimensional abelian Lie algebra. The isomorphism~\eqref{gltosl}  induces an isomorphism $\gl_n^*(C) \,\cong\, \sll_n^*(C) \oplus C\,$. Hence,
\begin{equation} \la{cecomp} \C^c(\gl_n^*(C), \gl_n^*(k);k) \,\cong\, \C^c(\mathfrak{sl}_n^*(C), \mathfrak{sl}_n^*(k);k) \otimes \C^c(\bar{C};k) \,\text{.} \end{equation}
Similarly, the $S_n$ equivariant isomorphism $\h_n \cong \mathfrak{t}_n \oplus k$ yields
\begin{equation} \la{cecart}
\C^c(\h_n^*(C), \h_n^*(k);k)^{S_n} \cong \C^c(\mathfrak{t}_n^*(C), \mathfrak{t}_n^*(k);k)^{S_n} \otimes \C^c(\bar{C};k)\,\text{.}
\end{equation}
Let $\Phi_{\gl_n}$ (resp., $\Phi_{\mathfrak{sl}_n}$) denote the derived Harish-Chandra homomorphisms for $\gl_n$ (resp., $\mathfrak{sl_n}$). Since the isomorphism $\h_n \,\cong\, \mathfrak{t}_n \oplus k $ is the restriction of the isomorphism  $\gl_n \cong \mathfrak{sl}_n \oplus k$ to $\h_n$, one has the following commutative diagram:
$$
\begin{diagram}
  \C^c(\gl_n^*(C), \gl_n^*(k);k)  & \rTo^{\cong} &  \C^c(\mathfrak{sl}_n^*(C),\mathfrak{sl}_n^*(k);k) \otimes \C^c(\bar{C};k)\\
 \dTo^{\Phi_{\gl_n}} & & \dTo^{\Phi_{\mathfrak{sl_n}} \otimes \mathrm{Id}} \\
 \C^c(\h_n^*(C), \h_n^*(k);k)^{S_n} & \rTo^{\cong} & \C^c(\mathfrak{t}_n^*(C), \mathfrak{t}_n^*(k);k)^{S_n} \otimes \C^c(\bar{C};k)
\end{diagram}$$
Hence, $\Phi_{\gl_n}$ is a quasi isomorphism if $\Phi_{\mathfrak{sl_n}}$ is. Conversely, if $\Phi_{\mathfrak{sl_n}}$ is not a quasi isomorphism, choose $k$ to be the minimum homological degree such that $\H_k(\Phi_{\mathfrak{sl_n}})$ is not an isomorphism. Since $\C^c(\bar{C};k) \,\cong\,k[x,y,\theta]$,
$$\H_k(\Phi_{\gl_n}) \,\cong\, \H_k(\Phi_{\mathfrak{sl_n}}) \otimes \mathrm{Id}_{k[x,y]} \oplus \H_{k-1}(\Phi_{\mathfrak{sl_n}}) \otimes \mathrm{Id}_{k[x,y].\theta} \,,$$
and we see that $\H_k(\Phi_{\gl_n})$ is not an isomorphism either. This proves the desired theorem.
\end{proof}

As a consequence of (the proof of) Theorem~\ref{glnvssln}, we have

\bcor
Conjecture~\ref{conj4} holds for $\, \mathfrak{sl}_n \,$ for all $ n \ge 2 $.
\ecor

\begin{proof}
We need to show the equality of Euler characteristics:
$$
\chi(\DRep_{\mathfrak{sl}_n}(\mathfrak{a})^{\SL_n}, q,t) = \chi(\C^c(\mathfrak{t}_n^*(C), \mathfrak{t}_n^*(k);k)^{S_n}, q,t)\,\text{.}
$$
In view of isomorphisms~\eqref{cecomp} and ~\eqref{cecart}, we have
$$\chi(\DRep_{\mathfrak{sl}_n}(\mathfrak{a})^{\SL_n}, q,t).\chi(k[x,y,\theta], q,t) =
\chi(\C^c(\gl_n^*(C), \gl_n^*(k);k) , q, t) \,\text{.}$$
and
$$ \chi(\C^c(\mathfrak{t}_n^*(C), \mathfrak{t}_n^*(k);k)^{S_n}, q,t).\chi(k[x,y,\theta],q, t) =
\chi(\C^c(\h_n^*(C), \h_n^*(k);k)^{S_n}, q,t) \,\text{.}$$
Note that $\chi(k[x,y,\theta],q,t) = \frac{(1-qt)}{(1-q)(1-t)}$. Hence, $\chi(k[x,y,\theta],q,t)$ is an invertible power series in $q,t$. The result follows now from
Theorem~\ref{Tconj2}, which says that the  Euler characteristics of
$ \C^c(\gl_n^*(C), \gl_n^*(k);k) $ and $ \C^c(\h_n^*(C), \h_n^*(k);k)^{S_n} $ are equal.
\end{proof}

\subsection{Orthogonal and symplectic Lie algebras}
\la{osl}
We now verify that Conjecture~\ref{conj3} holds in the limit $ n \to \infty $
for orthogonal and symplectic Lie algebras. We begin with
definitions of these classical Lie algebras.

\subsubsection{The orthogonal Lie algebras $\mfo_{2n+1}$} \la{orthliealg}
Fix a basis $\{e_1,\ldots,e_{2n+1}\}$ of $k^{2n+1}$.
Let $\M_{2n+1}$ denote the matrix of the nondegenerate symmetric bilinear form $Q$ satisfying
\begin{eqnarray*}
 Q(e_{2i-1}, e_{2i}) \,=\, Q(e_{2i}, e_{2i-1}) &=& 1\,,\,\,\,\, 1 \leq i \leq n\\
                Q(e_{2n+1}, e_{2n+1}) &=& 1\\
  Q(e_i, e_j) &=& 0 \,\,\,\, \text{otherwise.}
\end{eqnarray*}
Define $\mathrm{SO}(2n+1)$ to be the group of invertible matrices of determinant $1$ preserving the bilinear form $Q$. In other words, $\mathrm{SO}(2n+1)$ is the group of invertible matrices $P$ of size $2n+1$ and determinant $1$ satisfying
$$ \M_{2n+1} = P^t\M_{2n+1}P \,\text{.}$$
The Lie algebra $\mfo_{2n+1}$ of $\mathrm{SO}(2n+1)$ is therefore the Lie subalgebra of $\gl_{2n+1}$ comprising those matrices $X$ satisfying
$$ X^t\M_{2n+1} +\M_{2n+1}X = 0 \,\text{.}$$
Further, let $\M_{2n}$ denote the matrix of the nondegenerate symmetric bilinear form $Q$ satisfying
\begin{eqnarray*}
 Q(e_{2i-1}, e_{2i}) \,=\, Q(e_{2i}, e_{2i-1}) &=& 1\,,\,\,\,\, 1 \leq i \leq n\\
  Q(e_i, e_j) &=& 0 \,\,\,\, \text{otherwise.}
\end{eqnarray*}
The Lie algebra $\mfo_{2n}$ can similarly be defined as the Lie subalgebra of $\gl_{2n}$ comprising those matrices $X$ that satisfy
$$ X^t\M_{2n} +\M_{2n}X = 0 \,\text{.}$$
Padding on the right and bottom with $0$'s gives an embedding $\mfo_n \hookrightarrow \mfo_{n+1} $ for all $n$. We may therefore, form the direct limits
$$\mfo_{2\infty+1}\,:=\,\varinjlim_n \mfo_{2n+1}\,,\,\,\,\,\mfo_{2\infty}\,:=\,\varinjlim_n \mfo_{2n} \,,\,\,\,\,\mfo_{\infty}\,:=\,\varinjlim_n \mfo_n\,\text{.}$$
It is easy to see that one has isomorphisms of Lie algebras: $\mfo_{2\infty+1}\,\cong\,\mfo_{\infty}\,\cong\,\mfo_{2\infty}\,\text{.}$

\subsubsection{Lie algebras with (anti)involution}

An { \it (anti) involution} on a DG Lie algebra $\g$ is a map of complexes $\sigma\,:\,\g \rar \g$  satisfying $\sigma^2=\id$ and
$$ \sigma([X, Y]) \,=\, (-1)^{|X||Y|} [\sigma(Y), \sigma(X)] \,\,\quad\,\, \forall\,\, X, Y\,\in\,\g \text{.}$$
Note that multiplication by $-1$ in an anti involution on every Lie algebra $\g$. For terminological brevity, we refer to an anti involution on a Lie algebra as an involution.  Further, for any involution $\sigma$ on a DG Lie algebra $\g$, the $-1$ eigenspace of $\sigma$ is a DG Lie subalgebra of $\g$.  For example, if $\M_{2n+1}$ is as in Section~\ref{orthliealg}, then
$$ X \mapsto \M_{2n+1}^tX^t \M_{2n+1} \,=\, \M_{2n+1}X^t\M_{2n+1} $$
is an involution on $\gl_{2n+1}$. The $(-1)$-eigensspace of this involution is precisely $\mfo_{2n+1}$.

Note that any involution on a DG Lie algebra $\mathfrak{L}$ extends to an involution on its universal enveloping algebra $\mathcal U \mathfrak{L}$. Equip $\M_n(k)$ with an involution. Recall from~\cite{BR1} that for an involutive DG algebra $A$, one has the commutative DG algebra $\Rep_n^*(A)$ of functions on the DG scheme parametrizing involution preserving representations of $A$ to $\M_n(k)$. Let $\gl_n^{-}$ denote the $(-1)$-eigenspace of this involution. The following proposition is proven in~\cite{BR1}.

\bprop \la{invreplie}
For any DG Lie algebra $\mathfrak{L}$, we have an isomorphism of commutative DG algebras
$$ k[\Rep_{\gl_n^{-}}(\mathfrak{L})]\,\cong\, k[\Rep_n^*(\mathcal U \mathfrak{L})]\,$$
where the involution on $\mathcal U \mathfrak{L}$ is the extension of the involution on $\mathfrak{L}$ given by multiplication by $-1$.
\eprop

Suppose that $\mathfrak{L} \stackrel{\sim}{\rar} \g$ is a cofibrant resolution in $\DGL_{k}$. Then, $\mathcal U \mathfrak{L} \stackrel{\sim}{\rar} \mathcal U\g $ is a cofibrant resolution in $\DGA_{k/k}$. Let $\sigma$ denote the involution on $\mathcal U \mathfrak{L}$ (resp., $\mathcal U\g$) extending multiplication by $-1$ on $\mathfrak{L}$ (resp., $\g$). Suppose that the involution chosen on $\M_n(k)$ is trace preserving. Let $G$ be a Lie subgroup of $\mathrm{GL}_n(k)$ whose Lie algebra is $\gl_n^{-}$. For $A$ an involutive DG algebra, denote
$$
A_{\natural, \sigma}\,:=\, A/(k+[A,A]+\mathrm{Im}(1-\sigma))\,\text{.}
$$
Then there is a natural morphism of complexes
\begin{equation} \la{involutivetrace}
\mathrm{Tr}_n:\,
{\mathcal U}(\mathfrak{L})_{\natural, \sigma} \rar k[\Rep_n^*(\mathcal U \mathfrak{L})]^G \ ,
\end{equation}
induced by the trace map $\,\Tr_n[{\mathcal U}(\mathfrak{L})]:\, U(\mathfrak{L})_{\natural } \to
(\mathcal U \mathfrak{L})_n^{\GL}\,$  (see~\cite{BR1}). This gives the commutative diagram at the level of homology groups:
\begin{equation} \la{cdinv2}
  \begin{diagram}[small, tight]
     \rHC_{\bullet}(\mathcal U \g) & \rTo^{\mathrm{Tr}_n} & \H_{\bullet}(\mathcal U\g, n)^{\GL}\\
       \dOnto & & \dTo\\
      \rHD_{\bullet}(\mathcal U\g) & \rTo^{\mathrm{Tr}_n} & \H_{\bullet}^*(\mathcal U\g, n)^G
      \end{diagram} \end{equation}
Here, $\rHD_{\bullet}(\mathcal U\g)$ denotes the reduced dihedral homology of $\mathcal U\g$ with respect to the involution on $\mathcal U\g$ that extends the involution on $\g$ given by multiplication with $-1$.

Let $\mathfrak{a} = k.x \oplus k.y$. Let $A\,:=\, \mathcal U \mathfrak{a} \,\cong\,\, k[x,y]$ denote the universal enveloping algebra of $\mathfrak{a}$ equipped with the involution that takes $x$ (resp., $y$) to $-x$ (resp., $-y$). Let $C\,:= \bSym^c(\mathfrak{a}[1])$. Then, $\mathfrak{L}\,:=\,\cb_{\mathtt{Comm}}(C)$ is a cofibrant resolution of the abelian Lie algebra $\mathfrak{a}$. Explicitly, $\mathfrak{L}$ is the free Lie algebra generated by $x,y, \theta$ with $x,y$ in degree $0$, $\theta$ in degree $1$ satisfying $d\theta=[x,y]$. Its universal enveloping algebra is the free algebra $R\,:=\,k\langle x, y, \theta \,| \,d\theta=[x,y]\,\rangle$ resolving $k[x,y]$. By Proposition~\ref{invreplie}, we have
$$\DRep_{\mfo_{2n+1}}(\mathfrak{a})\,\cong\, k[\Rep_{\mfo_{2n+1}}(\mathfrak{L})] \,\cong \, k[\Rep^*_{2n+1}(R)] \,\text{.}$$
Hence,  the trace maps~\eqref{involutivetrace} become
$$\mathrm{Tr}_{2n+1}\,:\, R_{\natural, \sigma} \rar k[\Rep_{\mfo_{2n+1}}(\mathfrak{L})]^{\mathrm{SO}_{2n+1}} \,$$
which give the following maps at the level of homology groups.
$$\mathrm{Tr}_{2n+1}\,:\, \rHD_{\bullet}(A) \rar \H_{\bullet}(\mathfrak{a}, \mfo_{2n+1})^{\mathrm{SO}_{2n+1}}\,\text{.}$$
The inclusion of Lie algebras $\mfo_{2n+1} \hookrightarrow \mfo_{2n+3}$ induces a (degreewise surjective) homomorphism of commutative DG algebras
$$ \mu_{2n+1}\,:\,k[\Rep_{\mfo_{2n+3}}(\mathfrak{L})] \twoheadrightarrow  k[\Rep_{\mfo_{2n+1}}(\mathfrak{L})] \,\text{.}$$
It is easy to verify that $\mu_{2n+1}$ maps $k[\Rep_{\mfo_{2n+3}}(\mathfrak{L})]^{\mathrm{SO}_{2n+3}}$ to $k[\Rep_{\mfo_{2n+1}}(\mathfrak{L})]^{\mathrm{SO}_{2n+1}}$ and that
$$\mu_{2n+1} \circ \mathrm{Tr}_{2n+3} \,=\,\mathrm{Tr}_{2n+1} \,\text{.}$$
Hence, one obtains a map of complexes
$$\mathrm{Tr}_{\infty}\,:\,R_{\natural, \sigma}\rar k[\Rep_{\mfo_{2\infty+1}}(\mathfrak{L})]^{\mathrm{SO}_{2\infty+1}} \,:=\, \varprojlim_n k[\Rep_{\mfo_{2n+1}}(\mathfrak{L})]^{\mathrm{SO}_{2n+1}} \,\text{.}$$
Here, the projective limit is taken in the category of {\it bigraded commutative} DG algebras. By multiplicativity, one obtains a morphism of commutative DG algebras
$$\bSym(\mathrm{Tr}_{\infty})\,:\, \bSym(R_{\natural, \sigma}) \rar k[\Rep_{\mfo_{2\infty+1}}(\mathfrak{L})]^{\mathrm{SO}_{2\infty+1}} \,\text{.}$$
The following result is a consequence of the stabilization theorem~\cite[Theorem 4.4]{BR}, for involutive DG algebras. We however, sketch a different proof of this statement.
\bprop \la{stabinvrep}
$\bSym(\mathrm{Tr}_{\infty})\,:\, \bSym(R_{\natural, \sigma}) \rar k[\Rep_{\mfo_{2\infty+1}}(\mathfrak{L})]^{\mathrm{SO}_{2\infty+1}} $ is an isomorphism of bigraded commutative DG algebras.
\eprop
\begin{proof}
Let $E:=C^{\ast}$ denote the bigraded dual of $C$. Then, the complex $R_{\natural, \sigma}$, whose homology is $\rHD_{\bullet}(k[x,y])$, coincides {\it on the nose} with the complex $
D(E)^{\ast}[-1]$, where $\ast$ denotes bigraded dual and $ D(E)$ is the complex computing the relative {\it skew} dihedral homology of $E$ with respect to the trivial involution. On the other hand, by Theorem~\ref{drepg} ,
$$k[\Rep_{\mfo_{2n+1}}(\mathfrak{L})]^{\mathrm{SO}_{2n+1}} \,\cong\, \C(\mfo_{2n+1}(E),\mfo_{2n+1}(k);k)^{\ast}\,\text{.}$$
Hence,
$$ k[\Rep_{\mfo_{2\infty+1}}(\mathfrak{L})]^{\mathrm{SO}_{2\infty+1}} \,\cong\, \C(\mfo_{2\infty+1}(E), \mfo_{2\infty+1}(k);k)^{\ast} \,\text{.}$$
Note that $\mfo_{2\infty+1} \,\cong\,\mfo_{\infty}$. The map $\bSym(\mathrm{Tr}_{\infty})$ can then be verified to be the bigraded dual of a the map
$$  \C(\mfo_{\infty}(E), \mfo_{\infty}(k);k) \rar \bSym(D(E)[1]) $$
induced by the map of complexes denoted in~\cite[Section 10.2.3]{L} by $\mathrm{tr}_* \circ \theta_*$. By the relative version of~\cite[Theorem 5.5]{LP} (see also~\cite[Section 10.5.7]{L}), this last map is an isomorphism of DG coalgebras. This proves the desired proposition.
\end{proof}

 On homologies, $\mathrm{Tr}_{\infty}$ gives a map
 $$ \mathrm{Tr}_{\infty}\,:\,\rHD_{\bullet}(A) \rar \H_{\bullet}(\mathfrak{a}, \mfo_{2\infty+1})^{\mathrm{SO}_{2\infty+1}}\,:=\, \H_{\bullet}(k[\Rep_{\mfo_{2\infty+1}}(\mathfrak{L})]^{\mathrm{SO}_{2\infty+1}})\,\text{.}$$ By Proposition~\ref{stabinvrep}, $\bSym(\mathrm{Tr}_{\infty})$ gives an isomorphism on homologies
$$ \bSym(\mathrm{Tr}_{\infty})\,:\,\bSym(\rHD_{\bullet}(A)) \stackrel{\sim}{\rar} \H_{\bullet}(\mathfrak{a}, \mfo_{2\infty+1})^{\mathrm{SO}_{2\infty+1}} \,\text{.}$$
We therefore need to compute the (relative) dihedral homology $\rHD_{\bullet}(A)$. Note that it is the space of covariants (and hence, invariants) of the $\Z_2$-action on $\rHC_{\bullet}(A)$ induced by the chosen involution on $A$. On the other hand, $\rHC_0(A)\,\cong\, A/k$ and $\rHC_1(A) \cong \Omega^1(A)/dA \,\cong\, A.ydx$ and other reduced cyclic homologies of $A$ vanish.  Let $A^{\mathrm{odd}}$ and $A^{\mathrm{even}}$ denote the subspaces of $A=k[x,y]$ spanned by monomials of odd and even weight respectively (for this definition, $x$ and $y$ are both taken to have weight $1$).

\blemma  \la{dhompoly}
There are isomorphisms of vector spaces
\begin{eqnarray*}
\rHD_0(A)\,\cong\, A^{\mathrm{ev}}/k\,,\,\,\,\,
 \rHD_1(A)\,\cong\, A^{\mathrm{odd}}.ydx
\end{eqnarray*}
\elemma

\begin{proof}
Recall from~\cite[Example 4.1]{BKR} that the monomial $x^ky^l$ in $A/k$ is identified with the cyclic chain $x^ky^l$ in $R_{\natural}$. The involution on $R$ maps $x^ky^l$ to $(-1)^{k+l}y^lx^k$ which coincides with $(-1)^{k+l}x^ky^l$ in $R_{\natural}$. Hence, $x^ky^l$ is a nontrivial basis element of $\frac{R}{k+[R,R]+\mathrm{Im}(1-\sigma)}$ iff $k+l$ is even.

Similarly,~\cite[Example 4.1]{BKR} tells us that the basis element $x^ky^ldx$ ($l \geq 1$) of $\rHC_1(A)$ is identified with the cycle
$\sum_{i=0}^{l-1} y^{l-1-i}x^ky^i\theta$ in $R_{\natural}$. This cycle is mapped by the involution on $R$ to
$ (-1)^{k+l} \sum_{i=0}^{l-1} \theta y^ix^ky^{l-1-i} $, which coincides with $(-1)^{k+l} \sum_{i=0}^{l-1} y^{l-1-i}x^ky^i\theta$ in $R_{\natural}$. It follows that $x^ky^ldx$ is a nontrivial basis element in $\rHD_1(A)$ iff $k+l$ is even and $l \geq 1$. This proves the desired lemma.
\end{proof}

\subsubsection{The derived Harish-Chandra homomorphism for the odd orthogonal case}

Note that for $\mfo_{2n+1}$, one can choose the Cartan subalgebra $\h_{2n+1}$ to be the abelian Lie subalgebra of $\mfo_{2n+1}$ spanned by the basis elements $h_i\,:= \,e_{2i-1,2i-1}-e_{2i,2i}\,,\,\,1 \leq i \leq n$. The Weyl group of $\mfo_{2n+1}$ is $(\Z_2)^n \rtimes S_n$, where elements of $S_n$ permute the $h_i$'s and the element $\gamma_i\,:=\,(1,\ldots, \gamma, \ldots, 1)$ of $(\Z_2)^n$ with the generator $\gamma$ of $\Z_2$ in the $i$-th coordinate transforms $h_i$ to $-h_i$ and leaves other $h_j$'s unchanged.

The derived Harish-Chandra homomorphism therefore becomes a morphism of commutative DG algebras
$$\Phi\,:\, k[\Rep_{\mfo_{2n+1}}(\mathfrak{L})]^{\mathrm{SO}_{2n+1}} \rar k[x_1,\ldots,x_n,y_1,\ldots,y_n,\theta_1,\ldots,\theta_n]^{(\Z_2)^n \rtimes S_n} $$
where elements of $S_n$ simultaneously permute the $x_i$'s, $y_i$'s and $\theta_i$'s and $\gamma_i$ multiplies the generators $x_i, y_i$ and $\theta_i$ by $-1$ and leaves other generators unchanged. In the limit, the derived Harish-Chandra homomorphism becomes the map
$$ \Phi\,:\,  k[\Rep_{\mfo_{2\infty+1}}(\mathfrak{L})]^{\mathrm{SO}_{2\infty+1}}  \rar k[x_1,\ldots; y_1,\ldots; \theta_1,\ldots ]^{(\Z_2)^{\infty} \rtimes S_{\infty}}$$
where $$k[x_1,\ldots; y_1,\ldots;\theta_1,\ldots]^{(\Z_2)^{\infty} \rtimes S_{\infty}}\,:=\,\varprojlim_n k[x_1,\ldots,x_n,y_1,\ldots,y_n,\theta_1,\ldots,\theta_n]^{(\Z_2)^n \rtimes S_n} \,\text{.}$$ Here, the limit is taken in the category of {\it bigraded} commutative algebras.

For a word $w \in R$, let $|w|_x$ (resp., $|w|_y,|w|_{\theta}$) denote the number of occurrences of $x$ (resp, $y,\theta$) in $w$. Note that the homological degree of $w$ is $|w|_{\theta}$. Let $|w|:=|w|_x+|w|_y+|w|_{\theta}$ be the length of $w$ and keep the notation $\,\mathrm{Tr}_{2n+1}\,$
for the composite map
$\,R \onto  R_{\natural, \sigma} \to
k[\Rep_{\mfo_{2n+1}}(\mathfrak{L})]^{\mathrm{SO}_{2n+1}}\,$.

\blemma \la{hcoftr}
$ (\Phi \circ \mathrm{Tr}_{2n+1})(w)\,=\, (1+{(-1)}^{|w|})\sum_{i=1}^n x_i^{|w|_x}y_i^{|w|_y}\theta_i^{|w|_{\theta}} \ $ for any $w \in R$.
\elemma

In particular, if $|w|$ is of odd length, or is of homological degree $\geq 2$, then $(\Phi \circ \mathrm{Tr}_{2n+1})(w)$ vanishes. Lemma~\ref{dhompoly} and Lemma~\ref{hcoftr} enable us to explicitly compute the map
$$\H_{\bullet}(\Phi) \circ \mathrm{Tr}_{2n+1}\,:\,\rHD_{\bullet}(A) \rar k[x_1,\ldots,x_n,y_1,\ldots,y_n,\theta_1,\ldots,\theta_n]^{(\Z_2)^n \rtimes S_n}\,\text{.}$$

\blemma \la{hcoftrforms}
We have
\begin{eqnarray*}
(\H_{\bullet}(\Phi) \circ \mathrm{Tr}_{2n+1})(x^ky^l) \,=\, 2\sum_{i=1}^n x_i^ky_i^l \,,\,\,\,\,\text{ for } k+l\text{ even}\\
(\H_{\bullet}(\Phi) \circ \mathrm{Tr}_{2n+1})(x^ky^ldx)\,=\, 2l \sum_{i=1}^n x_i^ky_i^{l-1}\theta_i\,,\,\,\,\,\text{ for } k+l\text{ even}\\
\end{eqnarray*}
\elemma

\begin{theorem} \la{oddorthlimit}
$(\mathrm{i})$ $\Phi\,:\,  k[\Rep_{\mfo_{2\infty+1}}(\mathfrak{L})]^{\mathrm{SO}_{2\infty+1}}  \rar k[x_1,\ldots;y_1,\ldots;\theta_1,\ldots]^{(\Z_2)^{\infty} \rtimes S_{\infty}}$ is a quasi-isomorphism.\\

$(\mathrm{ii})$ $\H_{\bullet}(\Phi)\,:\,\H_{\bullet}(\mathfrak{a}, \mfo_{2n+1})^{\mathrm{SO}_{2n+1}} \rar k[x_1,\ldots,x_n,y_1,\ldots, y_n,\theta_1,\ldots,\theta_n]^{(\Z_2)^n \rtimes S_n} $ is a surjection.
\end{theorem}
\begin{proof}
We first prove (i). By Proposition~\ref{stabinvrep}, it suffices to prove that the composite map
$$\begin{diagram} \bSym(R_{\natural, \sigma}) & \rTo^{\cong} & k[\Rep_{\mfo_{2\infty+1}}(\mathfrak{L})]^{\mathrm{SO}_{2\infty+1}}  & \rTo^{\Phi} & k[x_1,\ldots, y_1,\ldots,\theta_1,\ldots,\theta_n]^{(\Z_2)^{\infty} \rtimes S_{\infty}} \end{diagram}$$
is a quasi-isomorphism. This is equivalent to verifying that the map
$$ \begin{diagram} \bSym_k(\rHD_{\bullet}(A)) & \rTo^{\bSym(\H_{\bullet}(\Phi) \circ \mathrm{Tr}_{\infty})} & k[x_1,\ldots;y_1,\ldots;\theta_1,\ldots]^{(\Z_2)^{\infty} \rtimes S_{\infty}} \end{diagram} $$
is an isomorphism. For $\mathbf{a} := (s,p,l) \in \{0,1\} \times \Z_{\geq 0}^2$, let $P_{\mathbf{a}}$ denote the power sum $ \sum_{i \geq 1} x_i^py_i^l\theta_i^s$. Call an element  $\mathbf{a} := (s,p,l)$  of $\{0,1\} \times \Z_{\geq 0}^2$ even if $s+p+l$ is even. By Lemma~\ref{dhompoly}, $\rHD_{\bullet}(A)$ can be identified with the bigraded vector space $V$ spanned by basis vectors $q_{\mathbf{a}}$ for $\mathbf{a}  \in \{0,1\} \times \Z_{\geq 0}^2$  even. The basis element $q_{(1,p,l)}$ corresponds to the form $\frac{1}{2(l+1)}x^py^{l+1}dx$ in $\rHD_1(A)$ and the basis element
$q_{(0,p,l)}$ corresponds to the form $\frac{1}{2}x^py^l$ in $\rHD_0(A)$. Lemma~\ref{hcoftrforms} implies that $\H_{\bullet}(\Phi) \circ \mathrm{Tr}_{\infty}$ maps $q_{\mathbf{a}}$ to the power sum $P_{\mathbf{a}}$ for each even $\mathbf{a}$ in $\{0,1\} \times \Z_{\geq 0}^2$. (i) therefore follows once we verify that $k[x_1,\ldots;y_1,\ldots;\theta_1,\ldots]^{(\Z_2)^{\infty} \rtimes S_{\infty}}$ is isomorphic to the graded symmetric algebra generated by the power sums $P_{\mathbf{a}}$ for $\mathbf{a}$ even.

To see this, note that for any $n \geq 2$, the orbit sum
\begin{equation}
\la{ordsum}
O(x_1^{\alpha_1}\ldots x_n^{\alpha_n}y_1^{\beta_1} \ldots y_n^{\beta_n}\theta_1^{\gamma_1} \ldots \theta_n^{\gamma_n})\,:=\,\sum_{\sigma \in (\Z_2)^n \rtimes S_n} \sigma(x_1^{\alpha_1}\ldots x_n^{\alpha_n}y_1^{\beta_1} \ldots y_n^{\beta_n}\theta_1^{\gamma_1} \ldots \theta_n^{\gamma_n})
\end{equation}
is nonzero iff the triples $\,(\gamma_i, \alpha_i,\beta_i) \,$ are even for $\,1
\leq i \leq n\,$ and $\,(\gamma_i, \alpha_i, \beta_i) \neq
(\gamma_j,\alpha_j,\beta_j)\,$ whenever $\,1 \leq i<j\leq n\,$ and $\,\gamma_i=\gamma_j=1\,$.
Starting with this observation and proceeding as in the proof of Proposition~\ref{sympowsum}, one can show the following analog of Proposition~\ref{sympowsum} without difficulty.

\bprop
\la{sympowsumeven}
{\rm (i)} The homomorphism of bigraded commutative algebras
$$  \bSym_k V \rar k[x_1\,\ldots,x_n,y_1,\ldots,y_n,\theta_1,\ldots,\theta_n]^{(\Z_2)^n \rtimes S_n} \,,\,\,\,\, q_{\mathbf{a}} \mapsto P_{\mathbf{a}}\,\text{ for } \mathbf{a} \text{ even } $$
induces an isomorphism of bigraded vector spaces
$$ \bSym_k^{\leq n} V \,\cong\, k[x_1\,\ldots,x_n,y_1,\ldots,y_n,\theta_1,\ldots,\theta_n]^{(\Z_2)^n \rtimes S_n} \,\text{.}$$
{\rm (ii)}
The homomorphism
$$\bSym_k V \rar k[x_1,\ldots;y_1,\ldots;\theta_1,\ldots]^{(\Z_2)^{\infty} \rtimes S_{\infty}}\,,\,\,\,\, q_{\mathbf{a}} \mapsto P_{\mathbf{a}}$$
is an isomorphism of bigraded commutative algebras.
\eprop

Proposition~\ref{sympowsumeven} (ii) is precisely what we needed to verify to complete the proof of (i). Proposition~\ref{sympowsumeven} (i) implies that  that the natural map
$$k[x_1,\ldots; y_1,\ldots;\theta_1,\ldots]^{(\Z_2)^{\infty} \rtimes S_{\infty}} \rar k[x_1,\ldots, x_n,y_1,\ldots,y_n,\theta_1,\ldots,\theta_n]^{(\Z_2)^{n} \rtimes S_{n}}$$ is surjective. (ii) is now immediate from (i).
\end{proof}

\subsubsection{The even orthogonal and symplectic cases}

The DG algebras on both sides of the Harish-Chandra homomorphism in the even orthogonal case coincide in the limit with their odd  counterparts. Hence, the analog of Theorem~\ref{oddorthlimit} (i) holds in the even orthogonal case as well. The Weyl group for $\mfo_{2n}$ is however, $(\Z_2)^{n-1} \rtimes S_n$ rather than $(\Z_2)^n \rtimes S_n$. Here, $(\Z_2)^{n-1}$ is the subgroup of $(\Z_2)^n$ comprising those elements that flip the signs of an even number of the Cartan basis elements $h_i$.
In this case, it turns out that that the map
\begin{equation} \la{polyinftofin} k[x_1,\ldots; y_1,\ldots;\theta_1,\ldots]^{(\Z_2)^{\infty} \rtimes S_{\infty}} \rar k[x_1,\ldots x_n, y_1,\ldots,y_n,\theta_1,\ldots,\theta_n]^{(\Z_2)^{n-1} \rtimes S_{n}} \end{equation}
is not surjective for $n \geq 2$. For example, the element $x_1\ldots x_n$ of $k[x_1,\ldots,x_n, y_1,\ldots,y_n,\theta_1,\ldots,\theta_n]^{(\Z_2)^{n-1} \rtimes S_{n}}$ is not in the image of~\eqref{polyinftofin}. The argument we used to deduce Part (ii) of Theorem~\ref{oddorthlimit}
from Part (i) in the odd orthogonal case does not therefore work in the even case.

\vspace{2ex}

\begin{remark}
Note, however, that the elements $\,x_1 \dots x_n\,$, and similarly, the elements of the form
$$
O(x_1\dots x_k \,y_{k+1} \dots y_n)\qquad \mbox{and}\qquad
O(x_1\dots x_k\, y_{k+1} \dots y_{n-1}\,\theta_n)\ ,
$$
where $\,O(\dots)\,$ denotes the orbit sum as in \eqref{ordsum}, do not lead to trivial counterexamples to Conjecture~\ref{conj3}. In fact, these elements are the images of
the following `Pfaffian' cocycles in $\, k[\Rep_{\mfo_{2n}}(\mathfrak{L})]^{\mathrm{SO}_{2n}} $:
$$
\sum_\sigma (-1)^\sigma \,x_{\sigma(1)\sigma(2)} \,\dots \,x_{\sigma(2k-1)\sigma(2k)}\,
y_{\sigma(2k+1)\sigma(2k)}\, \dots \, y_{\sigma(2n-1)\sigma(2n)}
$$
and
$$
\sum_\sigma (-1)^\sigma \,x_{\sigma(1)\sigma(2)}\,\dots \,x_{\sigma(2k-1)\sigma(2k)}\, y_{\sigma(2k+1)\sigma(2k)}\, \dots\, y_{\sigma(2n-3)\sigma(2n-2)}\,\theta_{\sigma(2n-1)\sigma(2n)}\,.
$$
A small computation shows that the last element is indeed a cocycle, and it cannot be exact,
since its expression contains no repeated indices, while repeated indices are necessarily
introduced by the differential (recall that the differential of $\theta_{ij}$ is given by
$ \sum_k (x_{ik}y_{kj}-y_{ik}x_{kj})$).
\end{remark}

\vspace{2ex}

%
%

The case of $\,\mathfrak{sp}_{2n}\,$ is analogous to that of $\mfo_{2n+1}$. The only modification here is a different involution on $\M_{2n}(k)$: it is given by
$\, X \mapsto \M_{2n}^tX^t\M_{2n}\,$,
where $\M_{2n}$ is the matrix of the nondegenerate skew-symmetric bilinear form
$Q$ on $k^{2n}$ satisfying
\begin{eqnarray*}
 Q(e_{2i-1}, e_{2i}) \,=\, -Q(e_{2i}, e_{2i-1}) &=& 1\,,\,\,\,\, 1 \leq i \leq n\\
  Q(e_i, e_j) &=& 0 \,\,\,\, \text{otherwise.}
\end{eqnarray*}
In particular, the obvious analogues of both parts of Theorem~\ref{oddorthlimit} hold in the symplectic case.

\section{Macdonald conjectures}
\la{Macd}
In this section, we explain how our Conjecture~\ref{conj3} is related to the strong Macdonald
conjecture proved in \cite{FGT}. The key point is to consider the $G$-equivariant derived commuting scheme $\DRep_{\g}(\mathfrak{a})^G$ of the two-dimensional abelian Lie algebra $\mathfrak{a}$
with `shifted' homological degrees. Following the conventions of~\cite{FGT}, we will
assume in this section that $k=\c$. As in the previous section, $ \g $ will denote an arbitrary finite-dimensional reductive Lie algebra with associated complex algebraic group $ G $.

\subsection{Graded commuting schemes}

Let $\mathfrak{a} \,:=\, \c.u \oplus \c.v$ be a homologically graded Lie algebra with trivial bracket, where $u$ has degree $-1$ and $v$ has degree $-2$.
Let $ z \in (\mathfrak{a}[1])^* $ and $ w \in (\mathfrak{a}[1])^* $ denote basis vectors dual to
$ su \in \mathfrak{a}[1] $ and $ sv \in \mathfrak{a}[1] $ respectively. Thus $z$ and $w$ have homological degrees $0$ and $1$. We equip $\mathfrak{a}$ with a $\Z^2$-weight grading by setting the weight of $u$ to be $(1,0)$ and that of $v$ to be $(0,1)$. By convention, the weight of the dual of a finite-dimensional weight-homogeneous vector space $W$
will coincide with that of $W$. Also, in this section, $(\,\mbox{--}\,)^{\ast}$ shall mean bigraded dual.

\blemma \la{drepgl}
$\DRep_{\g}(\mathfrak{a})^G \,\cong\, \CE(\g[z, w],\g; \c)^{\ast}$.
\elemma

\begin{proof}
Since $\mathfrak{a}$ is abelian, a cofibrant resolution of $\mathfrak{a}$ in $\mathtt{DGL}_{k}$ is given by $\mathbf{\Omega}_{\mathtt{Comm}}(\bSym^{c}(\mathfrak{a}[1]))$. Note that the bigraded dual of $\bSym^{c}(\mathfrak{a}[1])$ is exactly $\c[z,w]$. By Theorem~\ref{drepg},
 $$\DRep_{\g}(\mathfrak{a})^G \,\cong\, \CE^c(\g^{\ast}(\bSym^{c}(\mathfrak{a}[1])), \g^{\ast}; \c) \,\text{.}$$
The desired lemma now follows from the isomorphism~\eqref{duality2}.
\end{proof}

We have the following restatement of the (now proven) strong Macdonald conjecture (see~\cite[Theorem 1.5]{FGT}) in terms of derived representation schemes.

\begin{theorem}
\la{macdonqism}
The map
$$ \bSym_k[\Tr_{\g}(\mathfrak{a})]\,:=\,\bSym_k[\oplus_{i=1}^l \Tr^{(d_i)}_{\g}]\,:\, \bSym_k[\oplus_{i=1}^l \L\lambda^{(d_i)}(\mathfrak{a})] \rar \DRep_{\g}(\mathfrak{a})^G$$
is a quasi-isomorphism. In particular, the induced map
\begin{equation} \la{fte1}
\bSym_k[\Tr_{\g}(\mathfrak{a})]\,:\,\bSym_k[\oplus_{i=1}^l \HC_{\bullet}^{(d_i)}(\mathtt{Lie},\mathfrak{a})] \rar \H_{\bullet}(\mathfrak{a}, \g)^G
\end{equation}
is an isomorphism of graded algebras.
\end{theorem}

\bproof
Since the bigraded dual of $\bSym^{c}(\mathfrak{a}[1])$, namely $\,\c[z,w]$, is a smooth graded commutative algebra, Proposition~\ref{phodgedual} applies. Thus, the map~\eqref{fte1}
becomes
  \begin{equation} \la{fte}
\bSym_k[\Tr_{\g}(\mathfrak{a})]\,:\, \bSym_k[\oplus_{i=1}^l \rHC_{\bullet-1}^{(m_i)}(\c[z,w])^{\ast}] \rar \CE(\g[z, w],\g; \c)^{\ast}\,,\end{equation}
where $m_i$, $1 \leq i \leq r$ are the exponents of $\g$. It turns out that~\eqref{fte} is precisely the map in~\cite[Equation 3.1]{Te}:~\cite[Theorem 1.5]{FGT} states that this map is a quasi-isomorphism.
\eproof

\bcor \la{rephommacd} $\H_{\bullet}(\mathfrak{a}, \g)^G$ is isomorphic to the symmetric algebra with one generator in homological degree $-(2m+1)$ and weight $(n+1,m)$ and one generator of homological degree $-(2m+2)$ and weight $(n,m+1)$ for each exponent $m$ of $\g$ and each $n \geq 0$.
\ecor

\bproof
Indeed,~\cite[(1.7)]{FGT} shows that $\rHC_{2m+1}^{(m)}(\c[z,w])$ is identified with the (weight graded) vector space $\c[z].w(dw)^m$ for each exponent $m$ of $\g$. The duals of the basis elements $z^nw(dw)^m$ of $\rHC_{2m+1}^{(m)}(\c[z,w])$ give generators of $ \bSym_k[\oplus_{i=1}^l \rHC_{\bullet-1}^{(m_i)}(\c[z,w])^{\ast}]$ having homological degree $-2m-2$ and weight $(n,m+1)$ for each $n \geq 0$ and for each exponent $m$ of $\g$. Similarly,~\cite[(1.7)]{FGT} shows that for each exponent $m$ of $\g$, $\rHC_{2m}^{(m)}(\c[z,w])$ is identified with the (weight graded) vector space
$ \c[z].dz w(dw)^{m-1}$. The duals of $z^ndzw(dw)^{m-1}$ give generators of $ \bSym_k[\oplus_{i=1}^l \rHC_{\bullet-1}^{(m_i)}(\c[z,w])^{\ast}]$ having homological degree $-2m-1$ and weight $(n+1,m)$. Theorem~\ref{macdonqism} (more precisely, its proof) then implies the desired corollary.
\eproof

\subsection{Lie algebra (co)homology and parity}
For a fixed pair of integers $ (p,r) \in \Z^2 $, let $\mathfrak{a}_{p,r}\,:=\,\c.u \oplus \c.v$ denote the homologically graded abelian Lie algebra, with $u$ and $v$ having degrees $p$ and $r$ respectively. We now demonstrate that much of the behavior of $\DRep_{\g}(\mathfrak{a}_{p,r})^G$ depends only on the parity of $p$ and $r$.

\subsubsection{Functors on complexes}
Let $\Gamma$ be an abelian group. Let $\Gamma\text{-}\Com$ denote the category of $\Gamma$-weight graded complexes of $\c$-vector spaces. In other words, any $V \,\in\,\Gamma\text{-}\Com$ is a direct sum of subcomplexes
$$V\,:=\,\oplus_{\gamma \in \Gamma} V_{\gamma} \,\text{.}$$
Here, $V_{\gamma}$ is the subcomplex of $V$ of weight $\gamma$. Let $\Gamma\text{-}\Com_2$ denote the category of $\Gamma$-weight graded $\Z_2$-homologically graded complexes. One has a functor
$$ \Pi\,:\,\Gamma\text{-}\Com \rar \Gamma\text{-}\Com_2\,,\quad \,V \mapsto (\oplus_{n \in \Z} V_{2n} \rightleftarrows \oplus_{n \in \Z} V_{2n+1})\,\text{.}$$
The functor $\Pi$ (for ``parity") remembers only the parity of the homological grading while retaining the $\Gamma$-weight grading. Note that the functor $\Pi$ is faithful but not full. For most of this section, $\Gamma\,=\,\Z^2$. The following lemma is obvious.
\blemma \la{par}
A morphism $\phi\,\in\,\Gamma\text{-}\Com$ is a quasi-isomorphism iff $\Pi(\phi)$ is a quasi-isomorphism in $\Gamma\text{-}\Com_2$.
\elemma

For any $k \geq 1$, there is a functor
$$ \mathrm{F}_k\,:\, \Z^2\text{-}\Com \rar \Z\text{-}\Com\,,$$
that assigns to a $\Z^2$-weight graded complex $V\,=\,\oplus_{(a,b)\,\in\,\Z^2} V_{(a,b)}$ the $\Z$-weight graded complex $\mathrm{F}_k(V)$ with component of weight $r$ being $\mathrm{F}_k(V)_r\,=\,\oplus_{a+kb\,=\,r} V_{(a,b)}$. Note that  $\mathrm{F}_k$ only changes weights without changing the homological grading or differential.

\subsubsection{}
 As before, we equip $u$ and $v$ with $\Z^2$-weights $(1,0)$ and $(0,1)$ respectively. The proof of the following lemma is essentially the same as that of Lemma~\ref{drepgl}.
\blemma
\la{drepglgen}
There is an isomorphism of DG algebras
$$
\DRep_{\g}(\mathfrak{a}_{p,r})^G \,\cong\, \CE(\g[z,w],\g; \c)^{\ast}\ ,
$$
where $z$ has degree $-(p+1)$ and weight $(1,0)$ and $w$ has  degree $-(r+1)$ and weight $(0,1)$.
\elemma
Hence
\bprop
\la{parity}
$(i)$ The weighted Euler characteristic of
$\DRep_{\g}(\mathfrak{a}_{p,r})^G $ depends only on the parities of $p$ and $r$.\\
$(ii)$ If either the derived Harish-Chandra map
$$ \Phi_{\g} \,:\,\DRep_{\g}(\mathfrak{a})^G \rar \bSym_{\c}(\h^{\ast} \otimes \overline{\bSym^c}(\mathfrak{a}[1]))^W $$
or the Drinfeld trace map
$$\bSym_k[\Tr_{\g}(\mathfrak{a})]\,:\, \bSym_k[\oplus_{i=1}^l \L \lambda^{(d_i)}(\mathfrak{a})] \rar \DRep_{\g}(\mathfrak{a})^G $$
is a quasi-isomorphism for $\mathfrak{a}_{p,r}$, then it is a quasi-isomorphism for $\mathfrak{a}_{p',r'}$ whenever $p, p'$ and $r, r'$ have the same parity.
\eprop
\begin{proof}
 Note that if $p,p'$ and $r,r'$ have the same parity,
  $$ \Pi[\Phi_{\g}(\mathfrak{a}_{p,r})]\,=\,\Pi[\Phi_{\g}(\mathfrak{a}_{p',r'})]\,,\,\,\,\,\, \Pi(\bSym_k[\Tr_{\g}(\mathfrak{a}_{p,r})])\,=\,\Pi(\bSym_k[\Tr_{\g}(\mathfrak{a}_{p',r'})])\,\text{.}$$
 The desired proposition therefore follows from Lemma~\ref{par}
\end{proof}
Thus, when $p, r$ are both {\it even}, we expect that the derived Harish-Chandra homomorphism
gives the quasi-isomorphism
$$
\DRep_{\g}(\mathfrak{a}_{p,r})^G \stackrel{\sim}{\to}
\bSym(\h^{\ast}[p] \oplus \h^{\ast}[r] \oplus \h^{\ast}[p+r+1])^W
$$
and we get Conjecture~\ref{conj3} and Conjecture~\ref{conj4}.

On the other hand, when $p,r$ are of
{\it opposite} parity, by Theorem~\ref{macdonqism}, the Drinfeld trace map must be a quasi-isomorphism, so that the homology
of $ \DRep_{\g}(\mathfrak{a}_{p,r})^G $ is a free graded commutative algebra: and we get the classical $(q,t)$-Macdonald identity (see Section~\ref{eulermacd} below).

\subsection{Euler characteristics} \la{eulermacd}

Let $B_{\g}\,:=\, \CE(\g[z,w]/\g;\c)$. It follows from Lemma~\ref{drepgl} that the weighted Euler characteristic of $\DRep_{\g}(\mathfrak{a})^G $ coincides with that of its bigraded dual, namely, $B_{\g}^{\ad\,\g}\,=\,\CE(\g[z,w],\g;\c)$.  We compute this Euler characteristic by computing the character valued Euler characteristic
$$\chi(B_{\g},q,t,e^h)\,:=\, \sum_{a,b \geq 0}\sum_{i \in \Z} (-1)^i \mathrm{Tr}(e^h|_{(B_{\g})_{a,b}})q^at^b \,\text{.}$$
\blemma \la{chimd1}
$$
 \chi(B_{\g},q,t,e^h)\,=\, \prod_{n \geq 0} \left[\left(\frac{1-q^{n+1}}{1-q^nt}\right)^l \,\prod_{ \alpha \in R} \frac{1-q^{n+1}e^{\alpha}}{1-q^nte^{\alpha}}\right] \,\text{.}
$$
where $l:=\dim_{\c}\,\h$ is the rank of $\g$.
\elemma

\bproof
The proof is identical to that of Lemma~\ref{eulie1}.
\eproof

Note that $\,B_{\g}^G = B_{\g}^{\ad\,\g}\,$, where $G$ is the complex reductive Lie group whose Lie algebra is $\g$. Hence,
\bcor \la{chimd2}
\begin{equation} \la{macde0}
\chi(B_{\g}^{\ad\,\g},q,t)\,=\, \frac{1}{|W|} \prod_{n \geq 0} \left(\frac{1-q^{n+1}}{1-q^nt}\right)^l\, \mathrm{CT}
\left\{\prod_{n \geq 0} \prod_{ \alpha \in R} \frac{1-q^{n}e^{\alpha}}{1-q^nte^{\alpha}} \right\} \text{.}
\end{equation}
\ecor

\bproof
The proof is similar to the proof of Corollary~\ref{corlie1} from Lemma~\ref{eulie1}.
\eproof

As a consequence, we obtain the following proposition:

\bprop
The following identity holds:
\begin{equation}
\la{qtid}
\frac{1}{|W|} \mathrm{CT}\left\{\prod_{n \geq 0} \prod_{ \alpha \in R} \frac{1-q^{n}e^{\alpha}}{1-q^nte^{\alpha}} \right\}
  =  \prod_{n \geq 0} \prod_{i=1}^l  \frac{(1-q^nt)(1-q^{n+1}t^{m_i})}{(1-q^{n+1})(1-q^nt^{m_i+1})} \,,
\end{equation}
where $m_i\,:=\,d_i-1$ are the exponents of the Lie algebra $\g$.
\eprop

\bproof
Let $l$ bs the rank of $\g$. The Euler characteristic of $\DRep_{\g}(\mathfrak{a})^G$ equals that of its homology. By Corollary~\ref{rephommacd}, the latter Euler characteristic is equal to
\begin{equation} \la{macde1}
\prod_{n \geq 0} \prod_{i=1}^l  \frac{1-q^{n+1}t^{m_i}}{1-q^nt^{m_i+1}} \ .
\end{equation}
On the other hand, by Lemma~\ref{drepgl}, the Euler characteristic of $\DRep_{\g}(\mathfrak{a})^G$ is equal to that of $B_{\g}^{\ad\,\g}$. By Corollary~\ref{chimd2}, the Euler characteristic of $B_{\g}^{\ad\,\g}$ also equals the right hand side of~\eqref{macde0}. Equating the right hand side of~\eqref{macde0} with the expression~\eqref{macde1} and multiplying both sides by $\prod_{n \geq 0} \left(\frac{1-q^{n+1}}{1-q^nt}\right)^l$, we obtain the desired identity.
\eproof

\noindent
\textbf{Remark}.
The left-hand side of Eq.~\eqref{qtid} is exactly the expression denoted by
$\frac{1}{|W|} [\Delta_{q^{\frac{1}{2}}, t^{\frac{1}{2}}}]_0 $ in~\cite{Kir}. Thus
\eqref{qtid} is equivalent to the standard $(q,t)$-version of Macdonald's constant term identity
(see~\cite[Eq. ~2.7]{Kir}), which is the special case of the inner product identity  for
$ P_{\lambda} = 1 $ (see~\cite[Theorem 2.4]{Kir}).

\subsubsection{The $q$-Macdonald identity}. Let $V\,\in\,\Z^2\text{-}\Com$ be such that the subcomplexes $V_{a,b}$ of weight $(a,b)$ are finite-dimensional and nonzero only for $a,b \geq 0$. Recall the definition of the functor $\mathrm{F}_k\,:\, \Z^2\text{-}\Com \rar \Z\text{-}\Com$ above. There is an equality of Euler characteristics
\begin{equation*} \chi(V, q, q^k) \,=\, \chi(\mathrm{F}_k(V), q)\,\text{.} \end{equation*}
As a result, we get
\blemma \la{qmacd}
$$ \chi(\CE(\g[z]/z^k, \g;\c), q)\,=\,\chi(B_{\g}^{\ad\,\g}, q, q^k)\,\text{.}$$
\elemma
\bproof
Consider the DG Lie algebra
$ \g_{\partial}[z,w]\,:=\,(\g[z,w]\,,\,\,\partial w=z^k)$, which is quasi-isomorphic to $ \g[z]/z^k$. Hence, $\CE(\g_{\partial}[z,w], \g;\c)$ is quasi-isomorphic to $\CE(\g[z]/z^k, \g;\c)$. Note that the differential on $\CE(\g_{\partial}[z,w], \g;\c)$ is obtained by twisting that on $B_{\g}^{\ad\,\g}\,=\,\CE(\g[z,w], g;k)$ with the differential induced by $\partial$. The latter differential also preserves weights provided we apply the functor $\mathrm{F}_k$ to $B_{\g}^{\ad\,\g}$. Hence,
$$ \chi(\CE(\g[z]/z^k, \g;\c), q)\,=\, \chi(\CE(\g_{\partial}[z,w], \g;\c), q)\,=\, \chi(\mathrm{F}_k(B_{\g}^{\ad\,\g}), q)\,=\, \chi(B_{\g}^{\ad\,\g}, q, q^k)\,\text{.}$$
This finishes the proof of the lemma.
\eproof
On the other hand, the identity~\eqref{qtid} implies that
\blemma
The following identity holds:
\begin{equation}
\la{qid}
\frac{1}{|W|} \mathrm{CT}\left\{\prod_{j=0}^{k-1} \prod_{ \alpha \in R} (1-q^je^{\alpha}) \right\}
= \prod_{i=1}^l  \prod_{j=1}^{k-1} \frac{1-q^{km_i+j}}{1-q^j}
=  \prod_{i=1}^l  \binom{kd_i-1}{k-1}_q
\end{equation}
\elemma
\bproof
Putting $t=q^k$ in~\eqref{qtid}, we obtain~\eqref{qid}. Equivalently, we can multiply both sides of~\eqref{qid} by $ \prod_{j=1}^{k-1} (1-q^j)^l$ to obtain:
\begin{equation} \la{macde2}
\frac{1}{|W|}\prod_{j=1}^{k-1} (1-q^j)^l\, \mathrm{CT}\left\{ \prod_{j=0}^{k-1} \prod_{ \alpha \in R} (1-q^je^{\alpha}) \right\} = \prod_{i=1}^l  \prod_{j=1}^{k-1} (1-q^{km_i+j})
\end{equation}
By Corollary~\ref{chimd2} and Lemma~\ref{qmacd}, the left-hand side of the above identity is the Euler characteristic of $\CE(\g[z]/z^k, \g;\c)$. That this Euler characteristic equals the right hand side of the above identity follows from~\cite[Theorem A]{FGT}, which explicitly computes $\H_{\bullet}(\g[z]/z^k, \g;\c)$ as a free bigraded commutative algebra and specifies the weights and homological degrees of a set of homogeneous generators. We remind the reader that~\cite[Theorem A]{FGT} is proven from~\cite[Theorem 1.5]{FGT} by replacing $\g[z]/z^k$ by the quasi-isomorphic DG Lie algebra $\g_{\partial}[z,w]$ and computing $\H_{\bullet}(\g_{\partial}[z,w],g;\c)$ by appealing to~\cite[Theorem 1.5]{FGT} (which computes $\H_{\bullet}(\g[z,w],\g;\c)$) and using a simple spectral sequence argument.
\eproof

\noindent
\textbf{Remark}. Being the Euler characteristic of $\CE(\g[z]/z^k, \g;\c)$, the left-hand side of~\eqref{macde2}  can be rewritten as
$$
\int_G \,\prod_{j=1}^{k-1} \mathrm{det}(1-q^j\mathrm{Ad}g) dg \ .
$$
The identity~\eqref{qid} is thus equivalent to
\begin{equation} \la{macde3}
\int_G \, \prod_{j=1}^{k-1} \mathrm{det}(1-q^j\mathrm{Ad}g) dg =  \prod_{i=1}^l  \prod_{j=1}^{k-1} (1-q^{km_i+j})\ ,
\end{equation}
which is a version of the original Macdonald identity (see ~\cite[Conjecture 3.1$'\,$]{M}).

\vspace{2ex}

\appendix

\section{Derived representation schemes of algebras over an operad}
\la{appendix}
In this Appendix, we construct (derived) representation schemes for (DG) algebras
over an arbitrary binary quadratic operad. Our generalization covers the representation schemes of associative algebras studied in \cite{BKR, BR}, the representation algebras defined in~\cite{Tu} as well as the representation schemes of Lie algebras introduced in Section~\ref{sec6}. We also construct canonical trace maps from operadic cyclic homology to representation homology, generalizing the derived character maps of \cite{BKR}. For cyclic operads our notion of operadic cyclic homology agrees with that of~\cite{GK}. The main result of this section, Theorem~\ref{dreplieoperads}, unifies Theorem~\ref{main} and Theorem~\ref{drepg}.

Throughout this section, $\,\mathcal P$ will denote a finitely generated binary quadratic operad
and $\mathcal Q = {\mathcal P}^{!} $ will stand for its (quadratic) Koszul dual. Unlike Getzler and Kapranov~\cite{GK} who work with cyclic operads, which are not necessarily binary quadratic, we will work with binary quadratic operads, which are not necessarily cyclic. This is due to the following fact that we need for Theorem~\ref{dreplieoperads}:  if $A$ is a DG $\mathcal P$-algebra and if $B$ is a DG $\mathcal Q$-algebra then $A \otimes B$ has a natural DG Lie algebra structure (see~\cite[Prop. 7.6.5]{LV}).

\subsection{Internal $\Hom$-functor and the representation functor}

 Let $\mathtt{DGPA}$ (resp., $\mathtt{DGQA}$) denote the category of DG $\mathcal P$-algebras (resp., DG $\mathcal Q$-algebras) over $k$. If $\mathfrak{C}$ is a DG $\mathcal P$-coalgebra, the complex $\mathbf{Hom}(\mathfrak{C}, \bar{B}) $
naturally acquires the structure of a DG $\mathcal P$-algebra for any $B \in \cDGA_{k/k}$.  We therefore have a functor
 $$ \mathbf{Hom}_{\ast}(\mathfrak{C},\,\mbox{--}\,)\,:\, \cDGA_{k/k} \rar \mathtt{DGPA} \,,\,\,\,\, B \mapsto \mathbf{Hom}(\mathfrak{C},\bar{B})\,, $$
where $\ast$ indicates that we form the convolution algebra with the augmentation ideal $\bar{B}$. Let $A\,\in\,\mathtt{DGPA}$. Each generating operation $m\,\in\,\mathcal P(2)$ gives $k$-linear maps
$$m_A\,:\,A \otimes A \rar A\,,\,\,\,\, m_{\mathfrak{C}}\,:\,\mathfrak{C} \rar \mathfrak{C} \otimes \mathfrak{C}\,\text{.}$$
For $a,b\,\in\,A$ and $x \in \mathfrak{C}$, define $(m_{\mathfrak{C}}(x),a,b)\,\in\,\bSym_k({\mathfrak{C}} \otimes A)$ to be the image of the element $x \otimes a\otimes b\,\in\,{\mathfrak{C}} \otimes A \otimes A$ under the composite map
$$\begin{diagram} {\mathfrak{C}} \otimes A \otimes A & \rTo^{\,\,\,\,m_{\mathfrak{C}} \otimes \id \otimes \id\,\,\,\,} & {\mathfrak{C}} \otimes {\mathfrak{C}} \otimes A \otimes A & \rTo^{\,\,\,\,\id \otimes \tau_{23} \otimes \id\,\,\,\,} & ({\mathfrak{C}} \otimes A) \otimes ({\mathfrak{C}} \otimes A) & \rOnto^{\,\,\mathtt{can}\,\,} & \bSym^2({\mathfrak{C}} \otimes A) & \rInto & \bSym_k({\mathfrak{C}} \otimes A) \end{diagram}$$
The proof of the following proposition is a straightforward generalization of that of Proposition~\ref{lierep}. We therefore, leave it to the interested reader.
\bprop
\la{genoperadrep}
The functor $\mathbf{Hom}_{\ast}({\mathfrak{C}},\,\,\mbox{--}\,)\,:\,\cDGA_{k/k} \rar \mathtt{DGPA} $ has a left adjoint
$${\mathfrak{C}} \ltimes \,\mbox{--}\,\,:\,\mathtt{DGPA}_k \rar \cDGA_{k/k}\,,\,\,\,\, A \mapsto {\mathfrak{C}} \ltimes A \,:=\,\bSym_k({\mathfrak{C}} \otimes A)/{\mathcal I}_{A,{\mathfrak{C}}}\,,$$
where $\mathcal I_{A,{\mathfrak{C}}}$ is the ideal generated by the elements
$$ x \otimes m_A(a,b)-(m_{\mathfrak{C}}(x),a,b)\,, m \in \mathcal P(2)\,,\,x\,\in\,{\mathfrak{C}}\,,\,a,b \,\in A\,,\text{.}$$
In particular, there is a natural isomorphism
\begin{equation} \la{natadjoperads} \Hom_{\mathtt{DGPA}}(A, \mathbf{Hom}_{\ast}({\mathfrak{C}}, B))\,\cong\, \Hom_{\cDGA_{k/k}}({\mathfrak{C}} \ltimes A ,B)\,\text{.} \end{equation}
\eprop

The category $\mathtt{DGPA}$ has a natural model structure where fibrations are degreewise surjections and weak equivalences are quasi-isomorphisms (see~\cite{H}). It is easy to verify that the functor $\mathbf{Hom}_{\ast}({\mathfrak{C}},\,\,\mbox{--}\,)$ preserves fibrations (i.e, degreewise surjections) and acyclic fibrations. Hence,

\begin{theorem} \la{augrepm}
$\mathrm{(a)}$ The following functors form a Quillen pair:
$${\mathfrak{C}} \ltimes \,\mbox{--}\,\,:\,\mathtt{DGPA} \rightleftarrows \cDGA_{k/k}\,:\,\mathbf{Hom}_{\ast}({\mathfrak{C}},\,\mbox{--}\,) \,\text{.}$$

$\mathrm{(b)}$ The functor ${\mathfrak{C}} \ltimes \,\mbox{--}\, \,:\,\mathtt{DGPA} \rightleftarrows \cDGA_{k/k}$ has a (total) left derived functor
$$ {\mathfrak{C}} \stackrel{\L}{\ltimes} \,\mbox{--}\, :\,\Ho(\mathtt{DGPA}) \rar \Ho(\cDGA_{k/k})\,,\,\,\,\,A \mapsto {\mathfrak{C}} \ltimes QA\,,$$
where $QA \stackrel{\sim}{\twoheadrightarrow} A$ is any cofibrant resolution of $A$ in $\mathtt{DGPA}$.

$\mathrm{(c)}$ For any $A \in \mathtt{DGPA}$ and $B \in \cDGA_{k/k}$, there is a canonical isomorphism
$$ \Hom_{\Ho(\mathtt{DGPA})}(A, \mathbf{Hom}({\mathfrak{C}}, \bar{B}))\,\cong\,\Hom_{\Ho(\cDGA_{k/k})}({\mathfrak{C}} \ltimes A , B)\,\text{.}$$
\end{theorem}

We define \emph{derived representation scheme} of $A$ over coalgebra $\mathfrak{C}$ by
$$ \DRep_{\mathfrak{C}}(A) \,:=\, {\mathfrak{C}} \stackrel{\L}{\ltimes} A \,\in \Ho(\cDGA_{k/k})$$
and the \emph{representation homology} of $A$ over ${\mathfrak{C}}$ by
$$ \H_{\bullet}(A, {\mathfrak{C}})\,:=\,\H_{\bullet}[\DRep_{\mathfrak{C}}(A)]\,\text{.}$$

\subsection{Representation homology vs Lie homology: the operadic setting} \la{repoperadalg}

Recall that $\mathcal Q$ denotes the Koszul dual of the operad $\mathcal P$. The main result of this section, Theorem~\ref{dreplieoperads},  identifies the representation homology $\H_{\bullet}(A,{\mathfrak{C}})$ with the homology of the Lie coalgebra ${\mathfrak{C}} \otimes C$ constructed over a $\mathcal Q$-coalgebra $C$ Koszul dual to $A$: this unifies Theorem~\ref{main} and Theorem~\ref{drepg} in the main body of the paper. We begin by recalling some facts on twisting chains and Koszul duality in the operadic setting. The main reference for this material is~\cite{LV}.

\subsubsection{Twisting chains and Koszul duality}

Let $\mathtt{DGQC}$ denote the category of conilpotent DG $\mathcal Q$-coalgebras. Let $A \in \mathtt{DGPA}$ and $C \in \mathtt{DGQC}$. The following proposition is a special case of~\cite[Proposition 11.1.1]{LV}.

\blemma \la{liehomoperad}
$\mathrm{(a)}$ The complex $\mathbf{Hom}(C, A)$ has the natural structure of a DG Lie algebra.

$\mathrm{(b)}$ If ${\mathfrak{C}}$ is a DG $\mathcal P$-coalgebra, the complex ${\mathfrak{C}} \otimes C$ has the natural structure of a (conilpotent) DG Lie coalgebra.
\elemma

\begin{proof}
We note that $\mathbf{Hom}(C, A)$ has the natural structure of an algebra over the Hadamard product $\mathcal P \otimes_{\H} \mathcal Q$ of the operads $\mathcal P$ and $\mathcal Q$ (see~\cite[Section 5.1.12]{LV}). The Lie algebra structure on $\mathbf{Hom}(C, A)$ comes from the morphism of operads $\mathtt{Lie} \rar \mathcal P \otimes_{\H} \mathcal Q$ in~\cite[Prop. 7.6.5]{LV}. This proves (a).

 On the other hand, ${\mathfrak{C}} \otimes C$ is naturally a DG coalgebra over the operad $\mathcal P \otimes_{\H} \mathcal Q$. The DG Lie coalgebra structure on ${\mathfrak{C}} \otimes C$ then comes from the morphism of operads $\mathtt{Lie} \rar \mathcal P \otimes_{\H} \mathcal Q$ in~\cite[Prop 7.6.5]{LV}. The conilpotency of ${\mathfrak{C}} \otimes C$ follows from the conilpotency of $C$. This proves (b).

\end{proof}

A {\it twisting chain} from $C$ to $A$ is a degree $-1$ element $\tau\,\in\, \mathbf{Hom}(C, A)$ satisfying the Maurer-Cartan equation
$$d\tau+\frac{1}{2}[\tau\,,\,\tau] \,=\,0 \,\text{.}$$
Let $\Tw(C, A)$ denote the set of all twisting chains from $C$ to $A$. It can be shown (see~\cite[Prop. 11.3.1]{LV}) that, for a fixed $\mathcal P$-algebra $ A $, the functor
$$
\Tw(\mbox{--}, A):\,\mathtt{DGQC} \to \Sets\ ,\quad C \mapsto \Tw(C,A)\ ,
$$
is representable; the corresponding coalgebra $ \bB_{\mathcal P}(A) \in \mathtt{DGQC} $ is called the {\it bar construction} of $ A \,$ (see~\cite[Sect. 11.2.1]{LV}). Dually, for a fixed coalgebra $ C $, the functor
$$
\Tw(C, \mbox{--}):\,\mathtt{DGPA} \to \Sets\ ,\quad A \mapsto \Tw(C,A)\ ,
$$
is corepresentable; the corresponding algebra $ \bOmega_{\mathcal Q}(C) \in \mathtt{DGPA} $ is called the
{\it cobar construction} of $ C \,$ (see~\cite[Sect. 11.2.5]{LV}).
Thus, we have canonical isomorphisms
\begin{equation}
\la{fismoperad}
\Hom_{\mathtt{DGPA} }(\bOmega_{\mathcal Q}(C), A)\, = \,\Tw(C,A)\, =\, \Hom_{\mathtt{DGQC}}(C, \bB_{\mathcal P}(A))
\end{equation}
showing that $ \bOmega_{\mathcal Q}: \mathtt{DGQC} \rightleftarrows \mathtt{DGQC}: \bB_{\mathcal P} $ are adjoint functors. We say that $C\,\in\,\mathtt{DGQC}$ is {\it Koszul dual} to $A \in \mathtt{DGPA} $ if there is a quasi-isomorphism
$\cb_{\mathcal Q}(C) \stackrel{\sim}{\rar} A $ in $\mathtt{DGPA}$. Note that since $\cb_{\mathcal Q}(C)$ is free as a graded $\mathcal P$-algebra (see~\cite[Sect. 11.2.5]{LV}),
$\cb_{\mathcal Q}(C) \stackrel{\sim}{\rar} A$ is a cofibrant resolution of $A$ in $\mathtt{DGPA}$ when $C$ is Koszul dual to $A$. Further, if $ {\mathcal P} $ is a Koszul operad,
there is at least one $C \in \mathtt{DGQC}$ Koszul dual to $A$, namely,
$\bB_{\mathcal P}(A)$ (see \cite[Sect. 11.3.3]{LV}).

\subsubsection{} For $C\,\in\, \mathtt{DGQC} $ and for any DG $\mathcal P$-coalgebra ${\mathfrak{C}}$, let ${\mathcal Lie}^c({\mathfrak{C}} \otimes C)$ denote ${\mathfrak{C}} \otimes C$ equipped with the Lie coalgebra structure from Lemma~\ref{liehomoperad}. Then,

\begin{theorem} \la{dreplieoperads}
There is a natural isomorphism of functors from $\,\mathtt{DGQC} $ to $ \cDGA_{k/k}$:
$$ ({\mathfrak{C}} \ltimes \,\mbox{--}\,) \circ \cb_{\mathcal Q}(\,\mbox{--}\,) \,\cong\, \cb_{\mathtt{Lie}}[{\mathcal Lie}^c({\mathfrak{C}} \otimes\,\mbox{--}\,)]\, \text{.}$$
As a result, if $C \in \mathtt{DGQC}$ is Koszul dual to $A \in \mathtt{DGPA}$, there are isomorphisms in $\Ho(\cDGA_{k/k})$
$$ \DRep_{\mathfrak{C}}(A)\,\cong\,\C^c(\mathcal Lie^c({\mathfrak{C}} \otimes C); k)\,\text{.}$$
Consequently,
$$\H_{\bullet}(A,{\mathfrak{C}})\,\cong\,\H_{\bullet}({\mathcal Lie}^c({\mathfrak{C}} \otimes C);k) \,\text{.}$$
\end{theorem}

\begin{proof}
Let $C \,\in\,\mathtt{DGQC}$. For any $B \in \cDGA_{k/k}$, we have natural isomorphisms
 \begin{eqnarray*}
                   \Hom_{\cDGA_{k/k}}({\mathfrak{C}} \ltimes \cb_{\mathcal Q}(C), B) & \cong &   \Hom_{\mathtt{DGPA}}(\cb_{\mathcal Q}(C), \mathbf{Hom}({\mathfrak{C}}, \bar{B}))\\
                                                                            & \cong & \Tw(C, \mathbf{Hom}({\mathfrak{C}}, \bar{B}))\\
                                                                            & \cong & \Tw({\mathcal Lie}^c({\mathfrak{C}} \otimes C), \bar{B})\\
                                                                            & \cong & \Hom_{\cDGA_{k/k}}(\cb_{\mathtt{Lie}}[{\mathcal Lie}^c({\mathfrak{C}} \otimes C)], B)
\end{eqnarray*}
The DG Lie algebra structure on $\mathbf{Hom}(C, \mathbf{Hom}({\mathfrak{C}}, \bar{B}))$ comes from Lemma~\ref{liehomoperad} and the fact that $\mathbf{Hom}({\mathfrak{C}}, \bar{B})$ acquires the structure of a $\mathcal P$-algebra. The third isomorphism above is from the isomorphism of DG Lie algebras $\mathbf{Hom}(C, \mathbf{Hom}({\mathfrak{C}}, \bar{B}))\,\cong\, \mathbf{Hom}({\mathcal Lie}^c({\mathfrak{C}} \otimes C), \bar{B})$ coming from hom-tensor duality (and Lemma~\ref{liehomoperad}).  The natural isomorphism
$$ {\mathfrak{C}} \ltimes \cb_{\mathcal Q}(C)\,\cong\, \cb_{\mathtt{Lie}}[{\mathcal Lie}^c({\mathfrak{C}} \otimes C)] \,=\, \C^c(\mathcal Lie^c({\mathfrak{C}} \otimes C); k)$$
follows from the natural isomorphism
$$  \Hom_{\cDGA_{k/k}}({\mathfrak{C}} \ltimes \cb_{\mathcal Q}(C), B)\,\cong\, \Hom_{\cDGA_{k/k}}(\cb_{\mathtt{Lie}}[{\mathcal Lie}^c({\mathfrak{C}} \otimes C)], B) $$
by Yoneda's lemma. This proves the first statement of the desired theorem.

If $C \in \mathtt{DGQC}$ is Koszul dual to $A \in \mathtt{DGPA}$, $\cb_{\mathcal Q}(C) \stackrel{\sim}{\rar} A$ is a cofibrant resolution of $A$. Hence, $\DRep_{\mathfrak{C}}(A)\,\cong\,{\mathfrak{C}} \ltimes \cb_{\mathcal Q}(C)\,\cong\, \C^c(\mathcal Lie^c({\mathfrak{C}} \otimes C); k)$. This proves the second assertion in the above theorem. The final statement in the above theorem now follows immediately.
\end{proof}

\subsection{Examples}

\subsubsection{Derived reresentation schemes for associative algebras}
When $\mathcal P\,=\,\mathcal Q \,=\,\mathtt{Ass}$, the operad governing associative algebras, $\mathtt{DGPA}$ becomes $\DGA$, the category of {\it nonunital} DG $k$-algebras. Recall that there is an equivalence of categories (see Section~\ref{basicdga}):
$$ \DGA \rar \DGA_{k/k} \,,\,\,\,\,A \rar A' \,:=\,k \oplus A\,\text{.}$$
The inverse is the functor that assigns to each augmented DG algebra $A$ its augmentation ideal $\bar{A}$. In this case, Proposition~\ref{natadjoperads} is a generalization of~\cite[Lemma 2.1]{Tu} to the DG setting. Let $\mathfrak{C}$ be a counital DG coalgebra.  Denote the composite functor
$$
\DGA_{k/k} \rar \DGA \xrightarrow{{\mathfrak{C}} \ltimes \,\mbox{--}\,} \cDGA_{k/k}
$$
by ${\mathfrak{C}} \ltimes \,\mbox{--}\,$. For $A \in \DGA_{k/k}$, one has a natural isomorphism
$${\mathfrak{C}} \ltimes A\,\, \,\cong\, \, ({\mathfrak{C}} \rhd A)_{\nn} \,,$$
where
\begin{equation}
\la{swidpr}
{\mathfrak{C}} \rhd A\,:=\,
T_k({\mathfrak{C}} \otimes A)/ \langle m \otimes ab - (m^{(1)} \otimes a) \varotimes (m^{(2)} \otimes b) \,,\ m \otimes 1 - u(m).1 \rangle\
\text{.}
\end{equation}
Here, $ u(m) \in k $ and $ \Delta(m) = \sum m^{(1)} \otimes m^{(2)} $ are the counit and
the coproduct of $ {\mathfrak{C}} $ evaluated at an element $ m \in {\mathfrak{C}} $. This operation is introduced in~\cite[Section 3.4]{AJ}, where it is called the {\it Sweedler product}.

Note that the algebra homomorphisms $\, {\mathfrak{C}} \rhd A \to A \,$ correspond precisely
to the left actions $ {\mathfrak{C}} \otimes A \to A $ commuting with the
multiplication $ A \otimes A \to A $ and the unit map $ k \to A$ on $A$. This clarifies
the meaning of the defining relations and the notation for the algebra $ {\mathfrak{C}} \rhd A $
(cf. \cite{Maj}).

Now, if $\, {\mathfrak{C}} = \M^*_n(k) $ is the finite-dimensional coalgebra dual to the matrix algebra $\M_n(k)$, the functor $ \mathbf{Hom}({\mathfrak{C}}, \,\mbox{--}\,) $  is canonically isomorphic to $ \M_n(\,\mbox{--}\,) \,$, Hence, by adjunction (see Proposition~\ref{Bprop}), we have a natural isomorphism
$$
A_n \,\cong\, \M^*_n(k) \ltimes A\,\text{.}
$$
This allows one to regard the representation functor \eqref{rootab}
as a special case of the general construction of this section.
In particular, all results of Section~\ref{secdrep} (along with Theorem~\ref{main}) follow from results proved in
the this section (e.g., Theorem~\ref{nS2.1t2} is
a special case of Theorem~\ref{augrepm}, Theorem~\ref{main} is a special case of Theorem~\ref{dreplieoperads} etc.)\\

\noindent
\textbf{Remark.} The existence of the noncommutative `lifting' $\mathfrak{C} \rhd A$ of $\mathfrak{C} \ltimes A$ appears to be a special phenomenon that occurs only in the case $\mathcal P\,=\,\mathcal Q\,=\,\mathtt{Ass}$. In this case, the natural morphism of operads $\mathtt{Lie} \rar \mathtt{Ass} \otimes_{\H} \mathtt{Ass}$ factors through the associative operad. Hence, from the point of view of a coalgebra-algebra pair over a Koszul dual pair of operads, it is the abelianization of the Sweedler product (rather than the Sweedler product itself) that is the natural construction.

\subsubsection{Derived representation schemes of Lie algebras}  When $\mathcal P\,=\,\mathtt{Lie}$, the operad governing Lie algebras, then $\mathcal Q\,=\,\mathtt{Comm}$, the operad governing commutative algebras. In this case, the constructions of this section specialize to those of Section~\ref{sec6}. For example, Proposition~\ref{genoperadrep} specializes to Proposition~\ref{lierep}. Theorem~\ref{dreplieoperads} specializes to Proposition~\ref{lierepcobar} and Theorem~\ref{drepg}, etc.

\subsection{$\mathcal P$-cyclic homology} \la{secoptr}

We now construct derived trace maps relating (operadic) cyclic homology to (operadic) representation homology. Our construction generalizes the construction of traces in~\cite{BKR} as well as Drinfeld traces in Section~\ref{secdrinfeldtr}.

\subsubsection{Invariant bilinear forms}

We say that a symmetric bilinear form $B$ on a DG $\mathcal P$-algebra is {\it invariant} if for all $m \,\in\,\mathcal P(2)$,
$$ B(m(a\,,\,b)\,,\,c)\,=\, (-1)^{|a||m|} B(a\,,\,m(b\,,\,c))\,\,\,\,\,\text{ for all } a,b,c\,\in\, A \,\text{.}$$
Let $A_{\n}$ denote the target of the universal invariant bilinear form on $A$: this is equal to the quotient of $A \otimes A$ by the subcomplex spanned of the images of the maps
$$ (m \otimes \id-\id \otimes m)\,:\,A \otimes A \otimes A \rar A \otimes A \,,$$
where $m$ runs over all (homogeneous) elements of $\mathcal P(2)$. If $\mathcal P$ is a {\it cyclic} binary quadratic operad, our notion of an invariant bilinear form agrees with that of~\cite{GK} (cf.~\cite[Prop. 4.3]{GK}). In this case,  $A_{\n}$ is denoted in~\cite{GK} by $\lambda(A)$.

Dually, if $\mathfrak{C}$ is a DG $\mathcal P$-coalgebra, a (homogeneous) $2$-tensor $\alpha$ is called {\it invariant} if it is a (homogeneous) element of $\mathfrak{C} \otimes \mathfrak{C}$ such that for all $m \,\in\,\mathcal P(2)$, the composite map of complexes
$$ \begin{diagram} k[|\alpha|] & \rTo^{\alpha} & \mathfrak{C} \otimes \mathfrak{C} & \rTo^{\,\,\,m \otimes \id -\id \otimes m\,\,\,} & \mathfrak{C} \otimes \mathfrak{C} \otimes \mathfrak{C} \end{diagram} $$
vanishes. We denote the subcomplex of invariant tensors on 
$\mathfrak{C} $ by $\mathfrak{C}^{\n}$.

The following theorem is a generalization of~\cite[Theorem 5.3]{GK} to arbitrary binary quadratic operads. Since its proof is similar to that of Theorem~\ref{dlambda}, we omit it in order to avoid being repetitive.
\begin{theorem}
The functor
$$(\,\mbox{--}\,)_{\n}\,:\,\mathtt{DGPA} \rar \Com_k\,,\,\,\,\, A \mapsto A_{\n} $$
has a (total) left derived functor
$$\L (\,\mbox{--}\,)_{\n}\,:\,\Ho(\mathtt{DGPA}) \rar \Ho(\Com_k) \,,\,\,\,\, A \mapsto R_{\n} $$
for any cofibrant resolution $R \stackrel{\sim}{\rar} A$ in $\mathtt{DGPA}$.
\end{theorem}

We denote the homology $\H_{\bullet}(\L A_{\n})$ by $\HC_{\bullet}(\mathcal P, A)$ and call it the {\it $\mathcal P$-cyclic homology} of $A$. When $\mathcal P\,=\,\mathtt{Lie}$ and $A\,=\,\mathfrak{a}$, this is exactly the Lie cyclic homology $\HC_{\bullet}(\mathtt{Lie}, \mathfrak{a})$ used in Section~\ref{sec7} (which~\cite{GK} denote by $\mathrm{HA}_{\bullet}(\mathtt{Lie}, A)$).

\subsection{Traces}

Suppose that the DG $\mathcal P$-coalgebra $\mathfrak{C}$ is equipped with a degree $0$ invariant $2$-tensor 
$$ 
\cTr\,:\,k \rar \mathfrak{C}^{\n} \,\text{.}
$$ 
Then, for any $A \in \mathtt{DGPA}$,
there is a natural morphism of DG $\mathcal P$-algebras
$$\pi_A\,:\, A \rar \mathbf{Hom}_{\ast}(\mathfrak{C}, \mathfrak{C} \ltimes A) \hookrightarrow \mathbf{Hom}(\mathfrak{C}, \mathfrak{C} \ltimes A)\,, $$
where the first arrow corresponds under the adjunction~\eqref{natadjoperads} to the identity morphism on $\mathfrak{C} \ltimes A$. Denote by $\pi_{A \otimes A}$ the composite map
$$ \begin{diagram} A \otimes A & \rTo^{\pi_A \otimes \pi_A} &  \mathbf{Hom}(\mathfrak{C}, \mathfrak{C} \ltimes A) \otimes \mathbf{Hom}(\mathfrak{C}, \mathfrak{C} \ltimes A)
& \rTo & \mathbf{Hom}(\mathfrak{C}^{\otimes 2}, (\mathfrak{C} \ltimes A)^{\otimes 2}) & \rTo & \mathbf{Hom}(\mathfrak{C} \otimes \mathfrak{C}, \mathfrak{C} \ltimes A) \end{diagram}\,,$$
where the last arrow is induced by the product on $\mathfrak{C} \ltimes A$. The map of complexes
$$ \begin{diagram} A \otimes A & \rTo^{\pi_{A \otimes A}} & \mathbf{Hom}(\mathfrak{C} \otimes \mathfrak{C}, \mathfrak{C} \ltimes A) & \rTo & \mathbf{Hom}(\mathfrak{C}^{\n}, \mathfrak{C} \ltimes A) \end{diagram}$$
induced by the inclusion $\mathfrak{C}^{\n} \hookrightarrow \mathfrak{C} \otimes \mathfrak{C}$ clearly factors through $A_{\n}$ giving
$$ A_{\n} \rar \mathbf{Hom}(\mathfrak{C}^{\n}, \mathfrak{C} \ltimes A)\,\text{.}$$
Composing this last map with the map $\mathbf{Hom}(\mathfrak{C}^{\n}, \mathfrak{C} \ltimes A) \rar \mathbf{Hom}(k,\mathfrak{C} \ltimes A) \,=\, \mathfrak{C} \ltimes A$ induced by $\cTr$, we get
\begin{equation} \la{eoptrace1}  \Tr_{\mathfrak{C}}\,:\, A_{\n} \rar \mathfrak{C} \ltimes A \,\text{.} \end{equation}
We remark that, in general, $\Tr_{\mathfrak{C}}$ depends on the choice of 
invariant $2$-tensor $\cTr$ on $\mathfrak{C}$; however, instead of making that choice, 
we could work with the universal invariant $2$-tensor $\,
\mathfrak{C}^{\n} \into \mathfrak{C} \otimes \mathfrak{C}\,$.

By construction, $\Tr_{\mathfrak{C}}$ gives  a natural transformation of functors
\begin{equation} \la{eoptrace2}  (\,\mbox{--}\,)_{\n} \rar \mathfrak{C} \ltimes \,\mbox{--}\,\,:\, \mathtt{DGPA} \rar \Com_k \,\text{.}\end{equation}
It therefore, gives a natural transformation of {\it derived} functors
\begin{equation} \la{eoptrace3}  \Tr_{\mathfrak{C}}\,:\, \L(\,\mbox{--}\,)_{\n} \rar \mathfrak{C} \stackrel{\L}{\ltimes} \,\mbox{--}\,\,:\, \Ho(\mathtt{DGPA}) \rar \Ho(\Com_k) \,\text{.}\end{equation}
Applying this to $A \in \mathtt{DGPA}$ and taking homologies, one obtains maps of graded vector spaces
\begin{equation} \la{eoptrace4} \Tr_{\mathfrak{C}}\,:\, \HC_{\bullet}(\mathcal P, A) \rar \H_{\bullet}(A, \mathfrak{C}) \,\text{.} \end{equation}

We now give a few examples of the trace~\eqref{eoptrace4}.

\subsubsection{Example} Let $\mathcal P$ be a cyclic binary quadratic operad. Suppose that $\mathfrak{C}$ is a finite-dimensional $\mathcal P$ coalgebra in homological degree $0$. Let $\mathcal S\,:=\, \mathfrak{C}^{\ast}$ and let $\mathrm{tr}\,:\,\mathscr S_{\n}\rar k$ denote the dual of the map $\cTr$. Let $G:= \mathtt{Aut}(\mathcal S, \mathrm{tr})$ denote the algebraic group of trace preserving automorphisms of the $\mathcal P$-algebra $\mathscr S$ (see~\cite[Section 6]{Gi}). Suppose $A$ is a $\mathcal P$-algebra concentrated in homological degree $0$. In this case, $ \mathfrak{C} \ltimes A\,=\, k[\Rep_{\mathcal P}(A, \mathcal S)]$. The trace~\eqref{eoptrace4} induces a trace map
$$ \Tr_{\mathcal S} \,:\, A_{\n} \rar k[\Rep_{\mathcal P}(A, \mathcal S)] \,$$
on $0$-th homologies. It is easy to verify that the image of $\Tr_{\mathcal S}$ actually lies in $k[\Rep_{\mathcal P}(A, \mathscr S)]^{G}$. The map
$$\Tr_{\mathcal S}\,:\, A_{\n} \rar k[\Rep_{\mathcal P}(A, \mathcal S)]^G$$ appears in~\cite[Section 6]{Gi}.

\subsubsection{Example} Let $\mathcal P = \mathtt{Ass}$ and let $\mathfrak{C}= \M_n^{\ast}(k)$. Let $A$ be a {\it unital} DG $k$-algebra. Then (see~\cite[Section 4]{GK}) $A_{\n} \,\cong\, A/[A,A]$. Dually, since $\M_n^{\ast}(k)$ is counital, $[\M_n^{\ast}(k)]^{\n}$ can be identified with the cocommutator subspace of $\M_n^{\ast}(k)$. This subspace is a $1$-dimensional subspace of $\M_n^{\ast}(k)$ generated by the dual of the usual trace map. Choosing the dual of the usual trace map as our cotrace and applying the construction in this section, one sees that the trace map~\eqref{eoptrace4} becomes the trace map
$$\Tr_{n}(A)_{\bullet}\,:\,\HC_{\bullet}(A) \rar \H_{\bullet}(A, n) $$
constructed in~\cite[Section 4]{BKR}.

\subsubsection{Example} Let $\mathcal P\,=\,\mathtt{Lie}$ and let $\mathfrak{C}\,=\,\g^{\ast}$ where $\g$ is a semisimple Lie algebra. Then, $\mathfrak{C}^{\n}$ is a one dimensional $k$-vector space generated by the Killing form on $\g$. Taking the Killing form as our invariant $2$-tensor, the trace~\eqref{eoptrace4} specializes to the canonical trace~\eqref{canlietrace}.


\end{document}